\documentclass[11pt,a4paper]{amsart}

\usepackage[margin=3cm]{geometry}
\usepackage{mathtools,amsmath,amsthm,amssymb,amscd,amsfonts,color}
\usepackage{booktabs}
\usepackage{enumitem}
\usepackage{tikz}
\usepackage{thmtools,thm-restate}

\usepackage{multirow}
\usepackage{easyReview}
\usepackage{color}
\usepackage[outline]{contour}
\usepackage{booktabs}
\usepackage{enumitem}
\usepackage{tikz}
\usepackage[colorlinks=true]{hyperref}
\usepackage{easyReview}
\usepackage{pgfplots}
\pgfplotsset{compat=1.15}
\usetikzlibrary{arrows}
\usepackage{subcaption}

\renewcommand*{\geq}{\geqslant}
\renewcommand*{\leq}{\leqslant}

\theoremstyle{plain}
\newtheorem{theorem}{Theorem}[section]
\newtheorem{corollary}[theorem]{Corollary}
\newtheorem{lemma}[theorem]{Lemma}
\newtheorem{claim}[theorem]{Claim}
{\itshape}{\rmfamily}
{\itshape}{\rmfamily}
\newtheorem{claimnumb}{Claim}[theorem]{\itshape}{\rmfamily}

\theoremstyle{definition}

\newtheorem{remark}[theorem]{Remark}

\newtheorem{conjecture}[theorem]{Conjecture}%

\newcommand{\cal}[1]{\mathcal{#1}}

\newcommand{\R}{\mathbb{R}}
\newcommand{\Q}{\mathbb{Q}}
\newcommand{\Z}{\mathbb{Z}}
\newcommand{\F}{\mathbb{F}}

\DeclareMathOperator{\conv}{conv}

\DeclareMathOperator{\supp}{supp}

\DeclareMathOperator{\cone}{cone}
\DeclareMathOperator{\eqdet}{eqdet}
\DeclareMathOperator{\gcddet}{gcddet}
\renewcommand{\cal}[1]{\mathcal{#1}}
\newcommand{\trim}{/\!\!/}
\DeclareMathOperator{\diag}{diag}
\DeclareMathOperator{\nega}{neg}

\title{Totally equimodular matrices: decomposition and triangulation}
\author{Patrick Chervet}
\address{Lyc\'ee Olympe de Gouges, Rue de Montreuil \`a Claye, {Noisy le Sec}, {93130 France}}
\email{patrick.chervet@ac-creteil.com}

\author{Roland Grappe}
\address{{LAMSADE (CNRS UMR7243)}, {Universit\'e Paris-Dauphine, PSL},
	Place du Maréchal de Lattre de Tassigny, 75775 PARIS Cedex 16 
	{France}}
\email{roland.grappe@lamsade.dauphine.fr}

\author{Mathieu Vall\'ee}
\address{Université Sorbonne Paris Nord, LIPN, CNRS UMR 7030, F-93430, Villetaneuse, France}
\email{vallee@lipn.univ-paris13.fr}

\date{\today}

\begin{document}
\begin{abstract}
	Totally equimodular matrices generalize totally unimodular matrices and arise in the context of box-total dual integral polyhedra. 
	This work further explores the parallels between these two classes and introduces foundational building blocks for constructing totally equimodular matrices. 
	Consequently, we present a decomposition theorem for totally equimodular matrices of full row rank.
	
	Building on this decomposition theorem, we prove that simplicial cones whose generators form the rows of a totally equimodular matrix sa\-tisfy strong integrality decomposition properties.		
	More precisely, we provide the Hilbert basis for these cones and construct regular unimodular Hilbert triangulations in most cases. 
	We conjecture that cases not covered here do not exist.
\end{abstract}
	{\let\newpage\relax\maketitle}
	
\section{Introduction}

\subsection{Box-total dual integrality and matrices}

A rational system of linear inequalities $Ax\leq b$ is \emph{totally dual integral} (\emph{TDI}) if the minimization problem in the linear programming duality:
\begin{equation*}
	\max\{c^\top x\colon Ax\leq b\}=\min\{b^\top y\colon A^\top y = c, y\geq\mathbf{0}\}
\end{equation*}
admits an integer optimal solution for each integer vector $c$ such that the maximum is finite.

A system $Ax \leq b$ is \emph{box-totally dual integral}~\cite{Edmonds_Giles_1977} (\emph{box-TDI}) if $Ax\leq b$, $\ell\leq x\leq u$ is TDI for all rational vectors $\ell$ and $u$ (with possible infinite entries).
TDI and box-TDI systems were introduced in the late 1970's and serve as a general framework for establishing various min-max relations in combinatorial optimization~\cite{Schrijver_2003}.
General properties of such systems can be found in~\cite{Cook_1986},~\cite[Chap.~5.20]{Schrijver_2003} and~\cite[Chap.~22.4]{Schrijver_1999}.
Two famous examples of box-TDI systems are the systems behind the MaxFlow-MinCut theorem of Ford and Fulkerson~\cite{Ford_Fulkerson_1956} and K\H onig's theorem about matchings in bipartite graphs~\cite{Konig_1931}.

More precisely, the box-TDIness of the two latter systems comes from the total unimodularity of the underlying matrix.
A matrix is {\em totally unimodular} when all its subdeterminants are $0,\!\pm1$.
Totally unimodular matrices can be characterized in terms of box-TDI systems as follows.

\begin{theorem}[Hoffman and Kruskal~\cite{Hoffman_Kruskal_1956}]\label{thm:Hoffman_Kruskal_1956}
	A matrix $A$ of $\mathbb{Z}^{m\times n}$ is totally unimodular if and only if the system $Ax\leq b$ is box-TDI for all $b\in\mathbb{Z}^m$.
\end{theorem}
Then, the total unimodularity of incidence matrices of directed graphs and that of bipartite graphs yields the two aforementioned theorems.
\medskip

Although every polyhedron can be described by a TDI system~\cite[{Theorem~22.6}]{Schrijver_1999}, not every polyhedron can be described by a box-TDI system.
{\em Box-TDI polyhedra}~\cite{Cook_1986} are those that can be described by a box-TDI system.

In the context of box-TDI polyhedra, a generalization of totally unimodular matrices arises naturally: totally equimodular matrices~\cite{Chervet_Grappe_Robert_2021}.
An equivalent definition of total unimodularity is to ask for every set of linearly independent rows to be unimodular, where an $m\times n$~matrix is {\em unimodular} if it has full row rank and all its nonzero $m\times m$ determinants are $\pm 1$.
More generally, an $m\times n$~matrix is {\em equimodular} if it has full row rank and all its nonzero $m\times m$ determinants have the same absolute value.
Then, {\em totally equimodular} matrices are those for which every set of linearly independent rows forms an equimodular matrix.

It turns out that totally equimodular matrices fulfill the same role for box-TDI polyhedra as totally unimodular matrices do for box-TDI systems.

\begin{theorem}[{Chervet, Grappe, and Robert~\cite[Corollary~8]{Chervet_Grappe_Robert_2021}}]\label{TEMPBI}
	A matrix $A$ of $\mathbb{Q}^{m\times n}$ is totally equimodular if and only if the polyhedron $\{x:Ax\leq b\}$ is box-TDI for all $b\in\mathbb{Q}^m$.
\end{theorem}

The question of their recognition is raised in~\cite[Open problem~1]{Chervet_Grappe_Robert_2021}, and~\cite[Open problem~3.24]{Grappe_2021} asks whether there exists a decomposition theorem for them. 

\paragraph{\textbf{Contributions.}} We start by further investigations on the connections between totally unimodular and totally equimodular matrices, and provide two new parallels: one in terms of pivots and trims, and another involving the transpose and the inverse.
Then, we give our first main result: a decomposition theorem of totally equimodular matrices of full row rank.
As a consequence, this provides a unified framework in which well-known classes of matrices appear intertwined, such as minimally non-totally unimodular, complement totally unimodular, and complement minimally non-totally unimodular matrices.
Incidentally, linear systems associated with totally equimodular matrices are totally dual dyadic.

\subsection{Hilbert triangulation of cones}
Integer decomposition properties of cones and polytopes appear in various fields such as combinatorial optimization, toric geometry, and combinatorial commutative algebra.

By Carath\'eodory's theorem~\cite{Caratheodory1911}, each point of a pointed (rational polyhedral) cone $C$ is the nonnegative combination of at most $\dim (C)$ generators of $C$.
In combinatorial optimization, a preferable property is the so-called integer Carath\'eodory property.
A \emph{Hilbert basis} of $C$ is a minimal set of integer vectors of $C$ whose nonnegative integer combinations generate all the integer points of $C$.
A cone $C$ has the \emph{integer Carath\'eodory property} (\emph{ICP}) when each of its integer points can be expressed as a nonnegative integer combination of at most $\dim(C)$ Hilbert basis elements.
This property coincides with the integer Carath\'eodory property for polytopes~\cite{Gijswijt_Regts_2012} in the following sense: a polytope $P$ has the integer Carath\'eodory property if and only if the Hilbert basis of the cone $C=\cone(P\times\{1\})$ consists of the lattice points in $P\times\{1\}$ and $C$ has the integer Carath\'eodory property.

Hilbert bases play a fundamental role in combinatorial optimization, as they underlie the notion of TDI systems~\cite[{Theorem~22.5}]{Schrijver_1999}.
The integer Carath\'eodory property then ensures that an optimal solution of the dual problem is sparse~\cite{Cook_Fonlupt_Schrijver_1986}.
Unfortunately, finding a Hilbert basis of a cone is hard in general~\cite{Pap2011}.
\medskip

Over the years, a hierarchy of increasingly stronger integer decomposition properties has emerged: having a \emph{unimodular Hilbert cover} (\emph{UHC}), a \emph{unimodular Hilbert triangulation} (\emph{UHT}), and a \emph{regular unimodular Hilbert triangulation} (\emph{RUHT}).
We have $\text{RUHT } \Longrightarrow\text{ UHT }\Longrightarrow\text{ UHC }\Longrightarrow \text{ ICP}$.
Backgrounds about these properties can be found in~\cite{Firla_Ziegler_1999,Bruns_Gubeladze_Henk_Martin_Weismantel_1999,Bruns_Gubeladze_2002,Bruns_2007,Gijswijt_Regts_2012}, as well as the fact they strictly imply one another.
Like the integer Carath\'eodory property, all these properties have polyhedral counterparts~\cite[Section~1.2.5]{Haase_Paffenholz_Piechnik_Santos_2021}.

Notably, Seb\H o~\cite{Sebo_1990} proved that cones of dimension at most three have a unimodular Hilbert triangulation.
Matroid base polytopes satisfy the integer Carath\'eodory property~\cite{Gijswijt_Regts_2012}, and it was recently shown that they admit a regular unimodular Hilbert triangulation~\cite{Backman_Liu_2024}, although this triangulation is nonexplicit.
There exists a nonsimplicial cone for which no unimodular Hilbert triangulation is regular~\cite{Firla_Ziegler_1999}.
If a simplicial cone~\cite{Kuhlmann_2024} or its dual~\cite{Aliev_Henk_Hogan_Kuhlmann_Oertel_2024} has small determinants, then it has the integer Carath\'eodory property.
Incidentally, Kuhlmann asks~\cite{Kuhlmann_2024}: \textit{Does every simplicial cone have the integer Carath\'eodory property?}
A related notion studied in~\cite{Aliev_Henk_Hogan_Kuhlmann_Oertel_2024} and~\cite{Eisenbrand2006} is the \emph{integer Carath\'eodory rank}, which is the maximum number of Hilbert basis elements needed to generate the integer points of the cone.
\paragraph{\textbf{Contributions.}} 
We call {\em te-cones} the simplicial cones whose generators form the rows of a totally equimodular matrix.
As a consequence of the decomposition theorem mentioned in the previous section, we first derive the Hilbert basis of te-cones. Building on this, we construct an explicit regular unimodular Hilbert triangulation for most te-cones, with a well-described combinatorial structure. Supported by computational experiments, we conjecture that the untreated cases do not exist. In contrast to previous results on simplicial cones, our findings include simplicial cones of any dimension, with determinants that can grow exponentially with the dimension. For instance, the following new result is a special case of our main triangulation theorem:
\begin{corollary}\label{corollary:0/1_simplicial_box-integer_cones}
	Simplicial box-totally dual integral cones in the nonnegative orthant have the integer Carath\'eodory property.
\end{corollary}

\subsection{Outline.}
The remainder of the paper is organized in two parts.
The first part (Section~\ref{section:def} to~\ref{section:main_cones}) contains the main results with only sketches of proofs, while the second block (Sections~\ref{section:proofs_decomposition} and~\ref{section:proofs_cones}) contains the detailed proofs.

Section~\ref{section:def} introduces the definitions and notations used throughout. 
In Section~\ref{section:decomposition}, we first establish two new parallels between totally unimodular and totally equimodular matrices.
We then present our decomposition theorem for totally equimodular matrices of full row rank.
We also discuss connections with known matrix classes and propose a conjecture to refine the decomposition. 
Section~\ref{section:main_cones} is divided in two.
First, we provide the Hilbert basis of all te-cones.
Then, we construct a regular unimodular Hilbert triangulation for most te-cones.

The second part starts with Section~\ref{section:proofs_decomposition}, which is devoted to the proof of the decomposition theorem.
In Section~\ref{section:proofs_cones}, we prove that the aforementioned triangulation is indeed regular, unimodular, and Hilbert.

\section{Definitions and notation}\label{section:def}

\subsection{Sets of vectors and matrices.}

Throughout, all the entries will be rational.
Moreover, we will identify a set of vectors $A$ with the matrix, also denoted by $A$, whose rows are those vectors.
All the considered sets will be linearly independent, and equivalently the matrices will have full row rank.
A linearly independent set $A\subseteq \R^n$ is \emph{full-dimensional} when it spans $\R^n$, or equivalently when its associated matrix is square and invertible.	
A subset $X$ of $A$ is \emph{proper} when $\emptyset\subsetneq  X\subsetneq  A$.
The \emph{disjoint union} of two disjoint sets $A$ and $B$ is denoted by $A\sqcup B$.

We will use the following notations.
Let $A$ be an $m\times n$~matrix.
The $i$-th row of $A$ is denoted by $A_i$, its $j$-th column by $A^j$, and its coefficient $(i,j)$ by $A_i^j$.
The vectors $\mathbf{0}$ and $\mathbf{1}$ are the vectors with all entries being $0$ and $1$ of appropriate size, respectively.
The matrix $\mathbf{J}$ is the all-ones matrix.
The \emph{support} $\supp (x)$ of a vector $x$ is the set of its nonzero coordinates.

A {\em pivot}\footnote{Note that this is the definition of the classical Gauss-pivot, which differs from the pivot used by Seymour~\cite{Seymour1980} in the decomposition theorem of totally unimodular matrices.} will be a position $p=(i,j)$ at row index $i$ and column index $j$ of $A$ whose coefficient $A_i^j$ is not zero.
Then, {\em $p$-pivoting $A$}, or {\em pivoting $A$ with respect to $p$}, means dividing $A_i$ by $A_i^j$, and then adding an appropriate scalar multiplication of this new row to the other rows of $A$, so that afterwards all the coefficients in $A^j$ are zero, except that at row $i$ which is now $1$.
The resulting matrix is called the {\em $p$-pivot} of $A$ and will be denoted by $A/p$.

The matrix {\em trimmed} from $A$ with respect to $p=(i,j)$ is the matrix obtained from $A/p$ by deleting row $i$ and column $j$ and will be denoted by $A\trim p$, the {\em $p$-trim} of $A$.

The \emph{equideterminant} of an $m\times n$ equimodular matrix $A$ is the absolute value of any nonzero $m\times m$ determinant of $A$, and is denoted $\eqdet(A)$.
Then, \emph{rescaling} a totally equimodular matrix $A$ means dividing each row $A_i$ by $\eqdet(A_i)$, and yields a $0,\!\pm1$ matrix.
Multiplying some rows or columns by $-1$ is called \emph{resigning}.
Notice that resigning preserves total unimodularity and total equimodularity.


\subsection{Cones, Hilbert basis, and triangulations.}\label{section:cones}

A (finitely generated) \emph{cone} $C$ is the set of all nonnegative linear combinations of a finite collection of vectors $A$:
$$C=\cone(A) = \left\{\sum_{a\in A} \lambda_a a\colon \lambda_a\geq 0\right\}.$$
Let $C\subseteq\R^n$ be a cone and $A$ its set of \emph{generators}, that is, $A$ is inclusionwise minimal such that $C=\cone(A)$.
Throughout, we suppose all cones to be \emph{simplicial}, that is, their generators are linearly independent.
Thus, they are \emph{pointed}, that is, they contain no lines.
The \emph{dimension} $ \dim(C) $ of $C$ is the dimension of the linear subspace it spans. The cone $ C $ is \emph{full-dimensional} if $ \dim(C) = n $.
A \emph{face} of a cone $ C $ is the intersection of $ C $ with a hyperplane such that one of the half-spaces defined by the hyperplane contains $C$.
As a result, any face of $C$ is a cone generated by a subset of the generators of $C$.
The \emph{Minkowski sum} of two cones $\cone(A)$ and $\cone(B)$ is $\cone(A)+\cone(B)=\cone(A\cup B)$.



An integer vector $x\in\Z^n$ is called \emph{primitive} if the greatest common divisor of all its coefficients is $1$.
We will always assume the generators to be primitive.
An integer vector $h$ of $C$ is a \emph{Hilbert basis element} if we cannot express it as the sum of two nonzero integer vectors of $C$.
A \emph{Hilbert basis} $\cal{H}\subseteq C\cap\Z^n$ is a finite set of Hilbert basis elements generating all the integer points of the cone with nonnegative integer coefficients.
Here, since the cones are pointed, the Hilbert basis $\cal{H}(C)$ of a cone $C$ is unique~\cite{Schrijver1981}.
The Hilbert basis of $C$ contains its generators, the other Hilbert basis elements are called \emph{nontrivial}.
A simplicial cone $C=\cone (A)$ is \emph{unimodular} if the gcd of the maximum size determinants of $A$ is $1$.
In this case, we have $\cal H(C)=A$.

A \emph{triangulation} of a cone $C$ is a collection of cones whose union is $C$, with the property that the intersection of any two cones in the collection is a face of each.
A triangulation of $C$ is called \emph{Hilbert} if the generators of each cone in the triangulation are Hilbert basis elements of $C$.
It is \emph{unimodular} if all its cones are unimodular, and \emph{regular} if each cone corresponds to the domain of regularity of a convex piecewise linear function on $C$. We refer the reader to~\cite[Chap.~5]{Ziegler_1995} and~\cite[Chap.~8]{Sturmfels_1996} for more details.

\section{Decomposition of totally equimodular matrices}\label{section:decomposition}

\subsection{Parallels between total equimodularity and total unimodularity}\label{subsection:parallel_TE_TU}
In the introduction, we mentioned a parallel between totally equimodular and totally unimodular matrices in terms of box-TDIness, which is how totally equimodular matrices appeared.
Another connexion exists in terms of undirected graphs: the edge-vertex incidence matrix of a graph is always totally equimodular~\cite{Chervet_Grappe_Lacroix_Pisanu_Wolfler-Calvo_2023}, and it is totally unimodular if and only if the graph is bipartite~\cite{hoffmanKruskal:1956}.

We report two new parallels between these two classes: the first one concerns pivots and trims, and the second one involves the transpose and the inverse.

\medskip

Since values that appear when pivoting and trimming are values of subdeterminants of the original matrix, totally unimodular matrices can be characterized as follows.

\begin{theorem}[Folklore]\label{pivotTU}
	A matrix is totally unimodular if and only if any sequence of pivots and trims yields a $0,\!\pm1$ matrix.
\end{theorem}

It turns out that this extends to totally equimodular matrices as follows.
A matrix is {\em essentially $0,\!\pm1$} if, in each of its rows, the nonzero coefficients all have the same absolute value.
In other words, multiplying each row by an appropriate coefficient yields a $0,\!\pm1$ matrix.
Note that totally equimodular matrices are essentially $0,\!\pm1$.

%

\begin{theorem}[Theorem~\ref{pivotpreserveTE}]\footnote{Here and in the next section, such a reference points to the statement and its complete proof.}\label{theorem:pivot_TE}
	A matrix is totally equimodular if and only if any sequence of pivots and trims yields an essentially $0,\!\pm1$ matrix.
\end{theorem}
\begin{proof}[Sketch]
	It is immediate that pivoting and trimming preserve equimodularity.
	This can be turned into an equivalence as follows, by considering directly the involved determinants.
	\begin{lemma}[Lemma~\ref{pivoteqE}]\label{lemma:trim_E}
		Let $A_i$ be a $0,\!\pm1$ row of a full row rank matrix $A$.
		Then, $A$ is equimodular if and only if $A\trim(i,j)$ is equimodular for all $j\in\supp(A_i)$.
	\end{lemma}
	Then, the theorem follows.
	\end{proof}
We mention that Lemma~\ref{lemma:trim_E} will be essential throughout the proof of the decomposition theorem of the next section.

Note that Theorems~\ref{pivotTU} and~\ref{theorem:pivot_TE} can be rephrased as follows: A matrix is totally unimodular (resp. totally equimodular) if and only if any sequence of pivots and trims yields a totally unimodular (resp. totally equimodular) matrix. 

\medskip

It is known that taking the transpose or the inverse of a matrix preserves total unimodularity~\cite[Page~280]{Schrijver_1999}.
The transpose operation does not preserve total equimodularity in general, as it is highlighted for incidence matrices of graphs in~\cite{Chervet_Grappe_Lacroix_Pisanu_Wolfler-Calvo_2023}.
Neither does taking the inverse.
However, the combination of both operations preserves total equimodularity.

\begin{theorem}\label{theorem:invert_TE}
	Let $A$ be an invertible square matrix, then $A$ is totally equimodular if and only if $(A^{-1})^\top$ is totally equimodular.
\end{theorem}
\begin{proof}
	We will use the following~\cite[Theorem~2]{Chervet_Grappe_Robert_2021}: a cone is box-TDI if and only if the affine hull of each of its faces is described by an equimodular matrix.
	A full row rank matrix $B$ is {\em face-defining} for a cone $C$ when there exists a face $F$ of $C$ such that $\textrm{aff}(F)=\{x:Bx = 0\}$.
	Note that a face has an equimodular face-defining matrix if and only if all its face-defining matrices are equimodular.
	
	Let us prove that an invertible matrix $A$ is totally equimodular if and only if $(A^{-1})^\top$ is totally equimodular.
	
	Since $A$ is invertible, every subset of rows of $A$ forms a face-defining matrix of $C = \{x: Ax\geq 0\}$.
	Therefore, by \cite[{Theorem~2}]{Chervet_Grappe_Robert_2021}, $A$ is totally equimodular if and only if $C$ is box-TDI. 
	By \cite[Lemma~6]{Chervet_Grappe_Robert_2021}, the latter holds if and only if the polar $C^\star=\{x:x^\top z\leq 0, \text{ for all $z \in C$}\}$ of $C$ is box-TDI.
	Since $C$ is full-dimensional and simplicial, so is its polar, which is decribed by $C^\star=\{x: (A^{-1})^\top x \geq 0\}$.
	Then, this polar is box-TDI if and only if $(A^{-1})^\top$ is totally equimodular, which concludes.
	\end{proof}

\subsection{The decomposition theorem}\label{subsection:decomposition}

Recall that we identify a set of vectors as the matrix whose rows consist of those vectors.
A linearly independent set of $0,\!\pm1$ vectors is called a:
\begin{itemize}
	\item \emph{totally equimodular set} (\emph{te-set}) if its associated matrix is totally equimodular,
	\item \emph{totally unimodular set} (\emph{tu-set}) if its associated matrix is totally unimodular,
	\item \emph{te-lace} if it is a te-set, not a tu-set, and all its proper subsets are tu-sets,
	\item \emph{te-interlace} if it is a te-set, not a tu-set, and each pair of vectors is a te-lace.
\end{itemize}
Note that, since the vectors in a te-interlace \( A \) pairwise form te-laces, they share the same support. 
Therefore, up to permuting columns, we can write $A = \begin{bmatrix} A' & \mathbf{0} & \cdots & \mathbf{0} \end{bmatrix},$ where $A'$ is a $\pm 1$ matrix. 
When appropriate, the $\mathbf{0}$ columns can be omitted.
A te-interlace of size $\ell$ is {\em thin} if its equideterminant is $2^{\ell - 1}$ and {\em thick} if it is $2^\ell$.
A \emph{te-brick} is either a tu-set, a te-lace, a thin te-interlace, or a thick te-interlace.
Disjoint te-bricks are {\em mutually totally unimodular} (\emph{mutually-tu}) when, if a set intersects several of them, contains none of the te-laces, and at most one vector of each te-interlace, then it is a tu-set.

\medskip

These te-bricks are the basic structures onto which te-sets are built upon, as shown below.
The property of being mutually-tu implies that the decomposition is unique.

\begin{theorem}[Theorem~\ref{theorem:decomposition_te_full}]\label{theorem:decomposition_te}
	A linearly independent set of $0,\!\pm1$ vectors $A$ is a te-set if and only if it is the disjoint union of mutually-tu te-bricks. More precisely:
	$$A = \underbrace{U_{}}_{\text{tu-set}} \sqcup~\underbrace{L_1\sqcup \dots \sqcup L_k}_{\text{te-laces}}~\sqcup \underbrace{S_1\sqcup \dots\sqcup S_\ell}_{\text{thin te-interlaces}}\sqcup \underbrace{T_1\sqcup \dots\sqcup T_m}_{\text{thick te-interlaces}}\!\!.$$
\end{theorem}
\begin{proof}[Sketch]
	There are two directions to be proven.
	First, that a te-set $A$ is the disjoint union of mutually-tu te-bricks.
	The starting point is that a te-set which is not a tu-set contains a te-lace.
	Let $\cal L$ be the family of te-laces of $A$.
	Then, $U=A\setminus (\bigcup_{L\in\cal L} L)$ is a tu-set.
	Moreover, by the following key lemma, pairwise intersecting members of $\cal L$ form a te-interlace.
	\begin{lemma}[Lemma~\ref{te-knotinter}]\label{keylemma}
		In a te-set, if two distinct te-laces intersect, then they are both of size two and their symmetric difference is a te-lace.
	\end{lemma}
	\begin{proof}[Sketch]
		The proof starts with a minimal counterexample, and studies the impact of various trims.
		First, we prove that the intersection of the two sets is a singleton. 
		Following this, we show that one of them has size two and the other has size three.
		Afterwards, we study the different possibilities concerning the supports of the four involved vector, as two independent vectors form a te-lace if and only if they have the same support, and a tu-set if and only if they either coincide or are opposite on their common support.
		This ends up contradicting their linear independence when one of them has size three.
		\end{proof}
	Regroup the intersecting members of $\cal L$ into a family $\cal I$ of maximal te-interlaces: the remaining te-laces $\cal L'$ of $\cal L$ are pairwise disjoint and intersect no te-interlace of $\cal I$. 
	The last ingredient to finish the decomposition is to prove that these te-interlaces are only of two types: the thin ones $\cal S$ and the thick ones $\cal T$.
	This is done by systematically studying the trims of te-interlaces.
	
	By construction, no te-lace intersects distinct sets among, $U$, $L\in\cal L'$, $S\in \cal S$, and $T\in \cal T$, thus these sets are mutually-tu, and $A$ is the disjoint union of mutually-tu te-bricks.
	
	\medskip
	
	Now, there remains to prove that disjoint unions of mutually-tu te-bricks form te-sets.
	The proof is by induction: let $A$ be such a set, and assume that all smaller sets which are the mutually-tu disjoint union of te-bricks are te-sets.
	All that remains to prove is that $A$ is equimodular.
	
	We start by proving that every $A'\subsetneq A$ also is a mutually-tu disjoint union of te-bricks.
	Then, we study the impact of trimming on the different te-bricks involved.
	Let $B$ be a te-brick of $A$ and $a=A_i$ a row of $A$. 
	We first prove the following thanks to the mutually-tu property: if $a\notin B$, then the set obtained from $B$ after $(i,j)$-trimming $A$, for any column $j$ of $A$ with $A_i^j\neq0$, is of the same type as $B$.
	If $a\in B$, there are several cases: if $B$ is a tu-set or a te-lace of size two, then, after rescaling, $B\trim (i,j)$ is a tu-set; if $B$ is a te-lace of size at least three, then $B\trim (i,j)$ is a te-lace.
	
	Now, if $A$ contains a row which is in no te-interlace, this row is used to trim $A$, and to retrieve its equimodularity thanks to the above facts and Lemma~\ref{lemma:trim_E}.
	Otherwise, each row of $A$ is in a te-interlace, and we need an additional property: trimming and then rescaling a te-interlace yields a te-set in which there are no te-interlaces.
	By the first direction of the theorem, it becomes the disjoint union of a tu-set and te-laces which are mutually-tu.
	We then prove that the set resulting from $A$ by rescaling such a trim is the mutually-tu disjoint union of te-bricks, and hence is a te-set by the induction hypothesis.
	In particular, fixing a row and performing all possible trims yields an equimodular matrix, and hence $A$ is equimodular, again by Lemma~\ref{lemma:trim_E}.
	\end{proof}

Theorem~\ref{theorem:decomposition_te} raises a complexity question: {\em Can the decomposition be obtained in polynomial time?}
Finding a candidate for the decomposition into disjoint te-bricks can be done in polynomial time, because testing total unimodularity can be done in polynomial time~\cite{Seymour1980}.
However, deciding if these te-bricks are mutually-tu seems challenging.

Once the decomposition is known, so is the equideterminant:
if $A$ is a te-set that decomposes as in Theorem~\ref{theorem:decomposition_te}, we have $\eqdet(A) = 2^{k+\sum_i (|S_i|-1) +\sum_j |T_j|}$.
Together with~\cite[{Theorem~1.7}]{Abdi2024}, this yields the following\footnote{Note that the full row rank assumption is dropped here.}.

\begin{corollary}
	For a totally equimodular matrix $A\in\{0,\!\pm1\}^{m\times n}$ and $b\in\mathbb{Z}^m$, the system $Ax\leq b$ is totally dual
	dyadic.
\end{corollary}

We mention that the full row rank hypothesis is essential to derive the above decomposition theorem, as the key lemma fails without it.
The situation might be dramatically intricated as shows the example in Figure~\ref{ctrex}.
This raises the question:
\textit{Is there a decomposition theorem for general totally equimodular matrices?}
\begin{figure}[ht]
	\centering
	${ M=
		\begin{tabular}{c}
			1\\
			2\\
			3\\
			4\\
			5\\
			6				
		\end{tabular}
		\begin{bmatrix}
			1 & 1 & 0 & 0\\
			1 & 0 & 1 & 0\\
			1 & 0 & 0 & 1\\
			1 & 1 & 1 & 1\\
			0 & 1 & 0 & 1\\
			0 & 0 & 1 & 1
		\end{bmatrix}
	}
	$
	\caption{A totally equimodular matrix $M$ without full row rank, in which the te-laces are  $\{1,2,3,4\}$, $\{3,4,5,6\}$, $\{1,3,5\}$, and $\{2,3,6\}$, and pairwise intersect. Note that there are even intersections of size two.}\label{ctrex}
\end{figure}

\subsection{A conjecture and connections with other classes of matrices}

Supported by the fact that a brute force enumeration by computer showed that there are no thick te-interlaces of size $8$, we conjecture the following.
\begin{conjecture}\label{conjecture:size_thick_te-interlaces}
	There are only two full-dimensional thick te-interlaces, up to resigning or permuting rows and columns:
	
		$${\left[\begin{array}{rrrr}
				1&1&1&1\\
				1&-1&-1&1\\
				1&-1&1&-1\\
				1&1&-1&-1
			\end{array}\right]\text{\normalsize and}\left[\begin{array}{rrrrrr}
				1&1&1&1&1&1\\
				1&1&1&1&-1&-1\\
				1&1&1&-1&-1&1\\
				1&1&-1&-1&1&1\\
				1&-1&-1&1&1&1\\
				1&-1&1&1&1&-1
			\end{array}\right]}.$$
	\end{conjecture}
	This conjecture connects with other classes of matrices, beyond the parallel between totally equimodular and totally unimodular matrices.
	
	First, it turns out that there are further relations between each types of te-bricks.
	Let $A$ be a square invertible matrix. Then, $A$ is a thin te-interlace if and only if $\left(\frac{1}{2}A\right)^{-1}$ is a te-lace.
	In other words, thin te-interlaces are essentially the inverses of minimally non-totally unimodular matrices, where a matrix is \emph{minimally non-totally unimodular} if it is not totally unimodular, but all its proper submatrices are totally unimodular.
	Moreover, $A$ is a thick te-interlace if and only if $\left(\frac{1}{4}A\right)^{-1}$ is a thick te-interlace.
	
	More generally, te-interlaces encode classes of $0,\!1$ matrices studied by Truemper in~\cite{Truemper_1980} and~\cite{Truemper_1992}.
	Let $A$ be a $0,\!1$ matrix and $i$ a row index of $A$.
	The \emph{row-$i$ complement} of $A$ is the $0,\!1$ matrix obtained from $A$ whose $i$-th row is unchanged and whose $i'$-th rows are $A_i+A_{i'}\pmod{2}$, for $i'\neq i$.
	\emph{Column-$j$ complements} are defined similarly.
	The \emph{complement orbit} of $A$ is the set of $0,\!1$ matrices obtained from $A$ by any sequence of complement operations.
	A $0,\!1$ matrix $A$ is \emph{complement totally unimodular} if its complement orbit contains only totally unimodular matrices, and \emph{complement minimally non-totally unimodular} if its complement orbit contains only minimally non-totally unimodular matrices.
	Let $A$ be a $\pm1$ matrix that we write, up to resigning or permuting rows and columns, as
	\begin{equation}\label{equation:canonical_form}
		A=\begin{bmatrix}
			1 & \mathbf{1}^\top \\
			\mathbf{1} & \mathbf{J}-2B
		\end{bmatrix},
	\end{equation}
	for some $0,\!1$ matrix $B$.
	Then, up to resigning and rescaling, the trims of $A$ run across the complement orbit of $B$.
	An ingredient of the proof of Theorem~\ref{theorem:decomposition_te} is that, in~\eqref{equation:canonical_form}, $A$ is a thin te-interlace if and only if $B$ is complement totally unimodular, and $A$ is a thick te-interlace if and only if $B$ is complement minimally non-totally unimodular.

	\section{Hilbert triangulation of te-cones}\label{section:main_cones}

	\subsection{Hilbert basis of te-cones}\label{section:hb_te-bricks}
	
	In this section, we explicitly identify the Hilbert basis of te-cones.
	Recall that te-cones are simplicial cones generated by te-sets.
	By Theorem~\ref{theorem:decomposition_te}, a te-cone is the Minkowski sum of cones generated by mutually-tu te-bricks whose union is linearly independent.
	Thanks to the following, we can focus on each te-brick separately.
	
	\begin{lemma}[Lemma~\ref{lemma:hb_lattice_orthogonal}]\label{lemma:hb_minkowski_sum} The Hilbert basis of the Minkowski sum of cones generated by mutually totally unimodular te-bricks whose union is linearly independent is the union of the Hilbert basis of each.
	\end{lemma}
	
	Then, here is the Hilbert basis of each type of te-brick.
	\begin{theorem}[Theorem~\ref{theorem:hb_te-bricks_full}]\label{theorem:hb_te-bricks}
		Let $A = \{a^1,\ldots,a^n\}$ be a te-brick and $C = \cone (A)$.	
		\begin{enumerate}[label=\arabic*., ref=\arabic*.]
			\item If $A$ is a tu-set, then $\cal H(C)= A$.\label{equation:hb_tu-set}
			\item If $A$ is a te-lace, then $\cal H(C)= A\cup\{\frac{1}{2}\sum_j a^j\}$.\label{equation:hb_te-lace}
			\item If $A$ is a thin te-interlace, then $\cal H(C) = A\cup \left\{\frac{1}{2} (a^i + a^j)\right\}_{1\leq i<j\leq n}$.\label{equation:hb_thin_te-interlace}
			\item If $A$ is a thick te-interlace, then one of the following holds: \label{equation:hb_thick_te-interlace}
			\begin{enumerate}[label=\alph*., ref=\theenumi\alph*]
				\item $\cal H(C) = A\cup \left\{\frac{1}{2} (a^i + a^j)\right\}_{1\leq i< j\leq n}\cup\left\{\frac{1}{4}\sum_{j} a^j\right\},$ \label{equation:hb_thick_te-interlaces_cas1}
				\item $\cal H(C) = A\cup\left\{\frac{1}{2} (a^i + a^j)\right\}_{1\leq i< j\leq n}\cup\left\{\frac{3}{4}a^i+\frac{1}{4}\sum_{j\neq i} a^j\right\}_{i\in\{1,\ldots,n\}}.$ \label{equation:hb_thick_te-interlaces_cas2}
			\end{enumerate}
			Moreover, the number of $1$'s and $-1$'s in each column of $A$ have the same parity $p\in\{0,1\}$, and~\ref{equation:hb_thick_te-interlaces_cas1} occurs if and only if $n\equiv 2p \pmod{4}$.
		\end{enumerate}
	\end{theorem}
	\begin{proof}[Sketch]~
		\begin{enumerate}
			\item[{\em 1.}] Simplicial cones generated by tu-sets have no nontrivial Hilbert basis elements.
			\item[{\em 2.}] The only nontrivial Hilbert basis element of a simplicial cone generated by a te-lace is the half sum of its generators.
			\item[{\em 3-4.}] For cones generated by te-interlaces, a result of Seb\H o~\cite{Sebo_1990} tells us that the Hilbert basis of a cone lies within the half-open zonotope generated by the generators of the cone.
			We know that this number of points is the equideterminant of $A$. 
			Let $\cal S$ be the set of vectors in Case~\ref{equation:hb_thin_te-interlace}, \ref{equation:hb_thick_te-interlaces_cas1}, or~\ref{equation:hb_thick_te-interlaces_cas2}.
			In each case, we find out that the number of distinct nonnegative integer combinations of the vectors in $\cal S$ within the half-open zonotope is also equal to the equideterminant of $A$.
			This implies that $\cal{H}(C)\subseteq\cal S$.
			Finally, $\cal S=\cal H(C)$ since no vectors of $\cal S$ is a nonnegative integer combinations of the others.
		\end{enumerate}
		\end{proof}

	\subsection{Regular unimodular Hilbert triangulation of te-cones}\label{section:RUHT}
	
	In this section, we prove that cones generated by te-sets with no thick te-interlace of size greater than six admit a regular unimodular Hilbert triangulation.
	Thanks to Theorems~\ref{theorem:decomposition_te} and~\ref{theorem:hb_te-bricks}, and to the lemma below, it is enough to provide a regular unimodular Hilbert triangulation for each type of te-brick involved.
	
	The \emph{join} of two triangulations $\cal T_1$ and $\cal T_2$ of two cones generated by disjoint sets whose union is linearly independent is the triangulation $\cal T_1 * \cal T_2 = \{C_1+C_2\colon C_1\in \cal T_1,C_2\in \cal T_2\}$.
	\begin{lemma}[Lemma~\ref{lemma:joint_is_UHT_full}]\label{lemma:join_is_UHT}
		The join of the regular unimodular Hilbert triangulations of cones generated by disjoint te-bricks whose union is linearly independent is a regular unimodular Hilbert triangulation of their Minkowski sum.
	\end{lemma} 

\begin{proof}[Sketch]
	By Lemma~\ref{lemma:hb_minkowski_sum}, this construction gives a Hilbert triangulation.
	The te-bricks are disjoint and their union is linearly independent, so this join triangulation is regular by~\cite[Sect.~2.3.2]{Haase_Paffenholz_Piechnik_Santos_2021}.
	Moreover, each cone in the join is unimodular by the fact the te-bricks are mutually totally unimodular.
	\end{proof}


If Conjecture~\ref{conjecture:size_thick_te-interlaces} is true, then the following theorem together with Lemma~\ref{lemma:join_is_UHT} implies that all te-cones admit a regular unimodular Hilbert triangulation.

\begin{theorem}[Theorem~\ref{theorem:triangulation_te-sets_l}]\label{theorem:triangulation_te-sets}
	Let $A$ be a te-set without thick te-interlace of size greater than six.
	Then, $\cone (A)$ has a regular unimodular Hilbert triangulation.
\end{theorem}
\begin{proof}[Sketch]
	By Theorem~\ref{theorem:decomposition_te}, $A$ is the disjoint union of mutually-tu te-bricks, and then $C = \cone(A)$ is the Minkowski sum of the cones generated by these te-bricks.
	By Lemma~\ref{lemma:join_is_UHT}, there remains to triangulate the cones generated by each of the four te-bricks.
	We thus suppose that $A=\{a^1,\ldots,a^n\}$ is a te-brick.
	We organize the proof according to the different cases for the Hilbert basis of Theorem~\ref{theorem:hb_te-bricks}.
	\begin{enumerate}[label=\em\arabic*., ref=\arabic*.]
		\item When $A$ is a tu-set, the regular unimodular Hilbert triangulation is the cone.
		\item When $A$ is a te-lace, the stellar triangulation at $h=\frac{1}{2}\sum_j a^j$, namely the one formed by the $n$ cones generated by $h$ and $n-1$ generators among $n$, is regular since its coincides with the strong pulling at $h$ which preserved regularity~\cite[Lemma~2.1]{Haase_Paffenholz_Piechnik_Santos_2021}.	Moreover, a determinant computation yields the unimodularity. Finally all the cones are generated by Hilbert basis elements.
		\item Suppose $A$ is a thin te-interlace.\label{caseproof:second_hypersimplex_triang}
		Inspired by the regular triangulation in~\cite{DeLoera_Sturmfels_Thomas_1995}, we start this case with some definitions.
		A \emph{spanning} subgraph of a graph $G=(V,E(G))$ is a connected graph $H=(V,F)$ with $F\subseteq E(G)$.
		Let $\mathring{K}_n$ be a complete graph with $n$ vertices to which we added an edge $ii$ called a \emph{loop} at each vertex $i$.
		The edges of $\mathring{K}_n$ encode the Hilbert basis elements of $C$ as follows: an edge $ij$ represents $\frac{1}{2}(a^i+a^j)$. The latter is $a^i$ for a loop~$ii$.
		Embed $\mathring{K}_n$ as a convex $n$-gon in $\R^2$, with clockwise labeled vertices $v_1,\dots,v_n$, edges $ij$ embedded as line segments $[v_i,v_j]$, for each  $i\neq j$, and loops $ii$ as circles outside the $n$-gon, intersecting the $n$-gon only at $v_i$.
		We say that two distinct edges \emph{intersect} if the associated curves intersect.		
		This happens either if they have a common extremity, or if the edges are $ik$ and $jl$ with $i<j<k<l$.
		A \emph{stellar cycle} of this embedding $\mathring{K}_n$ is a spanning subgraph with $n$ pairwise intersecting edges or loops.
		Let $\cal S_n$ denote the set of stellar cycles of $ \mathring{K}_n$.
		
		\begin{claim}
			The cones $C_S=\cone(\frac{1}{2}(a^i+a^j)\colon ij\in E(S))$, for all $S\in \cal S_n$, form a regular unimodular Hilbert triangulation of $C$.
		\end{claim}
		\begin{proof}[Sketch]
			These cones are all Hilbert.
			Up to a linear transformation by $(2A^{-1})^\top$, it is the regular triangulation of the second hypersimplex given in~\cite{DeLoera_Sturmfels_Thomas_1995} to which we attach the simplex composed of $2e^i$ and the simplicial facet $\{x_i=1\}$, for $i=1,\ldots,n$.
			Attaching these simplices preserves regularity and determinant computations yield unimodularity.
			\end{proof}
		\item Suppose that $A$ is a thick te-interlace.
		There are four cases: $n=4$ or $6$ and Case~\ref{equation:hb_thick_te-interlaces_cas1} or~\ref{equation:hb_thick_te-interlaces_cas2}.
		We used Polymake~\cite{Gawrilow_Joswig_2000} to check regularity and a simple algorithm to check unimodularity. See Figure~\ref{figure:triangulation} for the case $n=4$.
		\begin{enumerate}[label=\em\alph*., ref=\theenumi.\alph*.]
			\item  By taking the triangulation of Case~\ref{caseproof:second_hypersimplex_triang} restricted to the boundary of $C$ and adding to each of its cones the generator $h=\frac{1}{4}\sum_j a^j$, we obtain a unimodular triangulation of $C$, see Figure~\ref{fig:first}.
			\item Here, we set $h^i = \frac{3}{4}a^i+\frac{1}{4}\sum_{j\neq i} a^j$, for $i=1,\ldots,n$, and the triangulation has a more complex structure, see Figure~\ref{fig:second}.
		\end{enumerate}
	\end{enumerate}

	\begin{figure}[!ht]
		\begin{subfigure}{0.33\textwidth}
			\begin{center}			
				\definecolor{wrwrwr}{rgb}{0,0,0}
				\definecolor{blue}{rgb}{0.08235294117647059,0.396078431372549,0.7529411764705882}
				\definecolor{orange}{rgb}{0.7803921568627451,0.3137254901960784,0}
				
				\begin{tikzpicture}[line cap=round,line join=round,>=triangle 45,x=1cm,y=1cm, scale=0.32]
					\fill[color=blue,fill=blue,fill opacity=0.10000000149011612] (-2.36,0.87) -- (0.52,-2.35) -- (1.06,2.45) -- cycle;
					\fill[color=blue,fill=blue,fill opacity=0.10000000149011612] (-2.36,0.87) -- (-1.39,0.23) -- (0.52,-2.35) -- cycle;
					\fill[color=blue,fill=blue,fill opacity=0.10000000149011612] (1.06,2.45) -- (-1.39,0.23) -- (-2.36,0.87) -- cycle;
					\fill[color=blue,fill=blue,fill opacity=0.10000000149011612] (1.06,2.45) -- (-1.39,0.23) -- (0.52,-2.35) -- cycle;
					\draw [color=wrwrwr] (-4.78,-0.05)-- (-2.9,-3.93);
					\draw [color=wrwrwr] (-2.9,-3.93)-- (3.94,-0.77);
					\draw [dash pattern=on 1pt off 1pt,color=wrwrwr] (-4.78,-0.05)-- (3.94,-0.77);
					\draw [color=wrwrwr] (-4.78,-0.05)-- (-1.82,5.67);
					\draw [color=wrwrwr] (-1.82,5.67)-- (-2.9,-3.93);
					\draw [color=wrwrwr] (-1.82,5.67)-- (3.94,-0.77);
					\draw [color=wrwrwr] (-3.3,2.81)-- (-3.84,-1.99);
					\draw [color=wrwrwr] (-2.36,0.87)-- (-3.3,2.81);
					\draw [color=wrwrwr] (-3.84,-1.99)-- (-2.36,0.87);
					\draw [color=wrwrwr] (-2.36,0.87)-- (0.52,-2.35);
					\draw [color=wrwrwr] (0.52,-2.35)-- (1.06,2.45);
					\draw [color=wrwrwr] (1.06,2.45)-- (-2.36,0.87);
					\draw [dash pattern=on 1pt off 1pt,color=wrwrwr] (-0.42,-0.41)-- (-3.84,-1.99);
					\draw [dash pattern=on 1pt off 1pt,color=wrwrwr] (-3.84,-1.99)-- (0.52,-2.35);
					\draw [dash pattern=on 1pt off 1pt,color=wrwrwr] (0.52,-2.35)-- (-0.42,-0.41);
					\draw [dash pattern=on 1pt off 1pt,color=wrwrwr] (-0.42,-0.41)-- (-3.3,2.81);
					\draw [dash pattern=on 1pt off 1pt,color=wrwrwr] (-3.3,2.81)-- (1.06,2.45);
					\draw [dash pattern=on 1pt off 1pt,color=wrwrwr] (1.06,2.45)-- (-0.42,-0.41);
					\draw [dash pattern=on 1pt off 1pt,color=wrwrwr] (-1.39,0.23)-- (-4.78,-0.05);
					\draw [dash pattern=on 1pt off 1pt,color=wrwrwr] (-1.39,0.23)-- (-0.42,-0.41);
					\draw [dash pattern=on 1pt off 1pt,color=wrwrwr] (-1.39,0.23)-- (-2.9,-3.93);
					\draw [dash pattern=on 1pt off 1pt,color=wrwrwr] (-1.39,0.23)-- (-3.84,-1.99);
					\draw [dash pattern=on 1pt off 1pt,color=wrwrwr] (-1.39,0.23)-- (-1.82,5.67);
					\draw [dash pattern=on 1pt off 1pt,color=wrwrwr] (-1.39,0.23)-- (3.94,-0.77);
					\draw [dash pattern=on 1pt off 1pt,color=wrwrwr] (-1.39,0.23)-- (-3.3,2.81);
					
					\draw [color=blue] (-2.36,0.87)-- (0.52,-2.35);
					\draw [color=blue] (0.52,-2.35)-- (1.06,2.45);
					\draw [color=blue] (1.06,2.45)-- (-2.36,0.87);
					\draw [dash pattern=on 1pt off 1pt,color=blue] (-2.36,0.87)-- (-1.39,0.23);
					\draw [dash pattern=on 1pt off 1pt,color=blue] (-1.39,0.23)-- (0.52,-2.35);
					\draw [dash pattern=on 1pt off 1pt,color=blue] (1.06,2.45)-- (-1.39,0.23);
					\begin{scriptsize}
						\draw [fill=black] (-4.78,-0.05) circle (2.5pt);
						\draw[color=black] (-5,0.3) node {$a^1$};
						\draw [fill=black] (-2.9,-3.93) circle (2.5pt);
						\draw[color=black] (-2.7,-3.6) node[below] {$a^2$};
						\draw [fill=black] (3.94,-0.77) circle (2.5pt);
						\draw[color=black] (4.5,-0.4) node {$a^3$};
						\draw [fill=black] (-1.82,5.67) circle (2.5pt);
						\draw[color=black] (-1.5,6.2) node {$a^4$};
						\draw [fill=black] (-3.84,-1.99) circle (2.5pt);
						\draw[color=black] (-4.5,-2.6) node {$m^{12}$};
						\draw [fill=black] (0.52,-2.35) circle (2.5pt);
						\draw[color=black] (1.2,-2.5) node {$m^{23}$};
						\draw [fill=black] (-2.36,0.87) circle (2.5pt);
						\draw[color=black] (-1.4,1.4) node {\contour{white}{$m^{24}$}};
						\draw [fill=black] (-3.3,2.81) circle (2.5pt);
						\draw[color=black] (-3.2,3.4) node {\contour{white}{$m^{14}$}};
						\draw [fill=black] (1.06,2.45) circle (2.5pt);
						\draw[color=black] (1.5,3.1) node {\contour{white}{$m^{34}$}};
						\draw [fill=black] (-0.42,-0.41) circle (2.5pt);
						\draw[color=black] (-0.1,0.2) node {\contour{white}{$m^{13}$}};
						\draw [fill=black] (-1.39,0.23) circle (2.5pt);
						\draw[color=black,fill=white] (-1.262234835279858,0.5809165742649476) node {\contour{white}{$h$}};
					\end{scriptsize}
				\end{tikzpicture}
				
				4 of this type.
				
				\begin{tikzpicture}[line cap=round,line join=round,>=triangle 45,x=1cm,y=1cm,scale=0.32]
					\fill[color=orange,fill=orange,fill opacity=0.1] (-1.39,0.23) -- (-2.36,0.87) -- (-3.84,-1.99) -- cycle;
					\fill[color=orange,fill=orange,fill opacity=0.1] (-1.39,0.23) -- (-3.84,-1.99) -- (-2.9,-3.93) -- cycle;
					\fill[color=orange,fill=orange,fill opacity=0.1] (-1.39,0.23) -- (-2.36,0.87) -- (-2.9,-3.93) -- cycle;
					\fill[color=orange,fill=orange,fill opacity=0.1] (-3.84,-1.99) -- (-2.9,-3.93) -- (-2.36,0.87) -- cycle;
					\draw [color=wrwrwr] (-4.78,-0.05)-- (-2.9,-3.93);
					\draw [color=wrwrwr] (-2.9,-3.93)-- (3.94,-0.77);
					\draw [dash pattern=on 1pt off 1pt,color=wrwrwr] (-4.78,-0.05)-- (3.94,-0.77);
					\draw [color=wrwrwr] (-4.78,-0.05)-- (-1.82,5.67);
					\draw [color=wrwrwr] (-1.82,5.67)-- (-2.9,-3.93);
					\draw [color=wrwrwr] (-1.82,5.67)-- (3.94,-0.77);
					\draw [color=wrwrwr] (-3.3,2.81)-- (-3.84,-1.99);
					\draw [color=wrwrwr] (-2.36,0.87)-- (-3.3,2.81);
					\draw [color=wrwrwr] (-3.84,-1.99)-- (-2.36,0.87);
					\draw [color=wrwrwr] (-2.36,0.87)-- (0.52,-2.35);
					\draw [color=wrwrwr] (0.52,-2.35)-- (1.06,2.45);
					\draw [color=wrwrwr] (1.06,2.45)-- (-2.36,0.87);
					\draw [dash pattern=on 1pt off 1pt,color=wrwrwr] (-0.42,-0.41)-- (-3.84,-1.99);
					\draw [dash pattern=on 1pt off 1pt,color=wrwrwr] (-3.84,-1.99)-- (0.52,-2.35);
					\draw [dash pattern=on 1pt off 1pt,color=wrwrwr] (0.52,-2.35)-- (-0.42,-0.41);
					\draw [dash pattern=on 1pt off 1pt,color=wrwrwr] (-0.42,-0.41)-- (-3.3,2.81);
					\draw [dash pattern=on 1pt off 1pt,color=wrwrwr] (-3.3,2.81)-- (1.06,2.45);
					\draw [dash pattern=on 1pt off 1pt,color=wrwrwr] (1.06,2.45)-- (-0.42,-0.41);
					\draw [dash pattern=on 1pt off 1pt,color=wrwrwr] (-1.39,0.23)-- (-4.78,-0.05);
					\draw [dash pattern=on 1pt off 1pt,color=wrwrwr] (-1.39,0.23)-- (-2.36,0.87);
					\draw [dash pattern=on 1pt off 1pt,color=wrwrwr] (-1.39,0.23)-- (-0.42,-0.41);
					\draw [dash pattern=on 1pt off 1pt,color=wrwrwr] (-1.39,0.23)-- (-2.9,-3.93);
					\draw [dash pattern=on 1pt off 1pt,color=wrwrwr] (-1.39,0.23)-- (0.52,-2.35);
					\draw [dash pattern=on 1pt off 1pt,color=wrwrwr] (-1.39,0.23)-- (1.06,2.45);
					\draw [dash pattern=on 1pt off 1pt,color=wrwrwr] (-1.39,0.23)-- (-1.82,5.67);
					\draw [dash pattern=on 1pt off 1pt,color=wrwrwr] (-1.39,0.23)-- (3.94,-0.77);
					\draw [dash pattern=on 1pt off 1pt,color=wrwrwr] (-1.39,0.23)-- (-3.3,2.81);
					
					\draw [color=orange] (-1.39,0.23)-- (-2.36,0.87);
					\draw [color=orange] (-2.36,0.87)-- (-3.84,-1.99);
					\draw [dash pattern=on 1pt off 1pt,color=orange] (-3.84,-1.99)-- (-1.39,0.23);
					\draw [color=orange] (-3.84,-1.99)-- (-2.9,-3.93);
					\draw [color=orange] (-2.9,-3.93)-- (-1.39,0.23);
					\draw [color=orange] (-2.36,0.87)-- (-2.9,-3.93);
					\begin{scriptsize}
						\draw [fill=black] (-4.78,-0.05) circle (2.5pt);
						\draw[color=black] (-5,0.3) node {$a^1$};
						\draw [fill=black] (-2.9,-3.93) circle (2.5pt);
						\draw[color=black] (-2.7,-3.6) node[below] {$a^2$};
						\draw [fill=black] (3.94,-0.77) circle (2.5pt);
						\draw[color=black] (4.5,-0.4) node {$a^3$};
						\draw [fill=black] (-1.82,5.67) circle (2.5pt);
						\draw[color=black] (-1.5,6.2) node {$a^4$};
						\draw [fill=black] (-3.84,-1.99) circle (2.5pt);
						\draw[color=black] (-4.5,-2.6) node {$m^{12}$};
						\draw [fill=black] (0.52,-2.35) circle (2.5pt);
						\draw[color=black] (1.2,-2.5) node {$m^{23}$};
						\draw [fill=black] (-2.36,0.87) circle (2.5pt);
						\draw[color=black] (-1.4,1.4) node {\contour{white}{$m^{24}$}};
						\draw [fill=black] (-3.3,2.81) circle (2.5pt);
						\draw[color=black] (-3.2,3.4) node {\contour{white}{$m^{14}$}};
						\draw [fill=black] (1.06,2.45) circle (2.5pt);
						\draw[color=black] (1.5,3.1) node {\contour{white}{$m^{34}$}};
						\draw [fill=black] (-0.42,-0.41) circle (2.5pt);
						\draw[color=black] (-0.1,0.2) node {\contour{white}{$m^{13}$}};
						\draw [fill=black] (-1.39,0.23) circle (2.5pt);
						\draw[color=black,fill=white] (-1.262234835279858,0.5809165742649476) node {\contour{white}{$h$}};
					\end{scriptsize}
				\end{tikzpicture}
				
				12 of this type.
			\end{center}
			\subcaption[t]{Case~\ref{equation:hb_thick_te-interlaces_cas1}.}
			\label{fig:first}
		\end{subfigure}
		\hfill
		\vrule
		\begin{subfigure}{0.66\textwidth}
			\begin{center}
				\definecolor{wrwrwr}{rgb}{0,0,0}
				\definecolor{blue}{rgb}{0.08235294117647059,0.396078431372549,0.7529411764705882}
				\definecolor{orange}{rgb}{0.7803921568627451,0.3137254901960784,0}
				\definecolor{green}{HTML}{6CB359}
				\definecolor{blue}{rgb}{0.08235294117647059,0.396078431372549,0.7529411764705882}
				\definecolor{purple}{rgb}{0.6,0.2,1}
				\begin{tabular}{cc}
					\begin{tikzpicture}[line cap=round,line join=round,>=triangle 45,x=1cm,y=1cm,scale=0.32]
						\fill[color=blue,fill=blue,fill opacity=0.10000000149011612] (-1.893333333333333,-1.1566666666666667) -- (-3.84,-1.99) -- (-2.9,-3.93) -- cycle;
						\fill[color=blue,fill=blue,fill opacity=0.10000000149011612] (-3.84,-1.99) -- (0.52,-2.35) -- (-2.9,-3.93) -- cycle;
						\fill[color=blue,fill=blue,fill opacity=0.10000000149011612] (-1.893333333333333,-1.1566666666666667) -- (-3.84,-1.99) -- (0.52,-2.35) -- cycle;
						\fill[color=blue,fill=blue,fill opacity=0.10000000149011612] (-1.893333333333333,-1.1566666666666667) -- (-2.9,-3.93) -- (0.52,-2.35) -- cycle;
						\draw [color=wrwrwr] (-4.78,-0.05)-- (-2.9,-3.93);
						\draw [color=wrwrwr] (-2.9,-3.93)-- (3.94,-0.77);
						\draw [dash pattern=on 1pt off 1pt,color=wrwrwr] (-4.78,-0.05)-- (3.94,-0.77);
						\draw [color=wrwrwr] (-4.78,-0.05)-- (-1.82,5.67);
						\draw [color=wrwrwr] (-1.82,5.67)-- (-2.9,-3.93);
						\draw [color=wrwrwr] (-1.82,5.67)-- (3.94,-0.77);
						\draw [color=wrwrwr] (-3.3,2.81)-- (-3.84,-1.99);
						\draw [color=wrwrwr] (-2.36,0.87)-- (-3.3,2.81);
						\draw [color=wrwrwr] (-3.84,-1.99)-- (-2.36,0.87);
						\draw [color=wrwrwr] (-2.36,0.87)-- (0.52,-2.35);
						\draw [color=wrwrwr] (0.52,-2.35)-- (1.06,2.45);
						\draw [color=wrwrwr] (1.06,2.45)-- (-2.36,0.87);
						\draw [dash pattern=on 1pt off 1pt,color=wrwrwr] (-0.42,-0.41)-- (-3.84,-1.99);
						\draw [dash pattern=on 1pt off 1pt,color=wrwrwr] (0.52,-2.35)-- (-0.42,-0.41);
						\draw [dash pattern=on 1pt off 1pt,color=wrwrwr] (-0.42,-0.41)-- (-3.3,2.81);
						\draw [dash pattern=on 1pt off 1pt,color=wrwrwr] (-3.3,2.81)-- (1.06,2.45);
						\draw [dash pattern=on 1pt off 1pt,color=wrwrwr] (1.06,2.45)-- (-0.42,-0.41);
						\draw [dash pattern=on 1pt off 1pt,color=wrwrwr] (-2.52,0.13666666666666666)-- (-4.78,-0.05);
						\draw [dash pattern=on 1pt off 1pt,color=wrwrwr] (-1.893333333333333,-1.1566666666666667)-- (-2.9,-3.93);
						\draw [dash pattern=on 1pt off 1pt,color=wrwrwr] (-1.893333333333333,-1.1566666666666667)-- (-3.84,-1.99);
						\draw [dash pattern=on 1pt off 1pt,color=wrwrwr] (-1.893333333333333,-1.1566666666666667)-- (0.52,-2.35);
						\draw [dash pattern=on 1pt off 1pt,color=wrwrwr] (-1.893333333333333,-1.1566666666666667)-- (-2.36,0.87);
						\draw [dash pattern=on 1pt off 1pt,color=wrwrwr] (-2.52,0.13666666666666666)-- (-3.3,2.81);
						\draw [dash pattern=on 1pt off 1pt,color=wrwrwr] (-2.52,0.13666666666666666)-- (-3.84,-1.99);
						\draw [dash pattern=on 1pt off 1pt,color=wrwrwr] (-2.52,0.13666666666666666)-- (-0.42,-0.41);
						\draw [dash pattern=on 1pt off 1pt,color=wrwrwr] (-2.36,0.87)-- (-1.5333333333333332,2.043333333333333);
						\draw [dash pattern=on 1pt off 1pt,color=wrwrwr] (-1.5333333333333332,2.043333333333333)-- (1.06,2.45);
						\draw [dash pattern=on 1pt off 1pt,color=wrwrwr] (-1.5333333333333332,2.043333333333333)-- (-3.3,2.81);
						\draw [dash pattern=on 1pt off 1pt,color=wrwrwr] (-1.5333333333333332,2.043333333333333)-- (-1.82,5.67);
						\draw [dash pattern=on 1pt off 1pt,color=wrwrwr] (1.06,2.45)-- (0.3866666666666666,-0.1033333333333333);
						\draw [dash pattern=on 1pt off 1pt,color=wrwrwr] (0.3866666666666666,-0.1033333333333333)-- (-0.42,-0.41);
						\draw [dash pattern=on 1pt off 1pt,color=wrwrwr] (0.3866666666666666,-0.1033333333333333)-- (0.52,-2.35);
						\draw [dash pattern=on 1pt off 1pt,color=wrwrwr] (0.3866666666666666,-0.1033333333333333)-- (3.94,-0.77);
						\draw [dash pattern=on 1pt off 1pt,color=wrwrwr] (-2.36,0.87)-- (-0.42,-0.41);
						\draw [dash pattern=on 1pt off 1pt,color=wrwrwr] (-2.36,0.87)-- (0.3866666666666666,-0.1033333333333333);
						\draw [dash pattern=on 1pt off 1pt,color=wrwrwr] (-0.42,-0.41)-- (-1.5333333333333332,2.043333333333333);
						\draw [dash pattern=on 1pt off 1pt,color=wrwrwr] (-0.42,-0.41)-- (-1.893333333333333,-1.1566666666666667);
						\draw [dash pattern=on 1pt off 1pt,color=wrwrwr] (-2.36,0.87)-- (-2.52,0.13666666666666666);
						\draw [color=blue] (-1.893333333333333,-1.1566666666666667)-- (-3.84,-1.99);
						\draw [color=blue] (-3.84,-1.99)-- (-2.9,-3.93);
						\draw [color=blue,dash pattern=on 1pt off 1pt] (-3.84,-1.99)-- (0.52,-2.35);
						
						\draw [color=wrwrwr,dash pattern=on 1pt off 1pt] (-2.52,0.13666666666666666)-- (-1.893333333333333,-1.1566666666666667);
						\draw [color=wrwrwr,dash pattern=on 1pt off 1pt] (-1.893333333333333,-1.1566666666666667)--(0.3866666666666666,-0.1033333333333333);
						\draw [color=wrwrwr,dash pattern=on 1pt off 1pt] (0.3866666666666666,-0.1033333333333333)-- (-1.5333333333333332,2.043333333333333);
						\draw [color=wrwrwr,dash pattern=on 1pt off 1pt] (-1.5333333333333332,2.043333333333333)-- (-2.52,0.13666666666666666);
						
						\draw [color=blue] (0.52,-2.35)-- (-2.9,-3.93);
						\draw [color=blue] (0.52,-2.35)-- (-1.893333333333333,-1.1566666666666667);
						\draw [color=blue] (-1.893333333333333,-1.1566666666666667)-- (-2.9,-3.93);
						\begin{scriptsize}
							\draw [fill=black] (-4.78,-0.05) circle (2.5pt);
							\draw[color=black] (-5,0.3) node {$a^1$};
							\draw [fill=black] (-2.9,-3.93) circle (2.5pt);
							\draw[color=black] (-2.7,-3.6) node[below] {$a^2$};
							\draw [fill=black] (3.94,-0.77) circle (2.5pt);
							\draw[color=black] (4.5,-0.4) node {$a^3$};
							\draw [fill=black] (-1.82,5.67) circle (2.5pt);
							\draw[color=black] (-1.5,6.2) node {$a^4$};
							\draw [fill=black] (-3.84,-1.99) circle (2.5pt);
							\draw[color=black] (-4.5,-2.6) node {$m^{12}$};
							\draw [fill=black] (0.52,-2.35) circle (2.5pt);
							\draw[color=black] (1.2,-2.5) node {$m^{23}$};
							\draw [fill=black] (-2.36,0.87) circle (2.5pt);
							\draw[color=black] (-1.4,1.4) node {\contour{white}{$m^{24}$}};
							\draw [fill=black] (-3.3,2.81) circle (2.5pt);
							\draw[color=black] (-3.2,3.4) node {\contour{white}{$m^{14}$}};
							\draw [fill=black] (1.06,2.45) circle (2.5pt);
							\draw[color=black] (1.5,3.1) node {\contour{white}{$m^{34}$}};
							\draw [fill=black] (-0.42,-0.41) circle (2.5pt);
							\draw[color=black] (-0.1,0.2) node {\contour{white}{$m^{13}$}};
							\draw [fill=black] (-2.52,0.13666666666666666) circle (2pt);
							\draw[color=black] (-2.2,-0.2) node {\contour{white}{$h^1$}};
							\draw [fill=black] (0.3866666666666666,-0.1033333333333333) circle (2pt);
							\draw[color=black] (01,-0.4) node {\contour{white}{$h^3$}};
							\draw [fill=black] (-1.893333333333333,-1.1566666666666667) circle (2pt);
							\draw[color=black] (-1.6,-1.6) node {\contour{white}{$h^2$}};
							\draw [fill=black] (-1.5333333333333332,2.043333333333333) circle (2pt);
							\draw[color=black] (-0.9,2.6) node {\contour{white}{$h^4$}};
						\end{scriptsize}
					\end{tikzpicture} & \begin{tikzpicture}[line cap=round,line join=round,>=triangle 45,x=1cm,y=1cm,scale=0.32]
						\fill[color=green,fill=green,fill opacity=0.1] (0.3866666666666666,-0.1033333333333333) -- (-1.5333333333333332,2.043333333333333) -- (1.06,2.45) -- cycle;
						\fill[color=green,fill=green,fill opacity=0.1] (-1.5333333333333332,2.043333333333333) -- (1.06,2.45) -- (-2.36,0.87) -- cycle;
						\fill[color=green,fill=green,fill opacity=0.1] (0.3866666666666666,-0.1033333333333333) -- (-1.5333333333333332,2.043333333333333) -- (-2.36,0.87) -- cycle;
						
						\fill[color=green,fill=green,fill opacity=0.1] (-2.36,0.87) -- (0.3866666666666666,-0.1033333333333333) -- (1.06,2.45) -- cycle;
						
						\draw [color=wrwrwr] (-4.78,-0.05)-- (-2.9,-3.93);
						\draw [color=wrwrwr] (-2.9,-3.93)-- (3.94,-0.77);
						\draw [dash pattern=on 1pt off 1pt,color=wrwrwr] (-4.78,-0.05)-- (3.94,-0.77);
						\draw [color=wrwrwr] (-4.78,-0.05)-- (-1.82,5.67);
						\draw [color=wrwrwr] (-1.82,5.67)-- (-2.9,-3.93);
						\draw [color=wrwrwr] (-1.82,5.67)-- (3.94,-0.77);
						\draw [color=wrwrwr] (-3.3,2.81)-- (-3.84,-1.99);
						\draw [color=wrwrwr] (-2.36,0.87)-- (-3.3,2.81);
						\draw [color=wrwrwr] (-3.84,-1.99)-- (-2.36,0.87);
						\draw [color=wrwrwr] (-2.36,0.87)-- (0.52,-2.35);
						\draw [color=wrwrwr] (0.52,-2.35)-- (1.06,2.45);
						\draw [color=wrwrwr] (1.06,2.45)-- (-2.36,0.87);
						\draw [dash pattern=on 1pt off 1pt,color=wrwrwr] (-0.42,-0.41)-- (-3.84,-1.99);
						\draw [dash pattern=on 1pt off 1pt,color=wrwrwr] (-3.84,-1.99)-- (0.52,-2.35);
						\draw [dash pattern=on 1pt off 1pt,color=wrwrwr] (0.52,-2.35)-- (-0.42,-0.41);
						\draw [dash pattern=on 1pt off 1pt,color=wrwrwr] (-0.42,-0.41)-- (-3.3,2.81);
						\draw [dash pattern=on 1pt off 1pt,color=wrwrwr] (-3.3,2.81)-- (1.06,2.45);
						\draw [dash pattern=on 1pt off 1pt,color=wrwrwr] (1.06,2.45)-- (-0.42,-0.41);
						\draw [dash pattern=on 1pt off 1pt,color=wrwrwr] (-2.52,0.13666666666666666)-- (-4.78,-0.05);
						\draw [dash pattern=on 1pt off 1pt,color=wrwrwr] (-1.893333333333333,-1.1566666666666667)-- (-2.9,-3.93);
						\draw [dash pattern=on 1pt off 1pt,color=wrwrwr] (-1.893333333333333,-1.1566666666666667)-- (-3.84,-1.99);
						\draw [dash pattern=on 1pt off 1pt,color=wrwrwr] (-1.893333333333333,-1.1566666666666667)-- (0.52,-2.35);
						\draw [dash pattern=on 1pt off 1pt,color=wrwrwr] (-1.893333333333333,-1.1566666666666667)-- (-2.36,0.87);
						\draw [dash pattern=on 1pt off 1pt,color=wrwrwr] (-2.52,0.13666666666666666)-- (-3.3,2.81);
						\draw [dash pattern=on 1pt off 1pt,color=wrwrwr] (-2.52,0.13666666666666666)-- (-3.84,-1.99);
						\draw [dash pattern=on 1pt off 1pt,color=wrwrwr] (-2.52,0.13666666666666666)-- (-0.42,-0.41);
						\draw [dash pattern=on 1pt off 1pt,color=wrwrwr] (-2.36,0.87)-- (-1.5333333333333332,2.043333333333333);
						\draw [dash pattern=on 1pt off 1pt,color=wrwrwr] (-1.5333333333333332,2.043333333333333)-- (1.06,2.45);
						\draw [dash pattern=on 1pt off 1pt,color=wrwrwr] (-1.5333333333333332,2.043333333333333)-- (-3.3,2.81);
						\draw [dash pattern=on 1pt off 1pt,color=wrwrwr] (-1.5333333333333332,2.043333333333333)-- (-1.82,5.67);
						\draw [dash pattern=on 1pt off 1pt,color=wrwrwr] (1.06,2.45)-- (0.3866666666666666,-0.1033333333333333);
						\draw [dash pattern=on 1pt off 1pt,color=wrwrwr] (0.3866666666666666,-0.1033333333333333)-- (-0.42,-0.41);
						\draw [dash pattern=on 1pt off 1pt,color=wrwrwr] (0.3866666666666666,-0.1033333333333333)-- (0.52,-2.35);
						\draw [dash pattern=on 1pt off 1pt,color=wrwrwr] (0.3866666666666666,-0.1033333333333333)-- (3.94,-0.77);
						\draw [dash pattern=on 1pt off 1pt,color=wrwrwr] (-2.36,0.87)-- (-0.42,-0.41);
						\draw [dash pattern=on 1pt off 1pt,color=wrwrwr] (-2.36,0.87)-- (0.3866666666666666,-0.1033333333333333);
						\draw [dash pattern=on 1pt off 1pt,color=wrwrwr] (-0.42,-0.41)-- (-1.5333333333333332,2.043333333333333);
						\draw [dash pattern=on 1pt off 1pt,color=wrwrwr] (-2.36,0.87)-- (-2.52,0.13666666666666666);
						\draw [color=wrwrwr,dash pattern=on 1pt off 1pt] (-2.52,0.13666666666666666)-- (-1.893333333333333,-1.1566666666666667);
						\draw [color=wrwrwr,dash pattern=on 1pt off 1pt] (-1.893333333333333,-1.1566666666666667)--(0.3866666666666666,-0.1033333333333333);
						\draw [color=wrwrwr,dash pattern=on 1pt off 1pt] (0.3866666666666666,-0.1033333333333333)-- (-1.5333333333333332,2.043333333333333);
						\draw [color=wrwrwr,dash pattern=on 1pt off 1pt] (-1.5333333333333332,2.043333333333333)-- (-2.52,0.13666666666666666);

						\draw [color=green,dash pattern=on 1pt off 1pt] (0.3866666666666666,-0.1033333333333333)-- (-1.5333333333333332,2.043333333333333);
						\draw [color=green] (0.3866666666666666,-0.1033333333333333)--(1.06,2.45);
						\draw [color=green] (0.3866666666666666,-0.1033333333333333)-- (-2.36,0.87);
						\draw [color=green] (-1.5333333333333332,2.043333333333333)--(1.06,2.45);
						\draw [color=green] (-1.5333333333333332,2.043333333333333)-- (-2.36,0.87);
						\draw [color=green] (1.06,2.45)-- (-2.36,0.87);

						\begin{scriptsize}
							\draw [fill=black] (-4.78,-0.05) circle (2.5pt);
							\draw[color=black] (-5,0.3) node {$a^1$};
							\draw [fill=black] (-2.9,-3.93) circle (2.5pt);
							\draw[color=black] (-2.7,-3.6) node[below] {$a^2$};
							\draw [fill=black] (3.94,-0.77) circle (2.5pt);
							\draw[color=black] (4.5,-0.4) node {$a^3$};
							\draw [fill=black] (-1.82,5.67) circle (2.5pt);
							\draw[color=black] (-1.5,6.2) node {$a^4$};
							\draw [fill=black] (-3.84,-1.99) circle (2.5pt);
							\draw[color=black] (-4.5,-2.6) node {$m^{12}$};
							\draw [fill=black] (0.52,-2.35) circle (2.5pt);
							\draw[color=black] (1.2,-2.5) node {$m^{23}$};
							\draw [fill=black] (-2.36,0.87) circle (2.5pt);
							\draw[color=black] (-1.4,1.4) node {\contour{white}{$m^{24}$}};
							\draw [fill=black] (-3.3,2.81) circle (2.5pt);
							\draw[color=black] (-3.2,3.4) node {\contour{white}{$m^{14}$}};
							\draw [fill=black] (1.06,2.45) circle (2.5pt);
							\draw[color=black] (1.5,3.1) node {\contour{white}{$m^{34}$}};
							\draw [fill=black] (-0.42,-0.41) circle (2.5pt);
							\draw[color=black] (-0.1,0.2) node {\contour{white}{$m^{13}$}};
							\draw [fill=black] (-2.52,0.13666666666666666) circle (2pt);
							\draw[color=black] (-2.2,-0.2) node {\contour{white}{$h^1$}};
							\draw [fill=black] (0.3866666666666666,-0.1033333333333333) circle (2pt);
							\draw[color=black] (01,-0.4) node {\contour{white}{$h^3$}};
							\draw [fill=black] (-1.893333333333333,-1.1566666666666667) circle (2pt);
							\draw[color=black] (-1.6,-1.6) node {\contour{white}{$h^2$}};
							\draw [fill=black] (-1.5333333333333332,2.043333333333333) circle (2pt);
							\draw[color=black] (-0.9,2.6) node {\contour{white}{$h^4$}};
						\end{scriptsize}
					\end{tikzpicture} \\
					12 of this type. & 8 of this type. 
				\end{tabular}
				\begin{tabular}{cc}
					\begin{tikzpicture}[line cap=round,line join=round,>=triangle 45,x=1cm,y=1cm,scale=0.32]
						\fill[color=orange,fill=orange,fill opacity=0.1] (0.3866666666666666,-0.1033333333333333) -- (-1.5333333333333332,2.043333333333333) -- (-0.42,-0.41) -- cycle;
						\fill[color=orange,fill=orange,fill opacity=0.1] (-1.5333333333333332,2.043333333333333) -- (-0.42,-0.41) -- (-2.36,0.87) -- cycle;
						\fill[color=orange,fill=orange,fill opacity=0.1] (0.3866666666666666,-0.1033333333333333) -- (-1.5333333333333332,2.043333333333333) -- (-2.36,0.87) -- cycle;
						
						\fill[color=orange,fill=orange,fill opacity=0.1] (-2.36,0.87) -- (0.3866666666666666,-0.1033333333333333) -- (-0.42,-0.41) -- cycle;
						
						\draw [color=wrwrwr] (-4.78,-0.05)-- (-2.9,-3.93);
						\draw [color=wrwrwr] (-2.9,-3.93)-- (3.94,-0.77);
						\draw [dash pattern=on 1pt off 1pt,color=wrwrwr] (-4.78,-0.05)-- (3.94,-0.77);
						\draw [color=wrwrwr] (-4.78,-0.05)-- (-1.82,5.67);
						\draw [color=wrwrwr] (-1.82,5.67)-- (-2.9,-3.93);
						\draw [color=wrwrwr] (-1.82,5.67)-- (3.94,-0.77);
						\draw [color=wrwrwr] (-3.3,2.81)-- (-3.84,-1.99);
						\draw [color=wrwrwr] (-2.36,0.87)-- (-3.3,2.81);
						\draw [color=wrwrwr] (-3.84,-1.99)-- (-2.36,0.87);
						\draw [color=wrwrwr] (-2.36,0.87)-- (0.52,-2.35);
						\draw [color=wrwrwr] (0.52,-2.35)-- (1.06,2.45);
						\draw [color=wrwrwr] (1.06,2.45)-- (-2.36,0.87);
						\draw [dash pattern=on 1pt off 1pt,color=wrwrwr] (-0.42,-0.41)-- (-3.84,-1.99);
						\draw [dash pattern=on 1pt off 1pt,color=wrwrwr] (-3.84,-1.99)-- (0.52,-2.35);
						\draw [dash pattern=on 1pt off 1pt,color=wrwrwr] (0.52,-2.35)-- (-0.42,-0.41);
						\draw [dash pattern=on 1pt off 1pt,color=wrwrwr] (-0.42,-0.41)-- (-3.3,2.81);
						\draw [dash pattern=on 1pt off 1pt,color=wrwrwr] (-3.3,2.81)-- (1.06,2.45);
						\draw [dash pattern=on 1pt off 1pt,color=wrwrwr] (1.06,2.45)-- (-0.42,-0.41);
						\draw [dash pattern=on 1pt off 1pt,color=wrwrwr] (-2.52,0.13666666666666666)-- (-4.78,-0.05);
						\draw [dash pattern=on 1pt off 1pt,color=wrwrwr] (-1.893333333333333,-1.1566666666666667)-- (-2.9,-3.93);
						\draw [dash pattern=on 1pt off 1pt,color=wrwrwr] (-1.893333333333333,-1.1566666666666667)-- (-3.84,-1.99);
						\draw [dash pattern=on 1pt off 1pt,color=wrwrwr] (-1.893333333333333,-1.1566666666666667)-- (0.52,-2.35);
						\draw [dash pattern=on 1pt off 1pt,color=wrwrwr] (-1.893333333333333,-1.1566666666666667)-- (-2.36,0.87);
						\draw [dash pattern=on 1pt off 1pt,color=wrwrwr] (-2.52,0.13666666666666666)-- (-3.3,2.81);
						\draw [dash pattern=on 1pt off 1pt,color=wrwrwr] (-2.52,0.13666666666666666)-- (-3.84,-1.99);
						\draw [dash pattern=on 1pt off 1pt,color=wrwrwr] (-2.52,0.13666666666666666)-- (-0.42,-0.41);
						\draw [dash pattern=on 1pt off 1pt,color=wrwrwr] (-2.36,0.87)-- (-1.5333333333333332,2.043333333333333);
						\draw [dash pattern=on 1pt off 1pt,color=wrwrwr] (-1.5333333333333332,2.043333333333333)-- (1.06,2.45);
						\draw [dash pattern=on 1pt off 1pt,color=wrwrwr] (-1.5333333333333332,2.043333333333333)-- (-3.3,2.81);
						\draw [dash pattern=on 1pt off 1pt,color=wrwrwr] (-1.5333333333333332,2.043333333333333)-- (-1.82,5.67);
						\draw [dash pattern=on 1pt off 1pt,color=wrwrwr] (1.06,2.45)-- (0.3866666666666666,-0.1033333333333333);
						\draw [dash pattern=on 1pt off 1pt,color=wrwrwr] (0.3866666666666666,-0.1033333333333333)-- (-0.42,-0.41);
						\draw [dash pattern=on 1pt off 1pt,color=wrwrwr] (0.3866666666666666,-0.1033333333333333)-- (0.52,-2.35);
						\draw [dash pattern=on 1pt off 1pt,color=wrwrwr] (0.3866666666666666,-0.1033333333333333)-- (3.94,-0.77);
						\draw [dash pattern=on 1pt off 1pt,color=wrwrwr] (-2.36,0.87)-- (-0.42,-0.41);
						\draw [dash pattern=on 1pt off 1pt,color=wrwrwr] (-2.36,0.87)-- (0.3866666666666666,-0.1033333333333333);
						\draw [dash pattern=on 1pt off 1pt,color=wrwrwr] (-0.42,-0.41)-- (-1.5333333333333332,2.043333333333333);
						\draw [dash pattern=on 1pt off 1pt,color=wrwrwr] (-2.36,0.87)-- (-2.52,0.13666666666666666);
						\draw [color=wrwrwr,dash pattern=on 1pt off 1pt] (-2.52,0.13666666666666666)-- (-1.893333333333333,-1.1566666666666667);
						\draw [color=wrwrwr,dash pattern=on 1pt off 1pt] (-1.893333333333333,-1.1566666666666667)--(0.3866666666666666,-0.1033333333333333);
						\draw [color=wrwrwr,dash pattern=on 1pt off 1pt] (0.3866666666666666,-0.1033333333333333)-- (-1.5333333333333332,2.043333333333333);
						\draw [color=wrwrwr,dash pattern=on 1pt off 1pt] (-1.5333333333333332,2.043333333333333)-- (-2.52,0.13666666666666666);

						\draw [color=orange] (0.3866666666666666,-0.1033333333333333)-- (-1.5333333333333332,2.043333333333333);
						\draw [color=orange] (0.3866666666666666,-0.1033333333333333)--(-0.42,-0.41);
						\draw [color=orange] (0.3866666666666666,-0.1033333333333333)-- (-2.36,0.87);
						\draw [color=orange,dash pattern=on 1pt off 1pt] (-1.5333333333333332,2.043333333333333)--(-0.42,-0.41);
						\draw [color=orange] (-1.5333333333333332,2.043333333333333)-- (-2.36,0.87);
						\draw [color=orange] (-0.42,-0.41)-- (-2.36,0.87);

						\begin{scriptsize}
							\draw [fill=black] (-4.78,-0.05) circle (2.5pt);
							\draw[color=black] (-5,0.3) node {$a^1$};
							\draw [fill=black] (-2.9,-3.93) circle (2.5pt);
							\draw[color=black] (-2.7,-3.6) node[below] {$a^2$};
							\draw [fill=black] (3.94,-0.77) circle (2.5pt);
							\draw[color=black] (4.5,-0.4) node {$a^3$};
							\draw [fill=black] (-1.82,5.67) circle (2.5pt);
							\draw[color=black] (-1.5,6.2) node {$a^4$};
							\draw [fill=black] (-3.84,-1.99) circle (2.5pt);
							\draw[color=black] (-4.5,-2.6) node {$m^{12}$};
							\draw [fill=black] (0.52,-2.35) circle (2.5pt);
							\draw[color=black] (1.2,-2.5) node {$m^{23}$};
							\draw [fill=black] (-2.36,0.87) circle (2.5pt);
							\draw[color=black] (-1.4,1.4) node {\contour{white}{$m^{24}$}};
							\draw [fill=black] (-3.3,2.81) circle (2.5pt);
							\draw[color=black] (-3.2,3.4) node {\contour{white}{$m^{14}$}};
							\draw [fill=black] (1.06,2.45) circle (2.5pt);
							\draw[color=black] (1.5,3.1) node {\contour{white}{$m^{34}$}};
							\draw [fill=black] (-0.42,-0.41) circle (2.5pt);
							\draw[color=black] (-0.1,0.2) node {\contour{white}{$m^{13}$}};
							\draw [fill=black] (-2.52,0.13666666666666666) circle (2pt);
							\draw[color=black] (-2.2,-0.2) node {\contour{white}{$h^1$}};
							\draw [fill=black] (0.3866666666666666,-0.1033333333333333) circle (2pt);
							\draw[color=black] (01,-0.4) node {\contour{white}{$h^3$}};
							\draw [fill=black] (-1.893333333333333,-1.1566666666666667) circle (2pt);
							\draw[color=black] (-1.6,-1.6) node {\contour{white}{$h^2$}};
							\draw [fill=black] (-1.5333333333333332,2.043333333333333) circle (2pt);
							\draw[color=black] (-0.9,2.6) node {\contour{white}{$h^4$}};
						\end{scriptsize}
					\end{tikzpicture} & \begin{tikzpicture}[line cap=round,line join=round,>=triangle 45,x=1cm,y=1cm,scale=0.32]
						\fill[color=purple,fill=purple,fill opacity=0.1] (-0.42,-0.41) -- (-3.84,-1.99) -- (0.52,-2.35) -- cycle;
						\fill[color=purple,fill=purple,fill opacity=0.1] (-1.893333333333333,-1.1566666666666667) -- (-0.42,-0.41) -- (-3.84,-1.99) -- cycle;
						\fill[color=purple,fill=purple,fill opacity=0.1] (-1.893333333333333,-1.1566666666666667) -- (-3.84,-1.99) -- (0.52,-2.35) -- cycle;
						\fill[color=purple,fill=purple,fill opacity=0.1] (-1.893333333333333,-1.1566666666666667) -- (0.52,-2.35) -- (-0.42,-0.41) -- cycle;
						\draw [color=wrwrwr] (-4.78,-0.05)-- (-2.9,-3.93);
						\draw [color=wrwrwr] (-2.9,-3.93)-- (3.94,-0.77);
						\draw [dash pattern=on 1pt off 1pt,color=wrwrwr] (-4.78,-0.05)-- (3.94,-0.77);
						\draw [color=wrwrwr] (-4.78,-0.05)-- (-1.82,5.67);
						\draw [color=wrwrwr] (-1.82,5.67)-- (-2.9,-3.93);
						\draw [color=wrwrwr] (-1.82,5.67)-- (3.94,-0.77);
						\draw [color=wrwrwr] (-3.3,2.81)-- (-3.84,-1.99);
						\draw [color=wrwrwr] (-2.36,0.87)-- (-3.3,2.81);
						\draw [color=wrwrwr] (-3.84,-1.99)-- (-2.36,0.87);
						\draw [color=wrwrwr] (-2.36,0.87)-- (0.52,-2.35);
						\draw [color=wrwrwr] (0.52,-2.35)-- (1.06,2.45);
						\draw [color=wrwrwr] (1.06,2.45)-- (-2.36,0.87);
						\draw [dash pattern=on 1pt off 1pt,color=wrwrwr] (-0.42,-0.41)-- (-3.84,-1.99);
						\draw [dash pattern=on 1pt off 1pt,color=wrwrwr] (-3.84,-1.99)-- (0.52,-2.35);
						\draw [dash pattern=on 1pt off 1pt,color=wrwrwr] (0.52,-2.35)-- (-0.42,-0.41);
						\draw [dash pattern=on 1pt off 1pt,color=wrwrwr] (-0.42,-0.41)-- (-3.3,2.81);
						\draw [dash pattern=on 1pt off 1pt,color=wrwrwr] (-3.3,2.81)-- (1.06,2.45);
						\draw [dash pattern=on 1pt off 1pt,color=wrwrwr] (1.06,2.45)-- (-0.42,-0.41);
						\draw [dash pattern=on 1pt off 1pt,color=wrwrwr] (-2.52,0.13666666666666666)-- (-4.78,-0.05);
						\draw [dash pattern=on 1pt off 1pt,color=wrwrwr] (-1.893333333333333,-1.1566666666666667)-- (-2.9,-3.93);
						\draw [dash pattern=on 1pt off 1pt,color=wrwrwr] (-1.893333333333333,-1.1566666666666667)-- (-3.84,-1.99);
						\draw [dash pattern=on 1pt off 1pt,color=wrwrwr] (-1.893333333333333,-1.1566666666666667)-- (0.52,-2.35);
						\draw [dash pattern=on 1pt off 1pt,color=wrwrwr] (-1.893333333333333,-1.1566666666666667)-- (-2.36,0.87);
						\draw [dash pattern=on 1pt off 1pt,color=wrwrwr] (-2.52,0.13666666666666666)-- (-3.3,2.81);
						\draw [dash pattern=on 1pt off 1pt,color=wrwrwr] (-2.52,0.13666666666666666)-- (-3.84,-1.99);
						\draw [dash pattern=on 1pt off 1pt,color=wrwrwr] (-2.52,0.13666666666666666)-- (-0.42,-0.41);
						\draw [dash pattern=on 1pt off 1pt,color=wrwrwr] (-2.36,0.87)-- (-1.5333333333333332,2.043333333333333);
						\draw [dash pattern=on 1pt off 1pt,color=wrwrwr] (-1.5333333333333332,2.043333333333333)-- (1.06,2.45);
						\draw [dash pattern=on 1pt off 1pt,color=wrwrwr] (-1.5333333333333332,2.043333333333333)-- (-3.3,2.81);
						\draw [dash pattern=on 1pt off 1pt,color=wrwrwr] (-1.5333333333333332,2.043333333333333)-- (-1.82,5.67);
						\draw [dash pattern=on 1pt off 1pt,color=wrwrwr] (1.06,2.45)-- (0.3866666666666666,-0.1033333333333333);
						\draw [dash pattern=on 1pt off 1pt,color=wrwrwr] (0.3866666666666666,-0.1033333333333333)-- (-0.42,-0.41);
						\draw [dash pattern=on 1pt off 1pt,color=wrwrwr] (0.3866666666666666,-0.1033333333333333)-- (0.52,-2.35);
						\draw [dash pattern=on 1pt off 1pt,color=wrwrwr] (0.3866666666666666,-0.1033333333333333)-- (3.94,-0.77);
						\draw [dash pattern=on 1pt off 1pt,color=wrwrwr] (-2.36,0.87)-- (-0.42,-0.41);
						\draw [dash pattern=on 1pt off 1pt,color=wrwrwr] (-2.36,0.87)-- (0.3866666666666666,-0.1033333333333333);
						\draw [dash pattern=on 1pt off 1pt,color=wrwrwr] (-0.42,-0.41)-- (-1.5333333333333332,2.043333333333333);
						\draw [dash pattern=on 1pt off 1pt,color=wrwrwr] (-0.42,-0.41)-- (-1.893333333333333,-1.1566666666666667);
						\draw [dash pattern=on 1pt off 1pt,color=wrwrwr] (-2.36,0.87)-- (-2.52,0.13666666666666666);
						\draw [color=purple] (-0.42,-0.41) -- (-3.84,-1.99);
						\draw [color=purple] (-0.42,-0.41) -- (0.52,-2.35);
						\draw [color=purple] (-0.42,-0.41) -- (-1.893333333333333,-1.1566666666666667);
						\draw [color=purple] (-1.893333333333333,-1.1566666666666667)-- (-3.84,-1.99);
						\draw [color=purple] (-1.893333333333333,-1.1566666666666667)-- (0.52,-2.35);
						\draw [color=purple] (-3.84,-1.99) -- (0.52,-2.35);
						\draw [color=wrwrwr,dash pattern=on 1pt off 1pt] (-2.52,0.13666666666666666)-- (-1.893333333333333,-1.1566666666666667);
						\draw [color=wrwrwr,dash pattern=on 1pt off 1pt] (-1.893333333333333,-1.1566666666666667)--(0.3866666666666666,-0.1033333333333333);
						\draw [color=wrwrwr,dash pattern=on 1pt off 1pt] (0.3866666666666666,-0.1033333333333333)-- (-1.5333333333333332,2.043333333333333);
						\draw [color=wrwrwr,dash pattern=on 1pt off 1pt] (-1.5333333333333332,2.043333333333333)-- (-2.52,0.13666666666666666);
						
						\begin{scriptsize}
							\draw [fill=black] (-4.78,-0.05) circle (2.5pt);
							\draw[color=black] (-5,0.3) node {$a^1$};
							\draw [fill=black] (-2.9,-3.93) circle (2.5pt);
							\draw[color=black] (-2.7,-3.6) node[below] {$a^2$};
							\draw [fill=black] (3.94,-0.77) circle (2.5pt);
							\draw[color=black] (4.5,-0.4) node {$a^3$};
							\draw [fill=black] (-1.82,5.67) circle (2.5pt);
							\draw[color=black] (-1.5,6.2) node {$a^4$};
							\draw [fill=black] (-3.84,-1.99) circle (2.5pt);
							\draw[color=black] (-4.5,-2.6) node {$m^{12}$};
							\draw [fill=black] (0.52,-2.35) circle (2.5pt);
							\draw[color=black] (1.2,-2.5) node {$m^{23}$};
							\draw [fill=black] (-2.36,0.87) circle (2.5pt);
							\draw[color=black] (-1.4,1.4) node {\contour{white}{$m^{24}$}};
							\draw [fill=black] (-3.3,2.81) circle (2.5pt);
							\draw[color=black] (-3.2,3.4) node {\contour{white}{$m^{14}$}};
							\draw [fill=black] (1.06,2.45) circle (2.5pt);
							\draw[color=black] (1.5,3.1) node {\contour{white}{$m^{34}$}};
							\draw [fill=black] (-0.42,-0.41) circle (2.5pt);
							\draw[color=black] (-0.1,0.2) node {\contour{white}{$m^{13}$}};
							\draw [fill=black] (-2.52,0.13666666666666666) circle (2pt);
							\draw[color=black] (-2.2,-0.2) node {\contour{white}{$h^1$}};
							\draw [fill=black] (0.3866666666666666,-0.1033333333333333) circle (2pt);
							\draw[color=black] (01,-0.4) node {\contour{white}{$h^3$}};
							\draw [fill=black] (-1.893333333333333,-1.1566666666666667) circle (2pt);
							\draw[color=black] (-1.6,-1.6) node {\contour{white}{$h^2$}};
							\draw [fill=black] (-1.5333333333333332,2.043333333333333) circle (2pt);
							\draw[color=black] (-0.9,2.6) node {\contour{white}{$h^4$}};
						\end{scriptsize}
					\end{tikzpicture} \\
					4 of this type.& 4 of this type.
				\end{tabular}
			\end{center}
			\subcaption[t]{Case~\ref{equation:hb_thick_te-interlaces_cas2}.}
			\label{fig:second}
		\end{subfigure}
		\caption{The regular unimodular Hilbert triangulation in the case $n=4$. We define $m^{ij}=\frac{1}{2}(a^i+a^j)$, for $i\neq j$. Up to symmetry, the cones in the triangulations are generated by the vertices of the colored tetrahedra in each case.}
		\label{figure:triangulation}
	\end{figure}
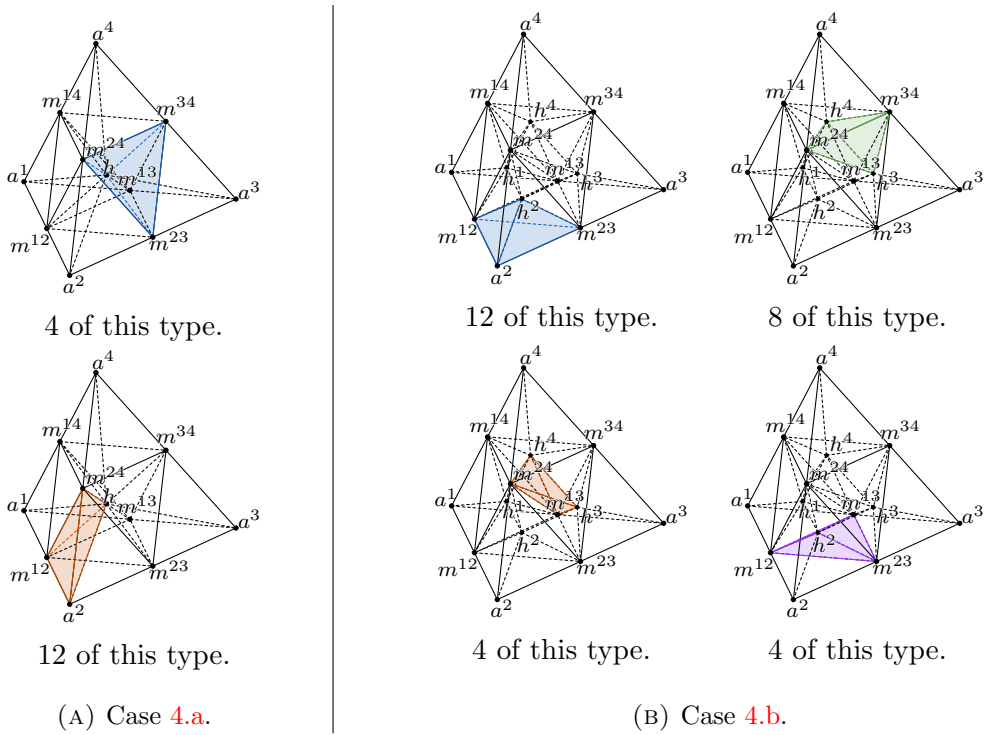
	\end{proof}

In Section~\ref{section:decomposition}, we raised the question of the existence of a decomposition theorem without the full row rank hypothesis.
One could hope nevertheless that Theorem~\ref{theorem:hb_te-bricks} provides a Hilbert basis in the general case, for instance by applying it to each simplicial subcone, which would be a te-cone. 
This is not the case as the following example shows.
For the cone $C=\cone(M)$ generated by the rows of the matrix $M$ given at the end of Section~\ref{section:decomposition}, this strategy yields four vectors, namely the half-sum of each te-lace:
$$h_1^\top=\begin{bmatrix}
	2&1&1&1
\end{bmatrix},
h_2^\top=\begin{bmatrix}
	1&1&1&2
\end{bmatrix},
h_3^\top=\begin{bmatrix}
	1&1&0&1
\end{bmatrix},
h_4^\top=\begin{bmatrix}
	1&0&1&1
\end{bmatrix}.$$
Yet, $h_1$ and $h_2$ are not Hilbert basis elements of $C$ as $h_1 = h_4+M_1^\top$ and $h_2 = h_3+M_6^\top$.
Independently from the existence of a decomposition theorem, the following question remains open:
\textit{In general, which integer decomposition properties do cones generated by totally equimodular matrices satisfy?}

\section{Proofs of the results of Section~\ref{section:decomposition}: decomposition}\label{section:proofs_decomposition}
We first provide notations and remarks that will be used throughout the next two sections.

Let $A$ be an $m\times n$ matrix.
For subsets $I\subseteq \{1,\dots,m\}$ and $J\subseteq \{1,\dots,n\}$ we write $A_I$ the $|I|\times n$ submatrix of $A$ whose rows are indexed by $I$, $A^J$ the $m\times |J|$ submatrix of $A$ whose columns are indexed by $J$, and $A^J_I=(A_I)^J=(A^J)_I$ the associated $|I|\times|J|$ submatrix of $A$.
Whenever $I$ or $J$ is a singleton, we omit the curly brackets. For instance, choosing of row index $i$ and a column index $j$, we write $A_i=A_{\{i\}}$ the $i$-th row vector of $A$, and $A^j=A^{\{j\}}$ the $j$-th column vector of $A$.
We write $A_{\widehat{i}}=A_{\{1,\ldots,i-1,i+1,\ldots,m\}}$ and $A^{\widehat{j}}=A_{\{1,\ldots,j-1,j+1,\ldots,n\}}$.
Moreover, we will write $A^{<j}=A^{\{1,\ldots,j-1\}}$ and $A^{>j}=A^{\{j+1,\ldots,m\}}$, with $A_{<i}$ and $A_{>i}$ defined similarly.
For a transposition $(ij)$, we define $\Pi_{(ij)}$ the matrix obtained from the identity matrix $\mathbf{I}_n$ by exchanging the coefficients $(i,i)$ and $(j,j)$ with the coefficients $(i,j)$ and $(j,i)$, respectively.

For $a\in \Q^n$, we define $\diag (a)$ the square matrix whose $i$-th diagonal coefficient is $a_i$, and whose other coefficients are zeroes.
When $a$ is $\pm 1$, $\diag(a)$ is called a \emph{signing matrix}.

To simplify notation, we will often directly write $r\in A$ to denote a row $r=A_i$ of $A$. Then, a pivot $(r,j)$ with $j\in \supp(r)$, is the same as the pivot $(i,j)$ defined in Section~\ref{section:def}.
When $A$ is totally equimodular, we denote by $\widehat{A}$ the rescaled version of $A$ obtained by dividing each row $r$ of $A$ by $\eqdet(r)$, which is a $0,\!\pm1$ matrix.



We say that {\em $r$ and $s$ coincide on their common support} when $r_{\supp(r)\cap\supp(s)}=s_{\supp(r)\cap\supp(s)}$.
They are {\em opposites on their common support} when $r$ and $-s$ coincide on their common support.
For two vectors $r$ and $s$ which coincide on they common support, $r\triangle s$ denotes the vector obtained from $r+s$ by setting the coordinates in $\supp(r)\cap \supp(r)$ to $0$.
\medskip

We will use intensively the following easy remarks about full row rank sets of size two, which come from the fact that $\begin{bmatrix}1 & 1\\1 & -1\end{bmatrix}$ is the only $2\times 2$ minimally non-totally unimodular matrix, up to resigning rows and columns:
\begin{itemize}
	\item[\textbullet] $\{r,s\}$ is a tu-set if and only if $\supp(r)\neq\supp(s)$ and $r$ and $s$ either coincide or are opposites on their common support.
	\item[\textbullet] $\{r,s\}$ is a te-lace if and only if $\supp(r)=\supp(s)$ and $r\neq\pm s$.
	\item[\textbullet] In a te-interlace, all the rows have the same support.
	\item[\textbullet] If $r$ and $s$ coincide on their common support and $c\in\supp(r)\cap\supp(s)$, then $s'=\{r,s\}/\!\!/(r,c)= s-r = s\triangle(-r)$.
	Hence $\supp(s')=\supp(s)\triangle\supp(r)$.
	\item[\textbullet] Trimming a te-interlace yields a matrix of the form $2B$, where $B$ can be made $0,\!1$ by resigning rows and columns.
\end{itemize}


%
%


\subsection{Pivoting and trimming in te-bricks}

In this section, we provide preliminary results about the impact of pivoting and trimming on equimodularity and te-bricks.

First, the following is immediate.
\begin{lemma}\label{pivotE}
Let $A$ be a full row rank $m\times n$ matrix and $(i,j)$ be a pivot of $A$.
Then, $A$ is equimodular if and only if $A/(i,j)$ is equimodular.
\end{lemma}
\begin{proof} 
The row operations involved in the $(i,j)$-pivot divide all the $m\times m$ determinants of $A$ by $A_i^j$.
\end{proof}
The previous result has a counterpart in terms of trims.
\begin{lemma}[Lemma~\ref{lemma:trim_E}]\footnote{Here and in the next section, such a reference points to the same statement in the first part.}\label{pivoteqE}
Let $r$ be a $0,\!\pm1$ row of a full row rank $m\times n$ matrix $A$.
Then, $A$ is equimodular if and only if $A/\!\!/(r,j)$ is equimodular for all $j\in\supp(r)$.
\end{lemma}
\begin{proof}
The ``only if'' part follows from Lemma~\ref{pivotE}, in fact, for $J\sqcup\{j\}\subseteq \{1,\ldots,n\}$ of size $m$, we have $\det\left(\left(A/(r,j)\right)^{J\sqcup\{j\}}\right)=\pm\det\left(\left(A/\!\!/(r,j)\right)^{J}\right)$.

Suppose now that, for all $k\in\supp(r)$, $A/\!\!/(r,k)$ is equimodular, and let $d_k=\eqdet(A/\!\!/(r,k))$.
We have $d_k\neq 0$ for all $k\in\supp(r)$.

Let $i\neq j$ be in $\supp(r)$.
Note that both $A^i$ and $A^j$ are nonzero.
If $A^i=\pm A^j$, then trimming with respect to $(r,i)$ and $(r,j)$ yield the same determinants, up to the sign, thus $d_i = d_j$.
Otherwise, there exists a submatrix $D$ of $A$ containing both $A^i$ and $A^j$ which is invertible.
Let $D_i=D/\!\!/(r,i)$ and $D_j=D/\!\!/(r,j)$.
By definition, $|\det(D_i)|=d_i$ and $|\det(D_j)|=d_j$.
Now, since $r$ is $0,\!\pm1$, and $D_i$ and $D_j$ are respectively trimmed from $D$ by pivoting with respect to $(r,i)$ and $(r,j)$, we have $|\det(D_i)|=|\det(D)|=|\det(D_j)|$.
Thus $d_i=d_j$, and all nonzero $d_k$'s have the same value.

Thoses values are presicely the values of the $n\times n$ nonzero determinants of $A$, thus $A$ is equimodular.
\end{proof}

Lemmas~\ref{pivotE} and~\ref{pivoteqE} imply the following, which is an equivalent formulation of Theorem~\ref{theorem:pivot_TE}.

\begin{theorem}[Theorem~\ref{theorem:pivot_TE}]\label{pivotpreserveTE}
A matrix is totally equimodular if and only if any sequence of pivots and trims yields a totally equimodular matrix.
\end{theorem}

We provide additional properties about trimming in te-laces and te-interlaces.
\begin{lemma}\label{pivotte-knot}
Trimming a te-lace of size at least three yields a te-lace.
\end{lemma}
\begin{proof}
This holds because trimming a minimally non-totally unimodular matrix of size at least three yields a minimally non-totally unimodular matrix, as trimming in a $0,\!\pm1$ square matrix preserves the determinant of the whole matrix, and total unimodularity of the proper submatrices.
\end{proof}

\begin{lemma}\label{lemma:trimming_te-interlaces}
After trimming and rescaling a te-interlace, there are no te-laces of size two. In particular, there are no te-interlaces.
\end{lemma}

\begin{proof}
Let $p=(r,j)$ be a pivot of a te-interlace $I$, and let $I'=I/\!\!/p$.
Recall that all the rows of $I$ have the same support.
Up to resigning rows and columns, we may assume that $r$ has only ones on its support, and that $t_j=-1$ for all $t\in I\setminus r$.
Then, $I'$ is composed of the rows $t+r$, for all $t\in I\setminus r$, to which we deleted the $j$-th column.
Since every such $t$ is $0,\!\pm1$, the vectors in $I'$ have entries only in $\{0,2\}$.
Hence, rescaling $I'$ yields a $0,\!1$ matrix which contains no te-lace of size two, as there are no $0,\!1$ minimally non-totally unimodular matrices of size $2\times 2$.
\end{proof}

\begin{lemma}\label{3rows}
Let $X=\{\ell,r,s\}$ be a te-set with $\{\ell,r\}$ a tu-set and $|\supp(z)|\geq 2$ for all $z\in X$.
Let $\{r',s'\}=(X/(\ell,j))\setminus\{\ell\}$ for some $j\in\supp(r)\cap\supp(\ell)$.
If $\supp(r')=\supp(s')$, then either $\{r,s\}$ or $\{\ell,r,s\}$ is a te-lace.
\end{lemma}
\begin{proof}
Suppose that $\{r,s\}$ is not a te-lace, that is, it is a tu-set and in particular $\supp(r)\neq\supp(s)$. 
Let us prove that then $\{\ell,r,s\}$ is a te-lace.

Since $\{\ell, r\}$ and $\{r,s\}$ are both tu-sets, up to resigning rows, we may assume that $r$ and $\ell$ coincide on their common support, as well as $r$ and $s$.
Moreover, $r'=r\triangle(-\ell)$.
Since $X$ has full row rank, $\{r',s'\}$ has full row rank, and hence is a te-lace as $\supp(r')=\supp(s')$.

If $s'=s$, then $\supp(s)\neq \supp(\ell)$ hence $\{s,\ell\}$ is a tu-set. 
Since $X/\!\!/(\ell,j)$ is a te-lace, $X$ is not tu.
But all its proper subsets are tu, thus $X$ is a te-lace.
If $s'\neq s$, then $s'=s\triangle (-\ell)$ up to resigning.
But then $\supp(r)=\supp(s)$, contradicting the assumption that $\{r,s\}$ is a tu-set.
\end{proof}

\subsection{Two kinds of te-interlaces}

This section is devoted to the proof that only two types of te-interlaces exist: the thin and thick ones.
Before, let us provide a few technical results that we shall use.
\medskip

\subsubsection{Minimally non-totally unimodular matrices}
Recall that a matrix is minimally non-totally unimodular when it is not totally unimodular but all its proper submatrices are.
First, Camion~\cite{Camion_1965} provides the following about minimal non-totally unimodular matrices.
\begin{theorem}[\cite{Camion_1965}]\label{theorem:camion}
Let $A$ be a minimally non-totally unimodular matrix, then $\det (A)=\pm 2$ and $A^{-1}$ has only $\pm\frac{1}{2}$ entries.
Furthermore, each row and each column of $A$ has an even number of nonzeroes and the sum of all entries in $A$ equals $2\pmod{4}$.
\end{theorem}
As a consequence we obtain the following useful lemma.
\begin{lemma}\label{lemma:adding_column_to_min-nonTU}
Adding a column $x$ to a minimally non-totally unimodular matrix $A$ yields a totally equimodular matrix $A'$ if and only if $x$ is $\mathbf{0}$ or $\pm A^j$ for some column $A^j$ of $A$.
\end{lemma}
\begin{proof}
The ``if'' part is direct.
Let us prove the ``only if'' part.
When $A$ is of size $2$, the result is immediate since there are only four $\pm 1$ vectors in $\R^2$ and up to resigning, the only minimally non-totally unimodular matrix of size $2$ is $\begin{bmatrix}
1&1\\-1&1
\end{bmatrix}$.
Let $A$ be of size greater than $2$ and $x\neq\mathbf{0}$.
There exists a column $A^j$ of $A$ such that the matrix $B$ obtained from $A'$ by removing column $A^j$ is invertible.
Since $A'$ is totally equimodular and $A$ is minimally non-totally unimodular, $B$ is minimally non-totally unimodular.
By the second part of Theorem~\ref{theorem:camion}, each row of both $A$ and $B$ has an even sum. Consequently, $x$ and $A^j$ have the same support.
Now, if $x\neq\pm A^j$, then, up to resigning, $A'$ contains 
$\begin{bmatrix}
1&1\\-1&1
\end{bmatrix}$ of determinant $2$ as a proper submatrix, which contradicts the total equimodularity of $A'$.		
\end{proof}
The following follows from the definition of total equimodularity and the fact that if a $0,\!\pm1$ matrix is not totally unimodular, then it contains a minimally non-totally unimodular matrix.

\begin{lemma}\label{containte-lace}
A te-set which is not a tu-set contains a te-lace.
\end{lemma}	
\subsubsection{Complement matrices}

During the proof of Lemma~\ref{lemma:trimming_te-interlaces}, we observed that a single trim in a te-interlace yields a $0,\!1$ matrix, up to rescaling and resigning.
We introduce specific classes of $0,\!1$ matrices that arise from trims of te-interlaces. 
Notably, these matrices generalize complement totally unimodular and complement minimally non-totally unimodular matrices, which are thoroughly studied in Truemper's book~\cite{Truemper_1992} and in~\cite{Truemper_1980}.

In $\{0,1\}$, we define the following operation which corresponds the mod~$2$ sum: $a\oplus b=a+b\pmod{2}$.
More explicitly, we have $0\oplus0 =1\oplus1=0$ and $1\oplus 0=0\oplus 1=1$.
Let $A$ be a $0,\!1$ matrix of size $m\times n$, and $i$ a row index of $A$.
The \emph{row-$i$ complement} $A_{[i]}$ of $A$ is the $0,\!1$ matrix whose $i'$-th row is $$\begin{cases}
A_{i}, \text{ if }i'=i,\\
A_{i'}\oplus A_i, \text{ otherwise}.
\end{cases}$$

Column-$j$ complements $A^{[j]}$ are defined similarly.
Some interesting properties arise from row or column complement operations.
These operations are involutions: $(A_{[i]})_{[i]}=A=(A^{[j]})^{[j]}$.
Any row and column complement operation commute: $(A_{[i]})^{[j]}=(A^{[j]})_{[i]}$.
Moreover, up to permutation of rows, respectively of columns, we have $(A_{[i]})_{[i']} = \Pi_{(ii')} A_{[i']} $, respectively $(A^{[j]})^{[j']} = A^{[j']} \Pi_{(jj')}$.
For convenience, we will write $A^{[0]}_{[0]}=A$, $A^{[0]}_{[i]} = A_{[i]}$, $A^{[j]}_{[0]}=A^{[j]}$, and $A^{[j]}_{[i]}=(A_{[i]})^{[j]}=(A^{[j]})_{[i]}$, for $1\leq i \leq m$ and $1\leq j\leq n$.
Hence, up to permutation of rows and columns, there are $(m+1)\cdot(n+1)$~matrices obtainable from $A$ using row and column complement operations, and the set composed of these matrices is called the \emph{complement orbit} $\cal O (A)$ of $A$:
$$\cal O (A) = \cal O (A^{[j]}_{[i]}) = \left\{A^{[j]}_{[i]}\colon i\in\{0,\ldots,m\},j\in\{0,\ldots,n\}\right\}.$$
A $0,\!1$ matrix is \emph{complement totally equimodular} if all the matrices in its complement orbit are totally equimodular.
Special cases are \emph{complement totally unimodular} matrices, for whom the complement orbit consists of totally unimodular matrices only, and \emph{complement minimally non-totally unimodular} whose complement orbit is composed of minimally non-totally unimodular matrices.
\medskip

Truemper~\cite{Truemper_1980} proved the following, for which we provide a shorter proof.
\begin{theorem}[\cite{Truemper_1980}]\label{theorem:Truemper_complement}
Complement minimally non-totally unimodular matrices are of odd size.
\end{theorem}
\begin{proof}
Let $A$ be a $0,\!1$ minimally non-totally unimodular matrix of even size $n$, and $A_i^j\neq 0$.
By Theorem~\ref{theorem:camion}, the column $A^j$ of $A$ has even support and has an even number of zeroes, since $n$ is even.
Performing the row-$i$ complement operation replaces the support of $A^j$ by its complement, except at coefficient $i$.
This yields a column of odd support, hence $A_{[i]}$ is not minimally non-totally unimodular
by Theorem~\ref{theorem:camion}, and $A$ is not complement minimally non-totally unimodular.
\end{proof}

Let $A$ be a $\pm1$ matrix of size $m\times n$. We define $\nega(A)$ to be the $0,\!1$ matrix whose support is the location of the $-1$ entries of $A$.
In particular $A=\mathbf{J}-2\nega(A)$.
Note that we have $\nega(A_i) = \nega(A)_i$ and $\nega (A^j)=\nega(A)^j$, for $1\leq i\leq m$ and $1\leq j \leq n$.
For $B$ a $0,\!1$ matrix, we denote by $\overline{B} = \mathbf{J} - B$ the $0,\!1$ matrix whose support is the complement of that of $B$.
Since $A$ is a $\pm1$ matrix, we have $\nega(-A) = \overline{\nega(A)}$.

\begin{lemma}\label{lemma:correspondence_sign}
Let $A$ be a $\pm1$ matrix of size $m\times n$. Then, for every $\varepsilon \in \{\pm1\}^m$, we have:

$$\nega{(A\diag(\varepsilon))} =\begin{bmatrix}
\nega{(A)}_1 \oplus \nega{(\varepsilon)}^\top\\
\vdots \\
\nega{(A)}_m \oplus \nega{(\varepsilon)}^\top
\end{bmatrix},$$

and for every $\mu \in \{\pm1\}^n$, we have:

$$\nega{(\diag(\mu)A)} =\begin{bmatrix}
\nega{(A)}^1 \oplus \nega{(\mu)} &\cdots &
\nega{(A)}^n \oplus \nega{(\mu)}
\end{bmatrix}.$$
\end{lemma}
\begin{proof}
This comes from $\nega(x\cdot y) = \nega(x)\oplus \nega(y)$, for all $x,y\in\{\pm1\}$.
\end{proof}

Up to resigining rows and columns, let us write
\begin{equation}\label{equation:core}
A=\begin{bmatrix}
1 & \mathbf{1}^\top \\
\mathbf{1} & \mathbf{J} - 2B
\end{bmatrix},
\end{equation}
for a $0,\!1$ matrix $B$ called a \emph{core} of $A$.
Note that we have $\nega{(A)} = \begin{bmatrix}
0&\mathbf{0}^\top\\
\mathbf{0} & B
\end{bmatrix}$.
Finally, we write $A\simeq A'$ when $A$ equals $A'$ up to permutation of rows and columns.
\medskip

In the following, we relate trims of $\pm1$ matrices with the complement orbit of their core.

\begin{lemma}\label{lemma:trimming_and_complement_orbit}
Let $A$ be a $\pm1$ matrix of size $(m+1)\times(n+1)$, and $B$ a core of $A$ as in~\eqref{equation:core}.
Then, for any pivot $p=(i+1,j+1)$ of $A$, with $0\leq i\leq m$ and $0\leq j \leq n$, we have
\begin{equation}\label{equation:trim_complement}
A\trim p \simeq -2A_{i+1}^{j+1}\diag(\varepsilon) B^{[j]}_{[i]}\diag(\mu),
\end{equation}
where $\diag(\varepsilon)$ and $\diag(\mu)$ are signing matrices with $\varepsilon = \mathbf{1} -2 B^j$ and $\mu=\mathbf{1}^\top -2 B_i$, and the convention $B^0=\mathbf{0}$ and $B_0=\mathbf{0}^\top$.
\end{lemma}
\begin{proof}
First, note that $A\trim(1,1) = -2B = -2A_1^1 \diag(\mathbf{1}-2B^0) B_{[0]}^{[0]}\diag(\mathbf{1}^\top-2B_0)$.
Similarly, with a matrix of the form
\begin{equation}\label{equation:nice_form}
A'=\begin{array}{c@{}c}
	& \begin{array}{ccc}
		& j & 
	\end{array}\\
	\begin{array}{c}
		\\i \\  \\
	\end{array} &\begin{bmatrix}
		\mathbf{J}-2D & \mathbf{1} & \mathbf{J}-2E\\
		\mathbf{1}^\top & 1 & \mathbf{1}^\top\\
		\mathbf{J}-2F & \mathbf{1} & \mathbf{J}-2G
	\end{bmatrix}\\&\end{array}\!\!, \text{ we have } \nega{(A')}=\begin{array}{c@{}c}
	& \begin{array}{ccc}
		& j & 
	\end{array}\\
	\begin{array}{c}
		\\i \\  \\
	\end{array} &\begin{bmatrix}
		D & \mathbf{0} & E\\
		\mathbf{0}^\top & 0 & \mathbf{0}^\top\\
		F & \mathbf{0} & G
	\end{bmatrix}\\&\end{array}\!\!,
	\end{equation}
	and we get $A'\trim(i,j) = -2\begin{bmatrix}
D&E\\F&G
\end{bmatrix}$.
\begin{claim}
For $A\in\{\pm1\}^{m\times n}$, $\varepsilon\in\{\pm1\}^m$, $\mu\in\{\pm1\}^n$, $i\in\{1,\ldots,m\}$, and $j\in\{1,\ldots,n\}$, we have:
\begin{equation}\label{equation:diag_trim_col}
	\left(\diag(\varepsilon)A\diag(\mu)\right)\trim(i,j) = \diag(\varepsilon_{\widehat{i}})(A\trim(i,j))\diag(\mu_{\widehat{j}}).
\end{equation}

\end{claim}
\begin{proof}
Since $A$ is $\pm1$, we have $A_i^j =\frac{1}{A_i^j}$, and the classical formula for pivoting yields 
$A\trim (i,j) = A_{\widehat{i}}^{\widehat{j}} - A_i^j A_{\widehat{i}}^{j} A_{i}^{\widehat{j}}$.
By applying this formula to $M=\diag(\varepsilon)A\diag(\mu)$, and since $\varepsilon$ and $\mu$ are $\pm1$, we have:
\begin{align*}
	M\trim (i,j) &= M_{\widehat{i}}^{\widehat{j}} - M_i^j M_{\widehat{i}}^{j} M_{i}^{\widehat{j}} \\ &= \diag(\varepsilon_{\widehat{i}})(A_{\widehat{i}}^{\widehat{j}})\diag(\mu_{\widehat{j}}) - \varepsilon_iA_i^j\mu_j (\diag(\varepsilon_{\widehat{i}})A_{\widehat{i}}^{j} \mu_j)(\varepsilon_i A_{i}^{\widehat{j}}\diag(\mu_{\widehat{j}}))\\
	&=\diag(\varepsilon_{\widehat{i}})\left(A_{\widehat{i}}^{\widehat{j}} - A_i^j A_{\widehat{i}}^{j} A_{i}^{\widehat{j}}\right)\diag(\mu_{\widehat{j}})\\
	&= \diag(\varepsilon_{\widehat{i}})(A\trim(i,j))\diag(\mu_{\widehat{j}}).
\end{align*}
Thus, the claim is proved.
\end{proof}
Our strategy is to sign $A$ in order to obtain a matrix $A'$ of the form~\eqref{equation:nice_form}, and then perform the $(i,j)$-trim, which is the matrix $\nega{(A')}$ after removing the $i$-th row and $j$-th column of $\mathbf{0}$, and multiplying it by $-2$.
We then compare it to the different types of matrices in the complement orbit of $B$.
\paragraph*{\textbf{Case 1: $(i+1,1)$-pivot and $(1,j+1)$-pivot}}~\\
Let  $1\leq i\leq m$, by Lemma~\ref{lemma:correspondence_sign}, we have
$$\nega{(A\diag(A_{i+1}))} =\begin{bmatrix}
\nega{(A)}_1 \oplus \nega{(A)}_{i+1}\\
\nega{(A)}_2 \oplus \nega{(A)}_{i+1}\\
\vdots \\
\nega{(A)}_{i} \oplus \nega{(A)}_{i+1}\\
\nega{(A)}_{i+1} \oplus \nega{(A)}_{i+1}\\
\nega{(A)}_{i+2} \oplus \nega{(A)}_{i+1}\\
\vdots \\
\nega{(A)}_m \oplus \nega{(A)}_{i+1}
\end{bmatrix} = \begin{bmatrix}
0&\mathbf{0}^\top \oplus B_i\\
0& B_1 \oplus B_i\\
\vdots & \vdots \\
0& B_{i-1}\oplus B_i\\
0& \mathbf{0}^\top\\
0& B_{i+1} \oplus B_i\\
\vdots & \vdots \\
0 & B_m \oplus B_i
\end{bmatrix},$$
Since $\mathbf{0}^\top \oplus B_i=B_i$, we have $(A\diag(A_{i+1}))\trim (i+1,1) \simeq -2B_{[i]}$, up to a cyclic permutation of the first $i$ rows. By~\eqref{equation:diag_trim_col} and $A_{i+1} = 1^\top -2B_i$ we obtain: $$A\trim(i+1,1)\simeq-2B_{[i]}\diag(\mathbf{1}^\top-2B_i)\simeq-2 A_{i+1}^1 B_{[i]}\diag(\mathbf{1}^\top-2B_i).$$
We obtain by similar computations and by~\eqref{equation:diag_trim_col} that, for $1\leq j\leq n$: $$A\trim(1,j+1)\simeq-2\diag(\mathbf{1}-2B^j)B^{[j]}\simeq-2A_1^{j+1}\diag(\mathbf{1}-2B^j)B^{[j]}.$$
\medskip

\paragraph*{\textbf{Case 2: $(i+1,j+1)$-pivots}}~\\
Now, let $1\leq i\leq m$ and $1\leq j\leq n$. We first look at how $B_{[i]}^{[j]}$ looks like, depending on $B_i^j$.
\renewcommand{\arraystretch}{1.5}
Let $J_1=\supp(B_i)\setminus \{j\}$, $J_0 =\{1,\ldots,m\}\setminus \supp(B_i)$, $I_1=\supp(B^j)\setminus\{i\}$ and $I_0 =\{1,\ldots,n\}\setminus \supp(B^j)$.
Up to putting row $i$ at the top of $B$ and column $j$ to the left, and up to reorganizing the rows and columns, and if $B_i^j=0$, we may assume that
$$B=\begin{array}{c@{}c}
& \begin{array}{lcc}
	j\;\; & J_1\;\;& J_0
\end{array}\\
\begin{array}{c}
	i\\I_1 \\ I_0
\end{array} &\begin{bmatrix}
	0 &\mathbf{1}^\top & \mathbf{0}^\top\\
	\mathbf{1} & B_{I_1}^{J_1} & B_{I_1}^{J_0} \\
	\mathbf{0} & B_{I_0}^{J_1} & B_{I_0}^{J_0} \\
\end{bmatrix}\\&\end{array}\!\!, \text{ therefore } B_{[i]}^{[j]}=\begin{array}{c@{}c}
& \begin{array}{lcc}
	j\;\; & J_1\;\; & J_0
\end{array}\\
\begin{array}{c}
	i\\I_1 \\ I_0
\end{array} &\begin{bmatrix}
	0 &\mathbf{1}^\top & \mathbf{0}^\top\\
	\mathbf{1} & B_{I_1}^{J_1} & \overline{B_{I_1}^{J_0}} \\
	\mathbf{0} & \overline{B_{I_0}^{J_1}} & B_{I_0}^{J_0} \\
\end{bmatrix}\\&\end{array}.$$
If $B_i^j=1$, we have 
$$B=\begin{array}{c@{}c}
& \begin{array}{lcc}
	j\;\; & J_1\;\; & J_0
\end{array}\\
\begin{array}{c}
	i\\I_1 \\ I_0
\end{array} &\begin{bmatrix}
	1 &\mathbf{1}^\top & \mathbf{0}^\top\\
	\mathbf{1} & B_{I_1}^{J_1} & B_{I_1}^{J_0} \\
	\mathbf{0} & B_{I_0}^{J_1} & B_{I_0}^{J_0} \\
\end{bmatrix}\\&\end{array}\!\!, \text{ therefore } B_{[i]}^{[j]}=\begin{array}{c@{}c}
& \begin{array}{ccc}
	j\;\; & J_1\;\; & J_0
\end{array}\\
\begin{array}{c}
	i\\I_1 \\ I_0
\end{array} &\begin{bmatrix}
	1 &\mathbf{0}^\top & \mathbf{1}^\top\\
	\mathbf{0} & \overline{B_{I_1}^{J_1}} & B_{I_1}^{J_0} \\
	\mathbf{1} & B_{I_0}^{J_1} & \overline{B_{I_0}^{J_0}} \\
\end{bmatrix}\\&\end{array}.$$

We now relate the two cases $B_i^j=0$ and $B_i^j=1$ with the cases $A_{i+1}^{j+1}=1$ and $A_{i+1}^{j+1}=-1$, respectively.

Assume $A_{i+1}^{j+1}=1$, and let $J_{-} = \supp(\nega(A_{i+1}))$ and $J_{+} = (\{2,\ldots,m+1\}\setminus \supp(\nega(A_{i+1})))\setminus\{j+1\}$ and 
$I_{-} = \supp(\nega(A^{j+1}))$ and $I_{+} = (\{2,\ldots,n+1\}\setminus \supp(\nega(A^{j+1})))\setminus\{i+1\}$.
We have, up to reordering the rows and columns: 
$$A \simeq \begin{array}{c@{}c}
& \begin{array}{cccc}
	1 & J_-& j+1 & J_+
\end{array}\\
\begin{array}{c}
	1\\I_- \\ i+1 \\ I_+
\end{array} &\begin{bmatrix}
	1& \mathbf{1}^\top &1&\mathbf{1}^\top\\ 
	\mathbf{1} & A_{I_-}^{J_-} & \mathbf{-1} &  A_{I_-}^{J_+}\\
	1 & \mathbf{-1}^\top & 1 & \mathbf{1}^\top \\
	\mathbf{1}& A_{I_+}^{J_-} & \mathbf{1} & A_{I_+}^{J_+}
\end{bmatrix}\\&\end{array},$$
thus,
$$\diag(A^{j+1})A\diag(A_{i+1}) \simeq \begin{array}{c@{}c}
& \begin{array}{cccc}
	1 & J_-& j+1 & J_+
\end{array}\\
\begin{array}{c}
	1\\I_- \\ i+1 \\ I_+
\end{array} &\begin{bmatrix}
	1 & \mathbf{-1}^\top & 1 & \mathbf{1}^\top \\
	\mathbf{-1} & A_{I_-}^{J_-} & \mathbf{1} & - A_{I_-}^{J_+}\\
	1& \mathbf{1}^\top &1&\mathbf{1}^\top\\ 
	\mathbf{1}& -A_{I_+}^{J_-} & \mathbf{1} & A_{I_+}^{J_+}
\end{bmatrix}\\&\end{array},$$
therefore,
$$\nega (\diag(A^{j+1})A\diag(A_{i+1})) \simeq\begin{bmatrix}
0 &\mathbf{1}^\top & 0& \mathbf{0}^\top\\
\mathbf{1} & B_{I_1}^{J_1} &\mathbf{0}& \overline{B_{I_1}^{J_0}} \\
0 & \mathbf{0}^\top & 0 &\mathbf{0}\\
\mathbf{0} & \overline{B_{I_0}^{J_1}} &\mathbf{0}& B_{I_0}^{J_0} \\
\end{bmatrix},$$
with $J_1 = \{j-1\colon j\in J_-\}$, $J_0 = \{j-1\colon j\in J_+\}$, $I_1 = \{i-1\colon i\in I_-\}$, $I_0 = \{i-1\colon i\in I_+\}$.
This yields, $\diag(A^{j+1})A\diag(A_{i+1})\trim(i+1,j+1) \simeq -2 B_{[i]}^{[j]}$, up to cyclic permutation of the first $i$ rows and $j$ columns.
Finally, by~\eqref{equation:diag_trim_col}
\begin{align*}
A\trim(i+1,j+1) &\simeq -2 \diag(\mathbf{1} -2 B^j) B_{[i]}^{[j]}\diag(\mathbf{1}-2 B_i)\\
& \simeq -2 A_{i+1}^{j+1} \diag(\mathbf{1} -2 B^j) B_{[i]}^{[j]}\diag(\mathbf{1}-2 B_i).
\end{align*}

Assume $A_{i+1}^{j+1}=-1$, and let $J_{-} = \supp(\nega(A_{i+1}))\setminus\{j+1\}$ and $J_{+} = \{2,\ldots,m+1\}\setminus \supp(\nega(A_{i+1}))$ and 
$I_{-} = \supp(\nega(A^{j+1}))\setminus\{i+1\}$ and $I_{+} = \{2,\ldots,n+1\}\setminus \supp(\nega(A^{j+1}))$, we have:
$$A = \begin{array}{c@{}c}
& \begin{array}{cccc}
	1 & J_-& j+1 & J_+
\end{array}\\
\begin{array}{c}
	1\\I_- \\ i+1 \\ I_+
\end{array} &\begin{bmatrix}
	1 & \mathbf{1}^\top & 1 & \mathbf{1}^\top \\
	\mathbf{1} & A_{I_-}^{J_-} & \mathbf{-1} & A_{I_-}^{J_+}\\
	1& \mathbf{-1}^\top &-1&\mathbf{1}^\top\\ 
	\mathbf{1}& A_{I_+}^{J_-} & \mathbf{1} & A_{I_+}^{J_+}
\end{bmatrix}\\&\end{array}.$$
	Multiplying by $\diag(-A^{j+1})$ and $\diag(A_{i+1})$, we obtain the following:
	$$\diag(-A^{j+1})A\diag(A_{i+1}) = \begin{array}{c@{}c}
& \begin{array}{cccc}
	1 & J_-& j+1 & J_+
\end{array}\\
\begin{array}{c}
	1\\I_- \\ i+1 \\ I_+
\end{array} &\begin{bmatrix}
	-1 & \mathbf{1}^\top & 1 & \mathbf{-1}^\top \\
	\mathbf{1} & -A_{I_-}^{J_-} & \mathbf{1} & A_{I_-}^{J_+}\\
	1& \mathbf{1}^\top &1&\mathbf{1}^\top\\ 
	\mathbf{-1}& A_{I_+}^{J_-} & \mathbf{1} & -A_{I_+}^{J_+}
\end{bmatrix}\\&\end{array},$$
therefore 
$$\nega (\diag(-A^{j+1})A\diag(A_{i+1})) =\begin{bmatrix}
1 &\mathbf{0}^\top & 0& \mathbf{1}^\top\\
\mathbf{0} & \overline{B_{I_1}^{J_1}} &\mathbf{0}& {B_{I_1}^{J_0}} \\
0 & \mathbf{0}^\top & 0 &\mathbf{0}\\
\mathbf{1} & {B_{I_0}^{J_1}} &\mathbf{0}& \overline{B_{I_0}^{J_0}} \\
\end{bmatrix},$$
with $J_1 = \{j-1\colon j\in J_-\}$, $J_0 = \{j-1\colon j\in J_+\}$, $I_1 = \{i-1\colon i\in I_-\}$, $I_0 = \{i-1\colon i\in I_+\}$.
This yields, $\diag(-A^{j+1})A\diag(A_{i+1})\trim(i+1,j+1) \simeq -2 B_{[i]}^{[j]}$, up to cyclic permutation of the first $i$ rows and $j$ columns.
Finally, by~\eqref{equation:diag_trim_col} \begin{align*}
A\trim(i+1,j+1) &\simeq 2 \diag(\mathbf{1} -2 B^j) B_{[i]}^{[j]}\diag(\mathbf{1}-2 B_i)\\
& \simeq -2 A_{i+1}^{j+1} \diag(\mathbf{1} -2 B^j) B_{[i]}^{[j]}\diag(\mathbf{1}-2 B_i).
\end{align*}



\end{proof}
\subsubsection{Classification of te-interlaces}
This section is devoted to the proof of the following.
\begin{lemma}\label{lemma:two_types_te-interlaces}
Only two types of te-interlaces exist: the thin and the thick ones.
\end{lemma}
\begin{proof}
Let $A$ be a te-interlace.
Since the rows of $A$ all have the same support, up to restricting to a subset of columns of $A$, we may assume that $A$ is $\pm 1$.
We write
$$A=\begin{bmatrix}
1 & \mathbf{1}^\top \\
\mathbf{1} & \mathbf{J}- 2B
\end{bmatrix},$$
for $B$ a core of $A$.
By Theorem~\ref{lemma:trim_E}, all the trims of $A$ are totally equimodular. Since the complement orbit of $B$ is composed of resigning of these trims by Lemma~\ref{lemma:trimming_and_complement_orbit}, $B$ and all the matrices in its complement orbit are totally equimodular $0,\!1$ matrices.
Moreover, $\eqdet (A) = 2^{n-1}\eqdet(B_{[i]}^{[j]})$ for every pivot $(i+1,j+1)$, hence $\eqdet(B_{[i]}^{[j]})=\eqdet(B)$.
\medskip

If $\eqdet(B)=1$, then $B$ is totally unimodular, as well as all the $B_{[i]}^{[j]}$'s, since $B$ and its complements are all totally equimodular. 
Therefore $B$ is complement totally unimodular.
Moreover, $\eqdet (A)=  2^{n-1}\eqdet(B) = 2^{n-1}$.
Hence, $A$ is a thin te-interlace.
\medskip

We now show that if $\eqdet(B)\geq 2$, then $A$ is a thick te-interlace.
Then, by Lemma~\ref{containte-lace}, $B$ contains a te-lace $D$.
Let $M$ be a $|D|\times|D|$ invertible submatrix of $D$.
Note that $M$ is minimally non-totally unimodular.
\begin{claimnumb}\label{claim:M_is_complement_min-nonTU}
$M$ is complement minimally non-totally unimodular.
\end{claimnumb}
\begin{proof}
Since $A$ is totally equimodular, its submatrix $$A^\prime=\begin{bmatrix}
	1 & \mathbf{1}^\top \\
	\mathbf{1} &  \mathbf{J} -2M
\end{bmatrix},$$ is totally equimodular.
Again, Lemmas~\ref{pivotpreserveTE} and~\ref{lemma:trimming_and_complement_orbit} imply that the complement orbit of $M$ contains only totally equimodular matrices of determinant $\det(M)=\pm2$.
Since proper submatrices of $M$ are totally unimodular, removing a row at index $s>1$ and a column at index $d>1$ of $A^\prime$, and then applying Lemmas~\ref{pivotpreserveTE} and~\ref{lemma:trimming_and_complement_orbit} yields a totally unimodular matrix.
All the proper submatrices of $M$ are therefore complement totally unimodular.
Therefore, $M$ is complement minimally non-totally unimodular, as desired.
\end{proof}
Moreover, $M$ is of odd size by Theorem~\ref{theorem:Truemper_complement}.
By Lemma~\ref{lemma:adding_column_to_min-nonTU} and since $D$ is $0,\!1$,  the nonzero columns of $D$ are copies of those of $M$.
\begin{claimnumb}\label{claim:B_is_te-lace}
$B$ is a te-lace.
\end{claimnumb}
\begin{proof}
Suppose not, then $D\subsetneq B$, and let $r\in B\setminus D$.
There exists a $(|D|+1)\times (|D|+1)$ totally equimodular invertible submatrix $N$ of $D\cup \{r\}$ that we write:
$$N = \begin{bmatrix}
	t& w^\top \\v & M
\end{bmatrix},$$
where $M$, the scalar $t$, and the vectors $v$ and $w$ are $0,\!1$, and $M$ is complement minimally non-totally unimodular by Claim~\ref{claim:M_is_complement_min-nonTU}.
Since $N$ is a submatrix of $B$ whose complement orbit contains only totally equimodular $0,1$ matrices, every matrix in the complement orbit of $N$ is totally equimodular.

\begin{itemize}
	\item[] \hspace{-.5cm} {\em Fact. } We may assume $v=\mathbf{0}$ and $t=1$.
	
	\medskip
	
	By Lemma~\ref{lemma:adding_column_to_min-nonTU}, since $N$ is totally equimodular and $0,\!1$, $v$ is either $0$ or $M^j$, for some $j$.
	Suppose $v=M^j$.
	Since $N$ is invertible, this implies that $t$ and the coefficient $w_j$ above $M^j$ are distinct, and up to permutation of these two columns, suppose $t=1$.
	Therefore, $N$ and its column-$(j+1)$ complement are as follows:
	$$N = \begin{bmatrix}
		1& w_{<j}^\top & 0 & w_{>j}^\top \\
		M^j & M^{<j} & M^j & M^{>j}
	\end{bmatrix} \text{ and } N^{[j+1]} = \begin{bmatrix}
		1& w^\top \\
		\mathbf{0} & M^{[j]}
	\end{bmatrix}.$$
	Moreover, $M^{[j]}$ is minimally non-totally unimodular by Claim~\ref{claim:M_is_complement_min-nonTU}.
	We are therefore in the case $v=\mathbf{0}$, as desired.
	
	\medskip
	
	\item[] \hspace{-.5cm}  {\em Fact. } We may assume $w=\mathbf{0}$.
	
	\medskip
	
	We have
	$$N = \begin{bmatrix}
		1&  w^\top\\
		\mathbf{0} & M
	\end{bmatrix} \text{ and } (N^{-1})^\top = \begin{bmatrix}
		1&  \mathbf{0}^\top\\
		z & (M^{-1})^\top
	\end{bmatrix},$$
	where $z=- (M^{-1})^\top w$.
	Up to replacing $N$ by $N^{[1]}$, which changes $w$ to $w' = 2\cdot\mathbf{1} - w$, having complementary support to $w$, and since $w$ is of odd size, we may assume that $\supp(w)$ is even. 
	By Theorem~\ref{theorem:camion}, $M^{-1}$ has only $\pm \frac{1}{2}$ entries, and hence $z=-(M^{-1})^\top w$ is integer.
	If $w\neq \mathbf{0}$, then $z\neq \mathbf{0}$ since $M$ is invertible.
	Then, any row of $(N^{-1})^\top$ indexed in  $\supp(-(M^{-1})^\top w)$ contains a nonzero integer in its first column and $\pm \frac12$ elsewhere. Such a row is not equimodular, contradicting the total equimodularity of $(N^{-1})^\top$ given by that of $N$ and Theorem~\ref{theorem:invert_TE}.
\end{itemize}

We now end the proof of Claim~\ref{claim:B_is_te-lace}.
By the two above facts, we may assume
$$N = \begin{bmatrix}
	1& \mathbf{0}^\top \\\mathbf{0} & M
\end{bmatrix}  \text{ and therefore }N_{[1]} = \begin{bmatrix}
	1& \mathbf{0}^\top \\\mathbf{1} & M
\end{bmatrix}.$$
By Theorem~\ref{theorem:camion}, since $M$ is $0,\!1$ and minimally non-totally unimodular, it has an even number of $1$'s in each column. As $M$ is of odd size, $\mathbf{1}$ is not a column of $M$.
Therefore, by Lemma~\ref{lemma:adding_column_to_min-nonTU}, $N_{[1]}$ is not totally equimodular, a contradiction since $N_{[1]}$ is in the complement orbit of $N$.
Thus, $B$ is a te-lace.	
\end{proof}
Therefore, $\eqdet (A) = \pm2^{n-1}\eqdet(B) = \pm2^n$, and hence $A$ is a thick te-interlace.
\end{proof}
\subsubsection{Consequences}
Lemma~\ref{lemma:two_types_te-interlaces} and the interplay between te-interlaces and complement matrices have several consequences.
First, combined with Lemma~\ref{lemma:trimming_and_complement_orbit}, this gives the following.
\begin{corollary}\label{corollary:trim_te-interlace}
Let $p$ be a pivot of a te-interlace $A$. Then, we have:
\begin{itemize}
\item $A$ is thin if and only if $\widehat{A\trim p}$ is a tu-set,
\item $A$ is thick if and only if $\widehat{A\trim p}$ is a te-lace.
\end{itemize}
\end{corollary}
In particular, the second point of Corollary~\ref{corollary:trim_te-interlace} combined with Lemma~\ref{lemma:adding_column_to_min-nonTU} yield the following.
\begin{corollary}\label{corollary:form-thick-te-inter}
Let $A'$ be an $n\times n$ invertible submatrix of a thick te-interlace $A$ of size $n$. Then, each nonzero column of $A$ is a column of~$A'$ or its opposite.
\end{corollary}
Corollary~\ref{corollary:trim_te-interlace} can be strengthened as follows, in the case of square matrices.
\begin{corollary}\label{corollary:interlace_core}
Let $A$ be a $\pm1$ square invertible matrix and $B$ a core of $A$. Then, the following holds:
\begin{itemize}
\item[\textbullet] $A$ is a thin te-interlace if and only if $B$ is complement totally unimodular,
\item[\textbullet] $A$ is a thick te-interlace if and only if $B$ is complement minimally non-totally unimodular.
\end{itemize}
\end{corollary}
\begin{proof}
For both cases, the ``only if'' part comes from the proof of Lemma~\ref{lemma:two_types_te-interlaces}. To see the ``if'' part, first note that, up to resigning, the trims of $A$ are in the complement orbit of $B$ by Lemma~\ref{lemma:trimming_and_complement_orbit}.
Since $B$ is either complement totally unimodular or complement minimally non-totally unimodular, all these trims are either totally unimodular or minimally non-totally unimodular, hence are all totally equimodular.
Therefore, so is $A$ by Theorem~\ref{theorem:pivot_TE}.
Since $A$ is $\pm1$, it is a te-interlace, and a determinant computation concludes each case.
\end{proof}
In particular, since proper minors of complement minimally non-totally unimodular matrices are complement totally unimodular matrices, we have the following.
\begin{remark}\label{remark:eqdet_te-interlaces}
Let $A$ be a te-interlace.
If $A$ is thin, then $\eqdet(B)=2^{|B|-1}$, for every $B\subseteq A$.
If $A$ is thick, then $\eqdet(B)=2^{|B|-1}$, for every $B\subsetneq A$, and $\eqdet(A)=2^{|A|}$.
\end{remark}

Moreover, we also obtain the following.

\begin{corollary}\label{corollary:transpose_TE}
For a square invertible $\pm1$ matrix $A$, the following statements are equivalent:
\begin{itemize}
\item $A$ is totally equimodular,
\item $A^\top$ is totally equimodular,
\item $A^{-1}$ is totally equimodular,
\item $\left(A^{-1}\right)^\top$ is totally equimodular.
\end{itemize}
\end{corollary}
\begin{proof}
Square invertible totally equimodular $\pm1$ matrices are te-interlaces.
By Lemma~\ref{lemma:two_types_te-interlaces}, if $A$ is totally equimodular then it is either a thick or a thin te-interlace.
In both cases, by Remark~\ref{remark:eqdet_te-interlaces} the $k\times k$ determinants of $A$ are either $0$ or $\pm 2^{k-1}$, for $k<n$. 
Hence, $A$ is totally equimodular if and only $A^\top$ is, 	and Theorem~\ref{theorem:invert_TE} concludes.
\end{proof}

Using the comatrix and determinant computations yield the following.
\begin{corollary}
Let $A$ be a square $\pm1$ matrix, then
\begin{itemize}
\item $A$ is a thin te-interlace if and only if $\left(\frac{1}{2}A\right)^{-1}$ is a te-lace,
\item $A$ is a thick te-interlace if and only if $\left(\frac{1}{4}A\right)^{-1}$ is a thick te-interlace.
\end{itemize}
\end{corollary}
The first statement means that, up to rescaling, the inverse of minimally non-totally unimodular matrices are the thin te-interlaces.
%

\subsection{Proof of Lemma~\ref{keylemma}}

A crucial result to decompose te-sets is the following, which can be restated as: if two distinct te-laces of a te-set intersect, then their union is a subset of a te-interlace.

In brief, the proof is constructed by considering a counterexample with a minimum number of rows. 
The impact of well-chosen trims is then examined. 
The analysis yields a set of properties for such a counterexample, in particular on how the trims modify the supports of the rows of the matrix, ultimately leading to a contradiction of the linear independence.

\begin{lemma}[Lemma~\ref{keylemma}]\label{te-knotinter}
In a te-set, if two distinct te-laces intersect, then they are both of size two and their symmetric difference is also a te-lace.
\end{lemma}

\begin{proof}
Let $M$ be a counterexample having a minimum number of rows, that is, $M$ is a te-set, such that $M$ is the union of two te-laces $A$ and $B$ with $r\in A\cap B$, both $B\setminus A$ and $A\setminus B$ nonempty, and without loss of generality $|A|\geq 3$.

\begin{claimnumb}\label{cl1}
We may assume that $A\cap B=\{r\}$. 
\end{claimnumb}
\begin{proof}
If $\supp(a) = \supp(b)$ for some $a\in A\setminus B$ and $b\in B\setminus A$, then $\{a,b\}$ is a te-lace and $A\cup \{a,b\}$ satisfies the claim.
Thus, assume that $\supp(a) \neq \supp(b)$ for all $a\in A\setminus B$ and $b\in B\setminus A$.

By contradiction, assume $|A\cap B|\geq 2$, and let $s\in A\cap B\setminus\{r\}$.
Let $p=(s,j)$ be a pivot of $M$, and $A'=A/\!\!/p$ and $B'=B/\!\!/p$.
Since $s\in A\cap B$ and pivoting preserves te-laces by Lemma~\ref{pivotte-knot}, both $A'$ and $B'$ are te-laces.
By $p$-pivoting $M$, the only row that disappeared is $s$ and the row obtained from $r$ belongs to both $A'$ and $B'$, hence $A'$ and $B'$ intersect.
Moreover, $A'\cup B'$ is a te-set by Theorem~\ref{pivotpreserveTE}.

Therefore, since $M$ is a minimum counterexample, we have $|A'|=|B'|=2$, and thus $|A|=|B|=3$ and $A\cap B =\{r,s\}$.
More precisely, let $A=\{a,r,s\}$, $B=\{b,r,s\}$, $A'=\{a',r'\}$, and $B'=\{b',r'\}$, where $x'$ is obtained from $x$ by the $p$-pivot for $x=a,r,b$.
Note that $a$, $b$, and $r$ all have a different support: $\supp(a)\neq \supp(r)$ and $\supp(b)\neq \supp(r)$ because $A$ and $B$ are both te-laces of size three; and $\supp(a)\neq \supp(b)$.
Moreover, for all $x=a,b,r$, since $\{x,s\}$ is a tu-set, we have either $\supp(x')=\supp(x)$ or $\supp(x')=\supp(x)\triangle \supp(s)$.

Now, since $A'$ and $B'$ are te-laces of size two, we have $\supp(a')=\supp(r')=\supp(b')$.
For that to happen, the $p$-pivot changed exactly: one row among $a$ and $r$, one among $b$ and $r$, and one among $a$ and $b$.
This is impossible.
\end{proof}

\begin{claimnumb}\label{cl2}
$|B|=2$.
\end{claimnumb}
\begin{proof}
By contradiction, assume that $|B|\geq 3$.
We first prove that: 
$$(\star) \quad A_b=A\setminus\{r\}\cup\{b\} \text{ is a tu-set for all $b\in B\setminus A$}.$$
Otherwise, there exists $b\in B\setminus A$ such that $A_b$ contains a te-lace $X$, by Lemma~\ref{containte-lace} and because every $A_b$ is a te-set.
If $X= A_b$, then $|A\cup X|=|A|+1<|A\cup B|$, since $|B|\geq 3$. Thus, $A\cup X$ is a smaller counterexample, because $|A|\geq 3$.
Therefore, $X\subsetneq A_b$, but then $B\cup X$ is a smaller counterexample because $|B|\geq 3$.
Hence, $(\star)$ is proved.
\medskip

Now, let $b\in B\setminus\{r\}$ and $p=(b,j)$ be a pivot of $M$. 
Let $r'$ denote the row obtained from $r$ by $p$-pivoting.
By Lemma~\ref{pivotte-knot}, $B/\!\!/p$ is a te-lace.
Since $A\cup\{b\}$ is a te-set and not a tu-set, and since trimming preserves te-sets and tu-sets, $A/\!\!/p$ contains a te-lace $X'=X/\!\!/p$ for some $X\subseteq A\cup\{b\}$, with $b\in X$.
Then, $X'$ contains $r'$ because $A/\!\!/p=(A_b/\!\!/p)\cup\{r'\}$, and because $A_b/\!\!/p$ is a tu-set, by Theorem~\ref{pivotTU} and by the fact that $A_b$ is a tu-set by $(\star)$.
In particular, $r\in X$.

Note that $X'$ and $B/\!\!/p$ are te-laces and intersect as both contain $r'$.
Moreover, we have $X'=X/\!\!/p\subseteq A/\!\!/p$.
Therefore, if $|X'|=|A|\geq 3$, then $X'$ and $B/\!\!/p$ form a smaller counterexample, a contradiction.
Otherwise, we have $|X'|<|A|$ and hence $|X|\leq|A|$.
Suppose $|X|=|A|$, since $b\in X$, we have $X\neq A$, so $X=A\setminus\{a\}\cup\{b\}$, for some $a\in A$. Moreover, $a\neq r$, as otherwise, $X=A_b$ is a tu-set, which is impossible since $X'$ is a te-lace.
Then, $X$ contains a te-lace $Y$ that contains $r$ and $b$ as $X\setminus\{r\}$ and $X\setminus\{b\}$ are tu-sets.
Hence $Y\setminus A = \{r,b\}$ and $A\setminus Y \neq \emptyset$, and $Y\cap A \neq \emptyset$, so $|Y| = 2+|Y\cap A|\geq 3$.
Since $|Y\cap B|=1$ and $|B|\geq 3$, there exists an element of $B$ neither in $A$ nor in $Y$. Thus, $A$ and $Y$ form a smaller counterexample.
Finally, we have $|X|<|A|$, and since $X$ is not a tu-set, it contains a te-lace $Y$, that must contain both $b$ and $r$ and another element of $A$, as otherwise it is a proper subset of $B$ which is a tu-set, so $|Y|\geq 3$.
We have $|Y\cup B|\leq |X\cup B|< |A\cup B|$, hence $Y$ and $B$ are two intersecting te-laces with $|B|\geq 3$, hence form a smaller counterexample, a contradiction.
\end{proof}

\begin{claimnumb}\label{cl3}
$|A|=3$.
\end{claimnumb}
\begin{proof}
Let $B=\{r,b\}$.
Take a pivot $p=(a,j)$ of $M$ with $a\in A\setminus B$, and let $A'=A/\!\!/p$ and $B'=B/\!\!/p=\{r',b'\}$. 

By Lemma~\ref{pivotte-knot}, $A'$ is a te-lace.
Let us prove that $B'$ is a te-lace.
Since $B$ is a te-lace of size two, we have $\supp(r)=\supp(b)$.
Hence, the $p$-pivot modified either both $r$ and $b$, or none of them. 
In the latter case, $B'=B$ is a te-lace.
In the former case, $x'=x\triangle (\pm a)$, for $x=r,b$, thus $\supp(r')=\supp(b')$ and $B'$ is a te-lace.

Since $M$ is a minimum counterexample and $M/\!\!/p$ is a smaller te-set composed of two intersecting te-lace $A'$ and $B'$, we get $|A'|=|B'|=2$.
Thus, $|A|=3$.
\end{proof}

From now on, let $A=\{a,s,r\}$ and $B=\{r,b\}$.
We will analyze the effect of $(s,j)$-pivoting $M$ for some pivot $p=(s,j)$.
Since $A$ is a te-lace of size three, $\supp(a)$ and $\supp(s)$ intersect, and we may assume that the $p$-pivot is performed with $j\in\supp(a)\cap\supp(s)$.
For $x=a,b,r$, let $x'$ denote the row obtained after $p$-pivoting $M$.
The choice of the pivot implies $\supp(a')=\supp(a)\triangle \supp(s)$ because $\{a,s\}$ is a tu-set.

Since $A$ and $B$ are te-laces respectively of size three and two, $a$, $s$, and $r$ have pairwise different supports and $\supp(r)=\supp(b)$.
In particular, as $M$ is a te-set, $\{a,b\}$ and $\{s,b\}$ are tu-sets.

\begin{claimnumb}\label{cl4}
$\supp(r)=\supp(b)=\supp(a)\triangle\supp(s)$.
\end{claimnumb}
\begin{proof}
We show that $\{a',b',r'\}$ is a te-interlace of size three.
By Lemma~\ref{pivotte-knot}, $\{a',r'\}$ is a te-lace because $A$ is.
In particular, $\supp(a')=\supp(r')$.
As $\{s,r\}$ is a tu-set, either $r'=r$ or $\supp(r')=\supp(r)\triangle \supp(s)$.
But since $\supp(a)\neq \supp(r)$ and $\supp(a')=\supp(a)\triangle \supp(s)$, we have $r'=r$.
Since $\supp(r)=\supp(b)$, we also have $b'=b$.
Thus $a'$, $b'$, and $r'$ all have the same support, and $\{a',b',r'\}$ is a te-interlace of size three.

Therefore, back to $M$, we have $\supp(r)=\supp(b)=\supp(a)\triangle\supp(s)$. 
\end{proof}

Since $A$ is a te-lace and has size three, $a$, $s$, and $r$ have pairwise intersecting support and, by Claim~\ref{cl4}, we have $\supp(a)\cup \supp(s)=\supp(a)\cup \supp(r)=\supp(r)\cup \supp(s)$.
Claim~\ref{cl4} also implies that $\supp(a)\cap \supp(s)\cap \supp(r)$ is empty.
In particular, $\supp(a)$ is the disjoint union of $\supp(a)\cap\supp(s)$ and $\supp(a)\cap\supp(r)$, and similarly for $s$ and $r$.

\begin{claimnumb}\label{cl5}
$a$, $s$, and $r$ pairwise coincide on their respective common support (up to resigning rows).
\end{claimnumb}
\begin{proof}
Since $\{a,s\}$ is a tu-set, $a$ and $s$ coincide on $\supp(a)\cap\supp(p)$, up to replacing $s$ by $-s$.
Since $\{a,r\}$ is a tu-set, then $a$ and $r$ either coincide or are opposite on their common support.
The same holds for $s$ and $r$.

As we may multiply rows by $-1$, it is enough to prove that either both coincide, or both are opposite, because $\supp(a)\cap \supp(s)\cap \supp(r)$ is empty.
By contradiction, suppose without loss of generality that $a$ and $r$ are opposite on $\supp(a)\cap \supp(r)$ and $s$ and $r$ coincide on $\supp(s)\cap \supp(r)$.
Then, we have $a-s+r=0$, a contradiction to the full row rank hypothesis.
\end{proof}

Since $\supp(b)=\supp(r)$, $\supp(b)$ is the disjoint union of $\supp(b)\cap\supp(a)$ and $\supp(b)\cap\supp(s)$.

\begin{claimnumb}\label{cl6}
$b$ coincides with $a$ and with $-s$ on their respective common support (up to multiplying $b$ by $-1$).
\end{claimnumb}
\begin{proof}
Since $\{a,b\}$ is a tu-set, $a$ and $b$ either coincide or are opposite on their common support.
The same holds for $b$ and $s$.
Since $\supp(b)=\supp(r)$ and $b\neq r$, the previous claim implies the result.
\end{proof}

Claims~\ref{cl5} and~\ref{cl6} imply that $b = s-a$, contradicting the full row rank of~$M$.
This proves that there is no minimal counterexample and the first part of the lemma follows.
\medskip

Therefore, $A=\{a,r\}$ and $B=\{b,r\}$ where $a$, $b$, and $r$ are linearly independent.
Since $A$ and $B$ are te-laces, we have $\supp(a)=\supp(b)=\supp(r)$, and hence $A\triangle B = \{a,b\}$ is also a te-lace.
\end{proof}

\subsection{Proof of the decomposition theorem}
Thank to the results of the previous sections, we are now in position to prove the decomposition theorem.

Recall that disjoint te-bricks are mutually-tu when the sets which intersect several of them while containing none of the te-laces and at most one vector of each te-interlace are tu-sets.

\begin{theorem}[{Theorem~\ref{theorem:decomposition_te}}]\label{theorem:decomposition_te_full}
A linearly independent set of $0,\!\pm1$ vectors $A$ is a te-set if and only if it is the disjoint union of mutually-tu te-bricks. More precisely:
$$A = \underbrace{U_{}}_{\text{tu-set}} \sqcup~\underbrace{L_1\sqcup \dots \sqcup L_k}_{\text{te-laces}}~\sqcup \underbrace{S_1\sqcup \dots\sqcup S_\ell}_{\text{thin te-interlaces}}\sqcup \underbrace{T_1\sqcup \dots\sqcup T_m}_{\text{thick te-interlaces}}\!\!.$$
\end{theorem}
\begin{proof}
Let us denote the property on the right by~$(\ast)$.
Note that the property of being mutually-tu implies that the decomposition is unique, up to considering te-laces of size two either as te-interlaces or as te-laces.
In this proof, we will make no distinction between thin and thick te-interlaces. In fact, once we obtain a decomposition of the form 
$$A = \underbrace{U_{}}_{\text{tu-set}} \sqcup~\underbrace{L_1\sqcup \dots \sqcup L_k}_{\text{te-laces}}~\sqcup \underbrace{I_1\sqcup \dots\sqcup I_s}_{\text{te-interlaces}}\!\!,$$
we use Lemma~\ref{lemma:two_types_te-interlaces} to obtain that the $I_i$'s are either thin or thick.
\medskip

The fact that te-sets satisfy~$(\ast)$ comes from Lemma~\ref{te-knotinter}.
Indeed, let $A$ be a te-set and $\cal L$ its family of te-laces. 
By Lemma~\ref{containte-lace}, $U = A\setminus (\bigcup_{L\in\cal L} L)$ is a tu-set.
By Lemma~\ref{te-knotinter}, intersecting te-laces are part of a te-interlace.
Thus, the intersecting sets of $\cal L$ can be regrouped into disjoint maximal te-interlaces.
Let $\cal I$ denote this family of te-interlaces.
The remaining te-laces $\cal L'$ of $\cal L$ are disjoint.
Note that $\{U\}\cup \cal L'\cup \cal I$ is a partition of $A$.
By construction, no te-lace intersects distinct sets among, $U$, $L\in\cal L'$, and $I\in \cal I$, thus these sets are mutually-tu, and $A$ satisfies~$(\ast)$.

\medskip

To see the converse, we will need the following claims, where $A$ is a full row rank matrix satisfying~$(\ast)$, $B\subsetneq A$ denotes a subset of rows of $A$, and $\ell$ is a row of $A$ which is not in $B$.

First remark that if $\ell\in A$ has $|\supp(\ell)|=1$, then $A\setminus\{\ell\}$ is a te-set if and only if $A$ is a te-set.

\begin{claimnumb}\label{cl1.0}
If $B\subseteq A$, then $B$ satisfies~$(\ast)$. 
\end{claimnumb}
\begin{proof}
Denote by $\cal U\cup \cal L \cup \cal I$ the decomposition of the rows of $A$ given by~$(\ast)$, where $\cal U=\{U\}$ and $U$ is a tu-set, every $L\in \cal L$ is a te-lace, and every $I\in\cal I$ is a maximal te-interlace.
Let $\mathcal{L}_B= \{L: \text{ for $L\in\cal L$ with $L\subseteq B$}\}$, $\mathcal{I}_B= \{I\cap B: \text{ for $I\in\cal I$ with $|I\cap B|\geq 2$}\}$, and $\cal U_B=\{U_B\}$, where $U_B$ is the union of $U\cap B$, $L\cap B$ for $L\in\mathcal{L}$ with $L\setminus B \neq L$, and $I\cap B$ for $I\in\mathcal{I}$ with $|I\cap B| = 1$.

Note that $\cal U_B\cup\mathcal{L}_B\cup \mathcal{I}_B$ is a partition of $B$.
By construction, $U_B$ is a tu-set since the te-bricks of $A$ are mutually-tu. 
Each set of $\mathcal{L}_B$ is a te-lace, and each set of $\mathcal{I}_B$ is a te-interlace.
Now, the sets of this partition are mutually-tu because $B\subseteq A$ and $A$ satisfies~$(\ast)$.

Therefore, $B$ satisfies~$(\ast)$.
\end{proof}

\begin{claimnumb}\label{cl1.1}
If $B$ is a te-lace contained in no te-interlace and $\ell\in A\setminus B$, then $(B\cup \{\ell\})/\!\!/(\ell,j)$ is a te-lace for all $j\in\supp(\ell)$.
\end{claimnumb}
\begin{proof}
Let $p=(\ell,j)$ be a pivot of $A$.

We first prove that $B\cup \{\ell\}$ is a te-set by induction on its number of rows.
Since the te-bricks of $A$ are mutually-tu, $B\setminus\{b\}\cup\{\ell\}$ is a tu-set for all $b\in B$. Moreover, $B$ is a te-lace.
Therefore, all there is to show is that $B\cup \{\ell\}$ is equimodular.

If $B=\{b_1,b_2\}$ has size two, then the result holds since $\supp(b_1)=\supp(b_2)$ and both $\{\ell,b_1\}$ and $\{\ell,b_2\}$ are tu-sets by $(\ast)$.
Thus, assume $|B|\geq 3$ and let $q=(b,j')$ be a pivot of $B\cup\{\ell\}$ for some $b\in B$.
By Lemma~\ref{pivotte-knot} and since trimming preserves tu-sets, $(B\cup\{\ell\})/\!\!/q$ is the disjoint union of the te-lace $B\trim q$ and the row obtained from $\ell$ by $q$-trimming, and they are mutually-tu.
In particular, the associated matrix is equimodular by induction.
Since this holds for all $j'\in\supp(b)$, $B\cup\{\ell\}$ is equimodular by Lemma~\ref{pivoteqE}.

Therefore, $B\cup \{\ell\}$ is a te-set, and hence so is $B'=(B\cup \{\ell\})/\!\!/p$ by Theorem~\ref{pivotpreserveTE}.
Moreover, $B'$ is not a tu-set as $B\cup \{\ell\}$ is not a tu-set.
The fact that $B'$ is a te-lace then follows, because all its proper subsets are tu.
Indeed, $B\setminus\{b\}\cup\{\ell\}$ is a tu-set for all $b\in B$, and tu-sets are preserved by trimming.
Hence, if $b'$ denotes the row obtained from $b$ by $p$-trimming, $B'\setminus\{b'\}$ is a tu-set for all $b'\in B'$.
\end{proof}

\begin{claimnumb}\label{cl1.2}
If $B$ is a maximal te-interlace of $A$ and $\ell\in A\setminus B$, then $B'=(B\cup \{\ell\})\trim(\ell,j)$ is a maximal te-interlace of $A/\!\!/(\ell,j)$ for all $j\in\supp(\ell)$.
Moreover, $\eqdet(B')=\eqdet(B)$.
\end{claimnumb}
\begin{proof}
Let $p=(\ell,j)$ be a pivot of $A$, $B'=(B\cup \{\ell\})\trim p$, and $A'=A\trim p$.
We may assume that the $p$-pivot modifies $B$, as otherwise $B'$ is obtained from $B$ by removing a column of zeros and the result holds.
Recall that all the rows of $B$ have the same support, which we denote by $\supp(B)$.
Since each pair of rows of $B$ is a te-lace, they remain te-laces after $p$-trimming by Claim~\ref{cl1.1}, thus $B'$ is not a tu-set, and each pair of rows of $B'$ forms a te-lace.

Let us prove that $B'$ is a te-set\footnote{We mention a subtlety here: in any full row rank $\pm1$~matrix, each pair of rows forms a te-lace. However, there are $\pm1$~matrices which are not te-sets, hence not te-interlaces.} and $\eqdet(B')=\eqdet(B)$. Since $\{\ell,b\}$ is a tu-set, $\ell$ and $b$ coincide on their common support (up to multiplying $b$ by $-1$), for all $b\in B$.
Therefore $B'$ is composed of $B$ on $\supp(B)\setminus\supp(\ell)$, of $-\ell$ on $\supp(\ell)\setminus\supp(B)$, and of zeros elsewhere.
The new nonzero columns indexed at $\supp(\ell)\setminus\supp(B)$ are either equal or opposite to the column of $B$ on which the $p$-pivot occured.
Since $B$ is a te-interlace, it is a te-set, and hence $B'$ is a te-set.
This also shows that all the nonzero maximal determinants of $B$ and $B'$ have the same absolute value.

Therefore, $B'$ is a te-interlace, and all that remains to prove is that it is a maximal te-interlace of $A'$, that is, no row of $A'\setminus B'$ has $\supp(B')$ as support.
By contradiction, suppose not and let $s'$ be a row of $A'\setminus B'$ with $\supp(s')=\supp(r')$ for some $r'\in B'$, let $s$ be the corresponding row of $A\setminus B$.
Lemma~\ref{3rows} applies as $\{\ell,r\}$ is a tu-set, hence either $\{r,s\}$ os $\{\ell,r,s\}$ is a te-lace.
The first possibility is impossible because $B$ is maximal and $s\notin B$, thus $\{r,s\}$ is a tu-set.
Neither is the other possibility, as it contradicts the fact that the te-bricks of $A$ are mutually-tu.
\end{proof}


Now, we proceed by induction on the cardinality of the set, that is, on the number of rows of the associated matrix.
The case of a singleton is immediate.
Assume that $A$ satisfies~$(\ast)$, and that all full row rank sets with fewer elements that satisfy~$(\ast)$ are te-sets.
Then, all there is to prove is that $A$ is equimodular.
We will trim $A$ while preserving~$(\ast)$, hence find equimodular matrices by the induction hypothesis.
Then, we will retrieve the equimodularity of $A$ using Lemma~\ref{pivoteqE}.

\medskip

Denote by $\{U\}\cup \cal L \cup \cal I$ the decomposition of the rows of $A$ given by~$(\ast)$, where $U$ is a tu-set, each $L\in\cal L$ is a te-lace, and each $I\in\cal I$ is a te-interlace, and these sets are disjoint and mutually-tu.
We assume that $|L|\geq 3$ for all $L\in \cal L$ by considering possible te-laces of size $2$ of $\cal L$ as te-interlaces.
Note that each of $U$, $\cal L$, and $\cal I$ can be empty.
Moreover, the full row rank of $A$ together with~$(\ast)$ imply that each $I\in\cal I$ is a maximal te-interlace of $A$.
In particular, for each $I\in \cal I$ and $\ell\in A\setminus I$, the support of $\ell$ differs from that of the rows of $I$.

\medskip

The proof is divided in two cases.

\medskip
\paragraph{\bf Case 1.} $A$ contains a row $\ell$ which is in no te-interlace.

\medskip

This case is handled in the following claim.

\begin{claimnumb}\label{cl1.3}
Then, $A\trim p$ satisfies~$(\ast)$ for all pivots $p=(\ell,j)$ of $A$.
\end{claimnumb}
\begin{proof}

Let $p=(\ell,j)$ be a pivot of $A$.
Recall that $\ell$ belongs either to $U$, or to some te-lace $L_\ell$ of $\cal L$. 
We treat both cases simultaneously.
The two previous claims imply that the te-lace and te-interlaces of $A$ and $A'=A\trim p$ are the same sets of rows, except possibly for $L_\ell$ if it exists.

When $\ell\in L_\ell$ of $\cal L$, we have $|L_\ell|\geq 3$ since $\ell$ is in no te-interlace, hence $L_\ell'=L_\ell\trim p$ is a te-lace of $A'$ by Lemma~\ref{pivotte-knot}.
Moreover, either $\ell\in U$, or $U\cup \{\ell\}$ is a tu-set by the mutually-tu property in~$(\ast)$.
In both cases, $U'=(U\cup \{\ell\})\trim p$ is a tu-set by Theorem~\ref{pivotTU}.
Let $\mathcal{L}'=\{L_\ell'\}\cup\left\{(L\cup \{\ell\})\trim p, \text{ for all $L\in \cal L$}\right\}$ and $\mathcal{I}'=\{(I\cup \{\ell\})\trim p, \text{ for all $I\in \cal I$}\}$, then $\{U'\}\cup \cal L' \cup \cal I'$ is a partition of $A'$.

To prove that $A'$ satisfies~$(\ast)$, there remains to prove that these sets are mutually-tu. 
Suppose by contradiction that a non-tu set $X'\subset A'$ intersects several sets in $\{U'\}\cup \cal L' \cup \cal I'$, while $X'$ contains no te-lace of $\cal L'$ and shares at most one row with each te-interlace of $\cal I'$.
Let $X\subset A$ be such that $X'=X\trim p$, with $\ell\in X$. 
Since $\ell$ is the only row index that differs between $X'$ and $X$, the same holds for $X\setminus\{\ell\}$ with respect to $\cal L$ and $\cal I$, hence $X\setminus\{\ell\}$ is a tu-set by~$(\ast)$ in $A$.
If $X=A$, then necessarily $X=A = U\cup L_\ell$.
Otherwise, by the induction hypothesis, $X$ is a te-set.
Since $X$ is not a tu-set as trimming preserves tu-sets, it contains a te-lace by Lemma~\ref{containte-lace}.
This te-lace contains $\ell$ since $X\setminus\{\ell\}$ is a tu-set.
If $\ell\in U$, $\ell$ belongs to no te-lace by the mutually-tu property, hence $\ell\notin U$.
Therefore, $\ell\in\ L_\ell$ and $L_\ell\subseteq X$.
But then, $L_\ell'\subseteq X'$, which contradicts the choice of $X'$.
\end{proof}

Since $A\trim (\ell,j)$ has full row rank, it is equimodular by the above claim and the induction hypothesis.
Note that this holds for all pivot $(\ell,j)$ of $A$, that is, for all $j\in\supp(\ell)$.
Therefore, Lemma~\ref{pivoteqE} implies that $A$ is equimodular.
\medskip

\paragraph{\bf Case 2.} Every row of $A$ is in a te-interlace.

\medskip

In this case, by~$(\ast)$, $A$ is the disjoint union of te-interlaces $I_1,\dots,I_t$ so that every set intersecting each $I_i$ in at most one row is a tu-set.
We may assume that $t\geq 2$, as otherwise we are done since a te-interlace is a te-set by definition.

Let $\ell$ be a row of $I_1$.
Let $p=(\ell,j)$ be a pivot of $A$ and let $I_i'$ be the rescaled matrix obtained from $I_i$ by $p$-trimming $A$, for $i=1,\dots,t$ and let $A'=\widehat{A\trim p}$.
By Claim~\ref{cl1.2}, $I_2',\ldots,I_t'$ are maximal te-interlaces of $A'$. 
Moreover, $I_2',\ldots,I_t'$ satisfy~$(\ast)$.
Indeed, their union is equal to $(A\setminus (I_1\setminus\{\ell\}))\trim p$, and $A\setminus (I_1\setminus\{\ell\})$ is a te-set by $|I_1|\geq 2$, the induction hypothesis, and Claim~\ref{cl1.0}.


%

Since $I_1$ is a te-interlace, by Corollary~\ref{corollary:trim_te-interlace}, $I_1'=\widehat{I_1\trim p}$ is either a tu-set or a te-lace.

\begin{claimnumb}\label{cl1.5}
The sets $I_1',I_2',\dots,I_t'$ are mutually-tu.
\end{claimnumb}
\begin{proof}
By contradiction, suppose that a minimal non-tu set $X'$ intersects several of these sets, with $I_1'\not\subseteq X'$ when $I_1'$ is a te-lace, and with $|X'\cap I_i'|\leq 1$, for $i\geq 2$, and let $X'=\widehat{X\trim p}$.
Since trimming preserves tu-sets, $X$ is not a tu-set.
Since $|X'\cap I_i'|\leq 1$ for all $i\geq 2$, we have $|X\cap I_i|\leq 1$ for all $i\geq 2$.
By Claim~\ref{cl1.0}, $A\setminus I_1$ satisfies~$(\ast)$, hence it is a te-set by the induction hypothesis.
Therefore $X\setminus I_1$ is a te-set, and hence in fact a tu-set, as otherwise it would contain a te-lace by Lemma~\ref{containte-lace} and contradict~$(\ast)$ in $A\setminus I_1$.

Since $I_1'$ is either a te-lace or a tu-set, it satisfies~$(\ast)$, hence $X'\setminus I_1'\neq \emptyset$.
Moreover, $X'\cap I_1'\neq \emptyset$ because $I_2',\dots,I_t'$ satisfy~$(\ast)$.
In particular, $X'\cap I_1'$ and $X'\setminus I_1'$ are both tu-sets.

To sum up, $X$ is the disjoint union of a te-interlace $I_X=X\cap I_1$ and a tu-set  $U_X=X\setminus I_1$, and moreover $\widehat{I_X\trim p}$ is a tu-set.
The latter implies that $\eqdet (I_X) = 2^{|I_X|-1}$.
By Claim~\ref{cl1.2} and the fact that trimming preserves tu-sets, successively $(r,j')$-trimming $X$ with respects to rows $r$ of $U_X$ maintains~$(\ast)$, and hence yields $\eqdet (X) = \eqdet (I_X) = 2^{|I_X|-1}$.

If $A$ is a square submatrix of $X$ and $A'$ the corresponding one in $X'$, the definition of $X$ gives the following, since the scaled rows to get $X'$ are precisely the rows of $I_X\setminus\{\ell\}$: $\det (A) = 2^{|I_X|-1}\det(A')$.
Thus $\det(A')=0,\!\pm1$ for all square submatrices $A'$ of $X'$.
Since $X'$ is a minimal nontu-set, all its proper subsets are tu-sets.
But then, $X'$ is also a tu-set, and this contradiction finishes the proof.
\end{proof}

Claim~\ref{cl1.5} implies that $A\trim p$ satisfies~$(\ast)$, and hence is equimodular by the induction hypothesis.
This holds for all $p=(\ell,j)$ with $j\in\supp(\ell)$, therefore $A$ is equimodular by Lemma~\ref{pivoteqE}.
\end{proof}

\section{Proofs of the results of Section~\ref{section:main_cones}}\label{section:proofs_cones}

\subsection{Proofs of the results of Section~\ref{section:hb_te-bricks}{: Hilbert bases}}

Recall that all the cones are pointed.
For a set $A$ of linearly independent integer vectors, we write $\gcddet(A)$ for the gcd of the determinants of size $|A|\times|A|$ in the matrix associated with $A$.
We set $\gcddet(\emptyset) = 1$.
\medskip

The following follows from the definition of Hilbert basis elements.
\begin{lemma}\label{lemma:face_Hilbert}
Hilbert basis elements of faces of a cone are Hilbert basis elements of the cone.
\end{lemma}
The \emph{half-open zonotope} generated by a finite set of vectors $A$ is:
$$\cal{Z}^<(A) = \left\{\sum_{a\in A} \lambda_a a \colon 0\leq\lambda_a< 1\right\},$$
and contains the Hilbert basis of the cone generated by $A$.
\begin{lemma}[\cite{Sebo_1990}]\label{lemma:nontrivial_Hilb_zono} 
For a cone $C=\cone (A)$, we have $\cal H(C)\subseteq\cal{Z}^<(A)\cap \Z^n$.
\end{lemma}

A common result in group theory gives the following.
\begin{theorem}[Folklore] \label{theorem:nbr_points_zono}
Let $A$ be a set of linearly independent vectors of $\Z^n$.
The number of integer points in the half-open zonotope $\cal Z^<(A)$ is equal to $\gcddet(A)$.
\end{theorem}
\begin{remark}
Notice that if $A$ is unimodular or a tu-set, the number of integer points in $\cal Z^<(A)$ is $1$,
and when $A$ is a te-set, we have $\gcddet(A_I)=\eqdet(A_I)$, for every $I\subseteq\{1,\ldots,k\}$.
\end{remark}

Any two disjoint sets of integer vectors $A$ and $B$ are \emph{lattice orthogonal} when $A\sqcup B$ is linearly independent and $\gcddet(A\sqcup B) = \gcddet(A)\gcddet(B)$.
In that case, there is a nice relation between the integer points of the associated half-open zonotopes.
\begin{corollary}\label{corollary:mutually_tu}
Let $A$ and $B$ be lattice orthogonal sets of vectors of $\Z^n$, then we have
$$|\cal Z^<(A\sqcup B)\cap \Z^n| =|\cal Z^<(A)\cap \Z^n|\cdot|\cal Z^<(B)\cap \Z^n|.$$
\end{corollary}
\begin{proof}
By applying three times Theorem~\ref{theorem:nbr_points_zono}, and by the lattice orthogonality, we have: 
\begin{align*}
|\cal Z^<(A\sqcup B)\cap \Z^n| &= \gcddet(A\sqcup B)\\
&=  \gcddet(A)\gcddet(B) \\
&= |\cal Z^<(A)\cap \Z^n|\cdot|\cal Z^<(B)\cap \Z^n|,
\end{align*} as desired.
\end{proof}

\begin{lemma}\label{lemma:hb_lattice_orthogonal}
The Hilbert basis of the Minkowski sum of cones generated by lattice orthogonal sets is the union of the Hilbert basis of each.
\end{lemma}
\begin{proof}
Let $C = \cone(A_1\sqcup \cdots\sqcup A_k) = \cone(A_1)+\cdots+\cone(A_k) = C_1+\cdots+C_k$ be the Minkowski sum of cones generated by lattice orthogonal sets $A_1,\ldots,A_k$ of $\mathbb{R}^n$.
Since $A_1\sqcup \cdots\sqcup A_k$ is linearly independent, each $C_i$ is a face of $C$.
By Lemma~\ref{lemma:face_Hilbert}, we have $\bigcup_i \cal H(C_i)\subseteq\cal H(C)$. An immediate induction relying on Corollary~\ref{corollary:mutually_tu} yields $|\cal Z^<(\bigcup_i A_i)\cap \Z^n|=\prod_i|\cal Z^<(A_i)\cap \Z^n|$.
Therefore, any integer point in $\cal Z^<(\bigcup_i A_i)$ can be expressed as a nonnegative integer combination of integer points in each $\cal Z^<(A_i)$. 
This yields $\cal H(C)\subseteq \bigcup_i \cal H(C_i)$, which concludes.
\end{proof}
The formula for the determinant of a te-set given after the proof of Theorem~\ref{theorem:decomposition_te} implies the following, which also proves Lemma~\ref{lemma:hb_minkowski_sum}.	
\begin{remark}\label{remark:mutually-tu}
Mututally-tu te-bricks are lattice orthogonal.
\end{remark}
\medskip

The following lemma will be very useful in the proof of Theorem~\ref{theorem:hb_te-bricks}.			
\begin{lemma}\label{lemma:finding_hilbert_basis}
Let $C = \cone (A)$ be a cone and $S\subseteq \cal{Z}^<(A) \cap \Z^n$ be a set in which no element is a nonnegative integral combination of other elements of $S$.
If the set of nonnegative integer combinations of $S$ contains $\cal{Z}^<(A)\cap \Z^n$, then $\cal{H}(C)=S$.
\end{lemma}
\begin{proof}
First, we have $\cal{H}(C)\subseteq S$, as if $h\in\cal{H}(C)\setminus S$, then $h$ belongs to $\cal{Z}^<(A)\cap\Z^n$ by Lemma~\ref{lemma:nontrivial_Hilb_zono}, yet cannot be expressed as an nonnegative integral combination of elements in $S$ since it is a Hilbert basis element, a contradiction.
Then, the definition of $S$ implies $\cal{H}(C)=S$.
\end{proof}

We will also use the following well-known formula in the proof of Theorem~\ref{theorem:hb_te-bricks}.
\begin{lemma}\label{lemma:odd_and_even_sum}
For every positive integer $n$, we have the following:
$$\sum_{i\text{ even}}^{n} \binom{n}{i} = \sum_{i\text{ odd}}^{n}\binom{n}{i} = 2^{n-1}.$$
\end{lemma}
\begin{proof}
This comes from that both quantities equal the half sum of
$$\sum_{i\text{ even}}^{n} \binom{n}{i} + \sum_{i\text{ odd}}^{n}\binom{n}{i} = (1+1)^n= 2^n,$$
and
$$\sum_{i\text{ even}}^{n} \binom{n}{i} -  \sum_{i\text{ odd}}^{n} \binom{n}{i} = \sum_{i\text{ even}}^{n} \binom{n}{i}(-1)^i + \sum_{i\text{ odd}}^{n} \binom{n}{i}(-1)^i = (1-1)^n= 0.$$
\end{proof}

We can now give the proof of Theorem~\ref{theorem:hb_te-bricks}, which we restate below.
\begin{theorem}[Theorem~\ref{theorem:hb_te-bricks}]\label{theorem:hb_te-bricks_full}
Let $A = \{a^1,\ldots,a^n\}$ be a te-brick and $C = \cone (A)$.	
\begin{enumerate}[label=\arabic*., ref=\arabic*.]
\item If $A$ is a tu-set, then $\cal H(C)= A$.\label{equation:hb_tu-set_l}
\item If $A$ is a te-lace, then $\cal H(C)= A\cup\{\frac{1}{2}\sum_j a^j\}$.\label{equation:hb_te-lace_l}
\item If $A$ is a thin te-interlace, then $\cal H(C) = A\cup \left\{\frac{1}{2} (a^i + a^j)\right\}_{1\leq i<j\leq n}$.\label{equation:hb_thin_te-interlace_l}
\item If $A$ is a thick te-interlace, then one of the following holds: \label{equation:hb_thick_te-interlace_l}
\begin{enumerate}[label=\alph*., ref=\theenumi\alph*]
	\item $\cal H(C) = A\cup \left\{\frac{1}{2} (a^i + a^j)\right\}_{1\leq i< j\leq n}\cup\left\{\frac{1}{4}\sum_{j} a^j\right\},$ \label{equation:hb_thick_te-interlaces_cas1_l}
	\item $\cal H(C) = A\cup\left\{\frac{1}{2} (a^i + a^j)\right\}_{1\leq i< j\leq n}\cup\left\{\frac{3}{4}a^i+\frac{1}{4}\sum_{j\neq i} a^j\right\}_{i\in\{1,\ldots,n\}}.$ \label{equation:hb_thick_te-interlaces_cas2_l}
\end{enumerate}
Moreover, the number of $1$'s and $-1$'s in each column of $A$ have the same parity $p\in\{0,1\}$, and~\ref{equation:hb_thick_te-interlaces_cas1} occurs if and only if $n\equiv 2p \pmod{4}$.
\end{enumerate}
\end{theorem}
\begin{proof}~Let $A = \{a^1,\ldots,a^n\}$ be a te-brick and $C = \cone (A)$.
\subsubsection*{\ref{equation:hb_tu-set_l}~tu-sets}
Simplicial cones generated by tu-sets have no nontrivial Hilbert basis elements.
\subsubsection*{\ref{equation:hb_te-lace_l}~te-laces}
The only nontrivial Hilbert basis element of a simplicial cone generated by a te-lace is the half sum of its generators.

\medskip

Now, suppose $A$ is a te-interlace of size $n$.
\begin{claimnumb}\label{claim:0}
$\left\{\frac{1}{2} (a^i + a^j)\right\}_{1\leq i\leq j\leq n} \subseteq \cal H(C)$.
\end{claimnumb}
\begin{proof}
For $1\leq i<j\leq n$, note that $\frac{1}{2} (a^i + a^j)$ is an integer point of $C$ since the $a^i$'s are $\pm 1$ vectors.
Moreover, $\cone \{a^i,a^j\}$ is a face of $C$, and $\cal H\left(\cone \{a^i,a^j\}\right) = \{a^i,a^j,\frac{1}{2}(a^i+a^j)\}$ since it is a cone generated by the te-lace $\{a_i,a_j\}$.
By Lemma~\ref{lemma:face_Hilbert}, each $\frac{1}{2}(a^i+a^j)$ is a Hilbert basis elements of $C$.
\end{proof}

\subsubsection*{\ref{equation:hb_thin_te-interlace_l}~Thin te-interlaces.}
In this case, $A$ has equideterminant $2^{n-1}$.
By Theorem~\ref{theorem:nbr_points_zono}, the number of points in $\cal Z^<(A)$ is $2^{n-1}$.
The result follows from Claim~\ref{claim:0}, Lemma~\ref{lemma:finding_hilbert_basis}, and the following claim.
\begin{claimnumb}\label{claim:2}
The number of distinct integer points of $\cal Z^<(A)$ which are nonnegative integer combinations of $\left\{\frac{1}{2} (a^i + a^j)\right\}_{1\leq i\leq j\leq n}$ is $2^{n-1}$.
\end{claimnumb}
\begin{proof}
Let us denote the set of integer points of $\cal{Z}^<(A)$ which are nonnegative combinations of $\left\{\frac{1}{2} (a^i + a^j)\right\}_{1\leq i\leq j\leq n}$ by $X$.
Since the cone is simplicial, a point $x\in X$ is uniquely expressed as
$$x = \sum_{k=1}^n \lambda_k a^k.$$
Since $x$ is in $\cal Z^<(A)$, we have $\lambda_k\in\{0,\frac{1}{2}\}$, for $k\in\{1,\ldots,n\}$. 
Moreover, since $x$ is integer and the $a^k$'s are $\pm1$, an even number of $\lambda_k$'s are equal to $\frac{1}{2}$.
Thus, the points of $X$ are obtained by setting to $\frac{1}{2}$ an even number $i$ of the $\lambda_k$'s, and by setting the rest to $0$ the $n-i$ remaining $\lambda_k$'s.
Each choice yields a different point since the $a^k$'s are linearly independent. Therefore, by Lemma~\ref{lemma:odd_and_even_sum}, we have		
$$|X| = \sum_{i\text{ even}} \binom{n}{i} = 2^{n-1}.$$
\end{proof}

\subsubsection*{\ref{equation:hb_thick_te-interlace_l}~Thick te-interlaces.}	
By Corollary~\ref{corollary:form-thick-te-inter}, we may assume that $A$ is square and invertible.
In this case, $A$ has determinant $2^{n}$.
By Theorem~\ref{theorem:nbr_points_zono}, the number of points in $\cal Z^<(A)$ is $2^{n}$.
We first compute a candidate set for the Hilbert basis of $\cone(A)$.
We have $\left\{\frac{1}{2} (a^i + a^j)\right\}_{1\leq i\leq j\leq n} \subseteq \cal H(C)$ by Claim~\ref{claim:0}. Moreover, the following holds.
\begin{claimnumb}\label{claim:3.5}
Either $\frac{1}{4}\sum_{j} a^j$ or $\left\{\frac{3}{4}a^i+\frac{1}{4}\sum_{j\neq i} a^j\right\}_{1\leq i\leq n}$ is contained in $\cal{Z}^<(A)\cap~\Z^n$.
\end{claimnumb}
\begin{proof}
By definition of the half-open zonotope, all these points are in $\cal{Z}^<(A)$.
By Lemma~\ref{lemma:trimming_and_complement_orbit} and Corollary~\ref{corollary:interlace_core}, we have:
\begin{equation}\label{equation:formula_R_with_sign}
	A = \diag(\varepsilon) \begin{bmatrix}
		1& \mathbf{1}^\top \\
		\mathbf{1} &  \mathbf{J} -2B
	\end{bmatrix}\diag(\mu),
\end{equation}
for two signing matrices $\diag(\varepsilon)$ and $\diag(\mu)$ and a complement minimally non-totally unimodular $0,\!1$ matrix $B$, of odd size by Theorem~\ref{theorem:Truemper_complement}.
The numbers of $1$'s and $-1$'s have the same parity in every column of $A$, because $\operatorname{diag}(\varepsilon) A \operatorname{diag}(\mu)$ contains an even number of $1$'s and $-1$'s in each column. In the first column, this is due to $B$ being of odd size. In the other columns, it holds because $B$ is $0,\!1$ and minimally non-totally unimodular, and hence contains an even number of $1$'s in each column, by Theorem~\ref{theorem:camion}.
Multiplying by the signing matrices then impacts the parity of the numbers of $1$'s and $-1$'s similarly in every column of $A$.

There are two cases, according to whether this parity is (E) even or (O) odd.

%

\begin{itemize}
	\item[(E)]
	Suppose $n\equiv 0\pmod{4}$.
	Then, since $A$ has an even number of $1$'s and $-1$'s in each column, the coordinates of $\sum_{j}a^j$ are all congruent to $0\pmod{4}$.
	Therefore, $\frac{1}{4}\sum_{j}a^j$ is an integer point of $C$ and is in $\cal{Z}^<(A)$.
	
	\noindent Suppose $n\equiv 2\pmod{4}$.
	Then, for $i\in\{1,\ldots,n\}$, the coordinates of $3a^i +\sum_{j\neq i}a^j$ are all congruent to $0\pmod{4}$.
	Therefore, $\frac{3}{4}a^i + \frac{1}{4}\sum_{j\neq i} a^j$ is an integer point of $C$ and is in $\cal{Z}^<(A)$, for $i\in\{1,\ldots,n\}$.
	\item[(O)] This case is similar to the previous one, except that the situations $n \equiv 0 \pmod{4}$ and $n \equiv 2 \pmod{4}$ are reversed.~\qed						
\end{itemize}
\end{proof}

According to Claim~\ref{claim:3.5}, there are two cases, namely cases~\ref{equation:hb_thick_te-interlaces_cas1_l} and~\ref{equation:hb_thick_te-interlaces_cas2_l}.
In each case, we will apply Lemma~\ref{lemma:finding_hilbert_basis} to determine the Hilbert basis. 
By Theorem~\ref{theorem:nbr_points_zono}, the half-open zonotope spanned by $A$ contains $2^n$ integer points.

\begin{itemize}
\item[\ref{equation:hb_thick_te-interlaces_cas1_l}]
Let $S=\left\{\frac{1}{2} (a^i + a^j)\right\}_{1\leq i\leq j\leq n}\cup\left\{\frac{1}{4}\sum_{j} a^j\right\}$.
\begin{claimnumb}\label{claim:4}
	The number of distinct integer points of $\cal Z^<(A)$ which are nonnegative integer combinations of $S$ is $2^n$.
\end{claimnumb}
\begin{proof}
	Let $X$ denote the set of integer points of $\cal Z^<(A)$ obtained by nonnegative integer combinations of points in $S$.
	Let $h=\frac{1}{4}\sum_{j} a^j$. A point $x\in X$ is uniquely expressed as
	$$x = \sum_{k=1}^n \lambda_k a^k.$$
	Since $x$ is in $\cal Z^<(A)$, we have $\lambda_k\in\{0,\frac{1}{4},\frac{1}{2},\frac{3}{4}\}$, for $k\in\{1,\ldots,n\}$.
	The points in $X$ with all $\lambda_k$'s in $\{0,\frac{1}{2}\}$ do not imply $h$ in the combination, hence their number is $2^{n-1}$ as in the proof of Claim~\ref{claim:2}.
	
	We need to compute the number of points of $X$ having coefficients also in $\{\frac{1}{4},\frac{3}{4}\}$.
	Since the $a^k$'s are $\pm1$, such points are expressed as $h+\tilde{x}$, where $\tilde{x}=\sum_{k=1}^n \mu_k a^k$ is a point in $X$ with all $\mu_k$'s in $\{0,\frac{1}{2}\}$, an even number of them being positive.
	We find all such $\tilde{x}$ by setting an even number of $\mu_k$'s to $\frac{1}{2}$ and the other ones to~$0$.
	
	Thus, such points $x$ are obtained by setting an even number of $\lambda_k$'s to $\frac{3}{4}$ and the others to $\frac14$, and, by Lemma~\ref{lemma:odd_and_even_sum}, their number is
	$$\sum_{i\text{ even}} \binom{n}{i}= 2^{n-1}.$$
	Finally, we have $|X| = 2^{n-1} + 2^{n-1} = 2^n$.	
	\end{proof}
\item[\ref{equation:hb_thick_te-interlaces_cas2_l}]
Let $S=\left\{\frac{1}{2} (a^i + a^j)\right\}_{1\leq i\leq j\leq n}\cup\left\{\frac{3}{4}a^i+\frac{1}{4}\sum_{j\neq i} a^j\right\}_{1\leq i\leq n}$.
\begin{claimnumb}\label{claim:5}
	The number of integer points of $\cal Z^<(A)$ which are nonnegative integer combinations of $S$ is $ 2^n$.
\end{claimnumb}
\begin{proof}
	Let us denote this set of integer points by $X$ and let $h^i=\frac{3}{4}a^i+\frac{1}{4}\sum_{j\neq i} a^j$, for $i\in\{1,\ldots,n\}$. A point $x\in X$ is uniquely expressed as
	$$x = \sum_{k=1}^n \lambda_k a^k.$$
	Since $x$ is in $\cal Z^<(A)$, we have $\lambda_k\in\{0,\frac{1}{4},\frac{1}{2},\frac{3}{4}\}$, for $k\in\{1,\ldots,n\}$.
	Since we are considering points in the zonotope, the points in $X$ with all $\lambda_k$'s in $\{0,\frac{1}{2}\}$ imply no $h^i$'s in the combination, hence their number is $2^{n-1}$ as in the proof of Claim~\ref{claim:2}.
	
	We need to compute the number of points $x$ of $X$ having $\lambda_k$'s also in $\{\frac{1}{4},\frac{3}{4}\}$, whose decomposition involves at least one $h^i$.
	Since $x\in\cal Z^<(A)$, its decomposition involves at most one $h^i$, hence $h^i+\tilde{x}$, where $\tilde{x}=\sum_{k=1}^n \mu_k a^k$ is a point in $X$ with all $\mu_k$'s in $\{0,\frac{1}{2}\}$, an even number of them being positive, and $\mu_i=0$. Thus, every such $x$ is obtained by setting to~$\frac34$ an odd number of $\lambda_k$'s, one coming from $h^i$ and the others from the nonzero $\mu_k$'s. The other $\lambda_k$'s are $\frac{1}{4}$.
	Hence, by Lemma~\ref{lemma:odd_and_even_sum}, their number~is
	$$\sum_{s\text{ odd}}^n \binom{n}{s}= 2^{n-1}.$$ 
	Finally, this implies $|X| = 2^{n}$.	
	\end{proof}
	\end{itemize}
	In both cases~\ref{equation:hb_thick_te-interlaces_cas1_l} and~\ref{equation:hb_thick_te-interlaces_cas2_l}, note that no point of $S$ is an integer combination of the others.
	The conclusion then comes from Lemma~\ref{lemma:finding_hilbert_basis} together with either Claim~\ref{claim:4} or Claim~\ref{claim:5}.
	\end{proof}

\subsection{Proofs of the results of Section~\ref{section:RUHT}{: triangulations}}{We first prove that the join of the triangulations of each individual cone generated by a te-brick yields a regular unimodular Hilbert triangulation.}
This section ends with a proof of Corollary~\ref{corollary:0/1_simplicial_box-integer_cones}.

Recall that the join of two triangulations $\cal T_1$ and $\cal T_2$ of two cones generated by two disjoint sets of vectors whose union is linearly independent is the triangulation $\cal T_1 * \cal T_2 = \{C_1+C_2\colon C_1\in \cal T_1,C_2\in \cal T_2\}$.
\begin{lemma}[Lemma~\ref{lemma:join_is_UHT}]\label{lemma:joint_is_UHT_full}
The join of the regular unimodular Hilbert triangulations of cones generated by disjoint mutually-tu te-bricks whose union is linearly independent is a regular unimodular Hilbert triangulation of their Minkowski sum.
\end{lemma} 
\begin{proof}
By Remark~\ref{remark:mutually-tu}, the te-bricks are lattice orthogonal.
Hence, by Lemma~\ref{lemma:joint_is_UHT_full}, the join of the triangulations of each individual cone generated by disjoint mututally-tu te-bricks is a Hilbert triangulation.
Let $C = \cone(A_1\sqcup \cdots\sqcup A_k) = \cone(A_1)+\cdots+\cone(A_k) = C_1+\cdots+C_k$ be the Minkowski sum of simplicial cones generated by mutually-tu te-bricks $A_1,\ldots,A_k$, with $A=A_1\sqcup\cdots \sqcup A_k$ linearly independent. 
Let $C_1'+\cdots+C_k'$ be a cone in the join, where $C_i'=\cone(B_i)$ is a unimodular cone in the triangulation of $C_i$ and $B_i\subseteq \cal H(C_i)$, for $i=1,\dots,k$.
Since $A_i$ and $B_i$ are sets of linearly independent vectors and span the same subspace, we have $B_i=  Q_i A_i$, for some square invertible matrices $Q_i$.
In particular, we have $\gcddet(B_i)=\eqdet(B_i)=|\det(C_i')|=1$.
Consequently, $B=QA$, where $B=B_1\sqcup\cdots\sqcup B_k$ and $Q$ is the block diagonal matrix whose blocks are the $Q_i$'s.
Combining the previous remarks and the fact that $A_1,\ldots,A_k$ are lattice orthogonal, we compute the determinant of the simplicial cone $C_1'+\cdots+C_k'$, whose generators are $B_1\sqcup \dots\sqcup B_k$:
\begin{align*}
\det(C_1'+\cdots+C_k')
&= \pm\gcddet(B_1\sqcup \dots\sqcup B_k)\\
&= \pm\left(\prod_i \det(Q_i)\right)\gcddet(A_1\sqcup \dots \sqcup A_k)\\
&=\pm \left(\prod_i \det(Q_i)\right)\left(\prod_i\gcddet(A_i)\right)\\
&=\pm \prod_i \det(Q_i)\eqdet(A_i)\\
&=\pm \prod_i \eqdet(B_i)\\
&=\pm 1
\end{align*}
Therefore, the join is also a unimodular triangulation.
The Minkowski sum of two simplicial cones generated by two sets of vectors whose union is linearly independent is a simplicial cone as it combinatorially corresponds to the join of two affinely independent simplices, which is again a simplex. Moreover, the join preserves regularity by~\cite[Section~2.3.2]{Haase_Paffenholz_Piechnik_Santos_2021}.
\end{proof}

Thanks to Theorem~\ref{theorem:hb_te-bricks_full} and Lemma~\ref{lemma:joint_is_UHT_full} all that remains to find is a regular unimodular Hilbert triangulation of the cones generated by each type of te-brick, to finally obtain the following.
\begin{theorem}[Theorem~\ref{theorem:triangulation_te-sets}]\label{theorem:triangulation_te-sets_l}
Let $A$ be a te-set without thick te-interlace of size greater than six.
Then, $\cone (A)$ has a regular unimodular Hilbert triangulation.
\end{theorem}
\begin{proof}
Let $A=\{a^1,\dots,a^n\}$ be a te-brick and $C=\cone(A)$.
\subsubsection*{1. $A$ is a tu-set.}
The regular unimodular Hilbert triangulation is the cone itself.

\subsubsection*{2. $A$ is a te-lace.}
The stellar triangulation at $h=\frac{1}{2}\sum_j a^j$, namely the one formed by the $n$ cones generated by $h$ and $n-1$ generators among $n$, is regular since its coincides with the strong pulling at $h$ which preserved regularity~\cite[Lemma~2.1]{Haase_Paffenholz_Piechnik_Santos_2021}.
For $j\in\{1,\dots,n\}$, we have $\gcddet(\{h\}\cup \{a^i: i\neq j\})=\frac{1}{2}\eqdet(A)=1$, which yields the unimodularity.
Finally, all the cones are generated by Hilbert basis elements.

\subsubsection*{3. $A$ is a thin te-interlace.}
Inspired by the regular triangulation in~\cite{DeLoera_Sturmfels_Thomas_1995}, we start this case with some definitions.
Let $G=(V,E(G))$ be a graph.
An edge of~$G$ whose extremities coincide is called a \emph{loop}.
Let $\{e^v\}_{v\in V}$ be the canonical basis of $\R^V$.
If $ij\in E(G)$ is an edge of $G$, its \emph{characteristic vector} is $\chi^{ij}=e^i+e^j$.
The characteristic vector of a loop $ii$ is $\chi^{ii}=e^i+e^i = 2e^i$.
The \emph{incidence matrix} $A_G\in\{0,1\}^{|E(G)|\times |V|}$ is the matrix whose rows are the characteristic vectors of the edges of $G$.
A \emph{spanning} subgraph of $G$ is a connected graph $H=(V,F)$ with $F\subseteq E(G)$.

Let $\mathring{K}_n$ denote a complete graph with $n$ vertices to which we added a loop $ii$ at each vertex $i$.
Embed $\mathring{K}_n$ as a convex $n$-gon in $\R^2$, with  with clockwise labeled vertices $v_1,\dots,v_n$, edges $ij$ embedded as line segments $[v_i,v_j]$, for each  $i\neq j$, and loops $ii$ as circles outside the $n$-gon intersecting the $n$-gon only at $v_i$.
The edges of $\mathring{K}_n$ encode the Hilbert basis elements of $C$ as follows: an edge $ij$ represents $\frac{1}{2}(a^i+a^j)$. The latter is $a^i$ for a loop~$ii$.
We say that two distinct edges \emph{intersect} if the associated curves intersect.		
This happens either if they have a common extremity, or if the edges are $ik$ and $jl$ with $i<j<k<l$.
A loop \emph{intersects} an edge if they share a vertex.
A \emph{stellar cycle} of this embedding $\mathring{K}_n$ is a spanning subgraph with $n$ pairwise intersecting edges or loops.
Let $\cal S_n$ denote the set of stellar cycles of $\mathring{K}_n$.
As a loop is considered as a cycle of length~$1$, note that a stellar cycle contains a unique cycle which is odd.
Consequently, for each loop, there is precisely one stellar cycle whose unique cycle is this loop.
\begin{remark}\label{remark:edges thrackle} We mention that stellar cycles are in one-to-one correspondence with odd sets of vertices of $\mathring{K}_n$.
Indeed, given an odd number of vertices of $\mathring{K}_n$, there is a unique cycle on these vertices whose edges pairwise intersect. 
Then, all the remaining edges are uniquely determined since they have to intersect all the edges of the cycle. 
Therefore, up to permutation of rows or columns, the incidence matrix $A_S$ of a stellar cycle $S$ is a square matrix of size $n$ of the form:
$$\begin{bmatrix}
	A_D & \mathbf{0}\\
	* & I_{n-|D|}\\
\end{bmatrix},$$
for $D$ the odd cycle of $S$ ($A_D=\begin{bmatrix}2\end{bmatrix}$ if the cycle is a loop). Since $D$ is an odd cycle, its incidence matrix $A_D$ has determinant $\pm 2$, therefore $\det (A_S) = \pm 2$.
We refer the reader to~\cite{DeLoera_Sturmfels_Thomas_1995} for more details.
\end{remark}
\begin{claim}
The collection $\cal T$ of the cones $C_S=\cone(\frac{1}{2}(a^i+a^j)\colon ij\in E(S))$, for all $S\in \cal S_n$, forms a regular unimodular Hilbert triangulation of $C$.
\end{claim}
\begin{proof}
By Theorem~\ref{theorem:hb_te-bricks_full}, all the cones in $\cal T$ are generated by Hilbert basis elements.
Note that a triangulation of $P=\conv (A)$ yields a triangulation of $C$.
Let $\Delta_{2,A}=\conv\left(\frac{1}{2}(a^i+a^j)\colon i\neq j\right)$ and $T_i=\conv \left(a^i,\frac{1}{2}(a^i+a^j)\colon j\neq i\right)$, for $i\in\{1,\ldots,n\}$.
All these sets have disjoint relative interiors and we have $P=\Delta_{2,A}\cup\bigcup_{i=1}^n T_i$.
The simplex $T_i$ shares only the facet $F_i=\conv(\frac{1}{2}(a^i+a^j)\colon j\neq i)$ with $\Delta_{2,A}$, for $i\in\{1,\ldots,n\}$.
Therefore, from any triangulation of $\Delta_{2,A}$, we obtain one for $P$ by attaching the simplex $T_i$ at each facet $F_i$.

By~\cite[{Theorem~2.3 and Lemma~2.4}]{DeLoera_Sturmfels_Thomas_1995} the set of simplices
$$\left\{\conv \left(\frac{1}{2}(e^i+e^j)\colon ij\in E(S)\right)\!\!, \text{ for all loopless } S\in\cal S_n\right\}$$
forms a regular triangulation of $\Delta_2=\conv\left(\frac{1}{2}(e^i+e^j)\colon i\neq j\right)$.
Therefore, the simplices $\conv \left(\frac{1}{2}(a^i+a^j):ij\in E(S)\right)$~for all loopless $S\in\cal S_n$, form a regular triangulation of $\Delta_{2,A}$, as $\Delta_{2,A} = A^\top\Delta_{2}$ is the image of $\Delta_2$ by the linear map $A^\top$.

We now attach the $T_i$'s. The triangulation we obtain remains regular by assigning a sufficiently large weight to every $a^i$, which is separated from the other $a^j$'s by the facet $F_i$, and those facets have maximal weight among the other ones of the triangulation.
Since the cones associated with these simplices are the cones $C_S$ defined above for all loopless $S$, the collection of Hilbert cones $\cal T$ is indeed a regular triangulation of $C$.
\medskip

All that remains to show is that each cone of $\cal T$ is unimodular.
Recall that $A$ is thin and $\eqdet (A) =\pm 2^{n-1}$.
For $i\in\{1,\ldots,n\}$, the cone $C_i$ associated with the stellar cycle with loop $ii$ is generated by $a_i$ and $\frac{1}{2}(a^i+a^j)$ for $j\neq i$.
It is unimodular since:
$$\det (C_i) = \pm\eqdet \left(\frac{1}{2}\begin{bmatrix}
	\mathbf{I}_{i-1} & \mathbf{1} & \mathbf{0}\\
	\mathbf{0}^\top & 2 & \mathbf{0}^\top\\
	\mathbf{0} & \mathbf{1} & \mathbf{I}_{n-i}
\end{bmatrix}A \right) = \pm  \frac{1}{2^n} 2\eqdet (A)= \pm1.$$
\medskip
For a stellar cycle $S\in\cal S_n$, by Remark~\ref{remark:edges thrackle}, we have:
$$\det (C_S) = \pm\eqdet \left(\frac{1}{2} A_S^\top A \right) = \pm \frac{1}{2^n} \det (A_S) \eqdet (A) = \pm1,$$
as desired.
\end{proof}
\subsubsection*{4. $A$ is a thick te-interlace of size at most $6$.}
There are four cases: either $n$ equals $4$ or $6$, and either \ref{equation:hb_thick_te-interlaces_cas1} or~\ref{equation:hb_thick_te-interlaces_cas2} occurs in Theorem~\ref{theorem:hb_te-bricks_full}.
In each case, we used Polymake~\cite{Gawrilow_Joswig_2000} to generate a regular Hilbert triangulation, inspired by~\cite{DeLoera_Sturmfels_Thomas_1995}, and a simple determinant computation algorithm to check unimodularity.
In Section~\ref{section:figures}, the reader will find figures, generated from the output of our Polymake script, representing the cones of the triangulations for each of these cases.
Each figure corresponds to a set of $n$ Hilbert basis elements generating a unimodular Hilbert cone in one of the four regular unimodular Hilbert triangulations.	
More precisely, in each figure there are $n$ vertices labelled from $1$ to $n$ and colored loops and edges, which correspond to Hilbert basis elements as follows.

\medskip

In all figures:
\begin{itemize}
\item[\textbullet] $a^i$ is represented by a blue loop at vertex $i$, for $i=1,\ldots,n$,
\item[\textbullet] $\frac{1}{2} (a^i + a^j)$ is represented by a blue edge $ij$, for $i\neq j$.
\end{itemize}

In Case~\ref{equation:hb_thick_te-interlaces_cas1}, the additional Hilbert basis element is:
\begin{itemize}
\item[\textbullet] $\frac{1}{4}\sum_{j} a^j$, which is represented by a green dot in the center of the figure.
\end{itemize} 

In Case~\ref{equation:hb_thick_te-interlaces_cas2}, there are $n$ additional Hilbert basis elements, which are:
\begin{itemize}
\item[\textbullet] $\frac{3}{4}a^i+\frac{1}{4}\sum_{j\neq i} a^j$, which are each represented by a red circle around vertex $i$, for $i=1,\ldots,n$.
\end{itemize}
All the associated figures are in Section~\ref{section:figures}.
\end{proof}

Finally, we prove Corollary~\ref{corollary:0/1_simplicial_box-integer_cones}.
\begin{corollary}[Corollary~\ref{corollary:0/1_simplicial_box-integer_cones}]
Simplicial box-totally dual integral cones in the nonnegative orthant have the integer Carath\'eodory property.
\end{corollary}
\begin{proof}
As mentioned in the proof of Theorem~\ref{theorem:invert_TE}, a simplicial cone is box-TDI if and only if it is generated by a te-set.
If the cone is in the nonnegative orthant, then the te-set is $0,\!1$.
Since $0,\!1$ te-sets contain no te-interlaces, Theorem~\ref{theorem:triangulation_te-sets_l} provides a regular unimodular Hilbert triangulation of the cone, hence it has the integer Carath\'eodory property.
\end{proof}
\newpage
\subsection{Figures for the case~\ref{equation:hb_thick_te-interlace} of Theorem~\ref{theorem:hb_te-bricks}}\label{section:figures}~

\begin{center}
$n=4$, Case~\ref{equation:hb_thick_te-interlaces_cas1}:
\smallskip
\hrule
\medskip
\includegraphics[width=\linewidth]{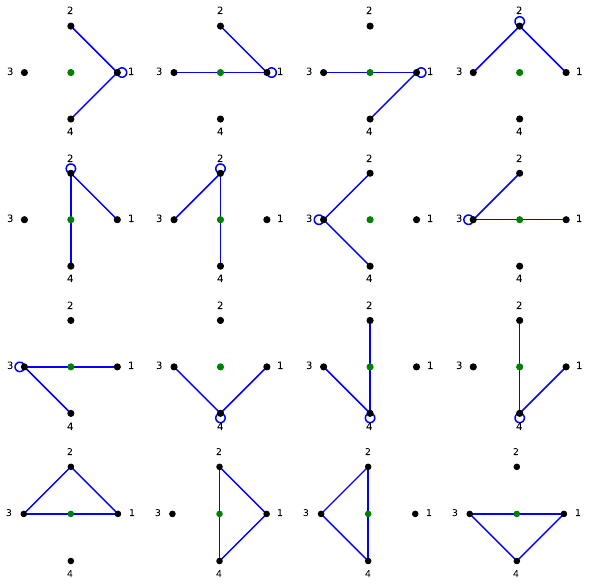}
\newpage
$n=6$, Case~\ref{equation:hb_thick_te-interlaces_cas1}:
\smallskip
\hrule	
\medskip

\includegraphics[width=0.9\linewidth]{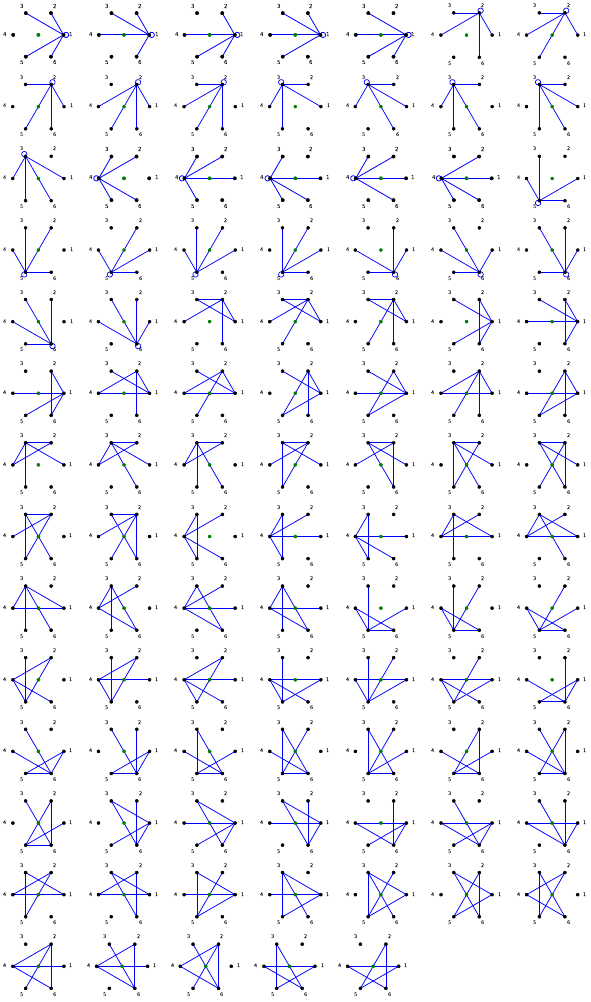}		
\newpage
$n=4$, Case~\ref{equation:hb_thick_te-interlaces_cas2}:
\smallskip
\hrule
\medskip

\includegraphics[width=\linewidth]{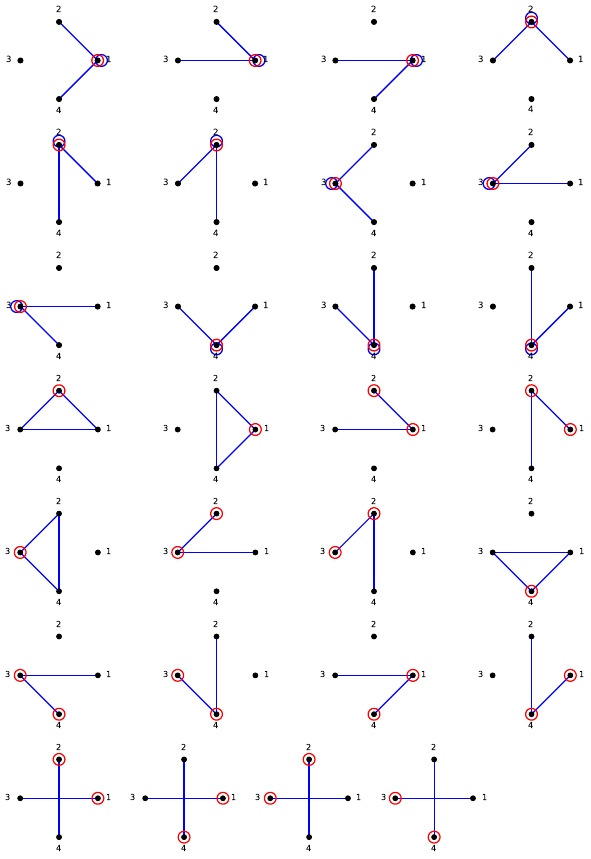}
\newpage
$n=6$, Case~\ref{equation:hb_thick_te-interlaces_cas2}:
\smallskip
\hrule
\medskip

\includegraphics[width=.9\linewidth]{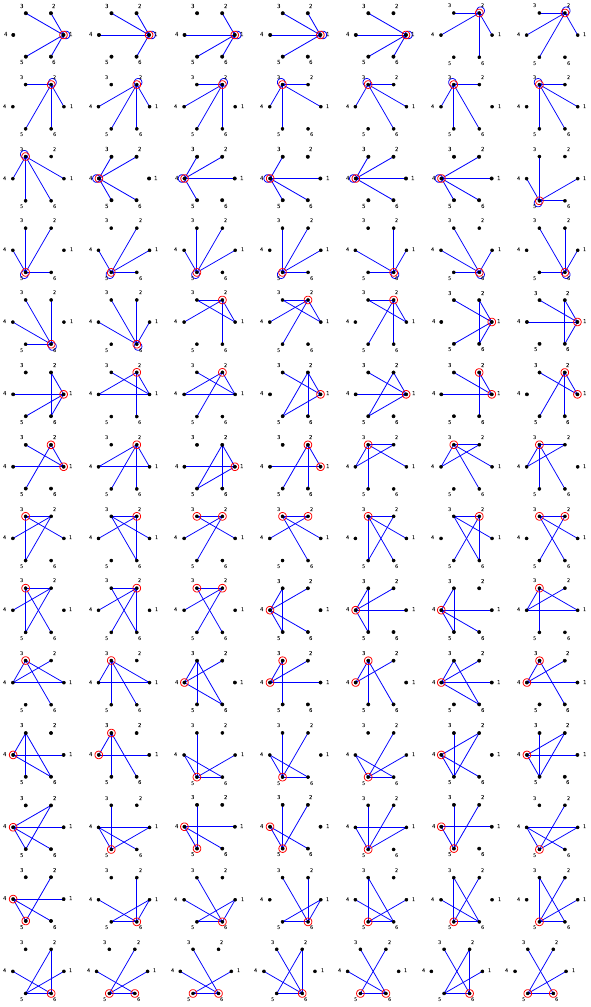}
\newpage
\includegraphics[width=.9\linewidth]{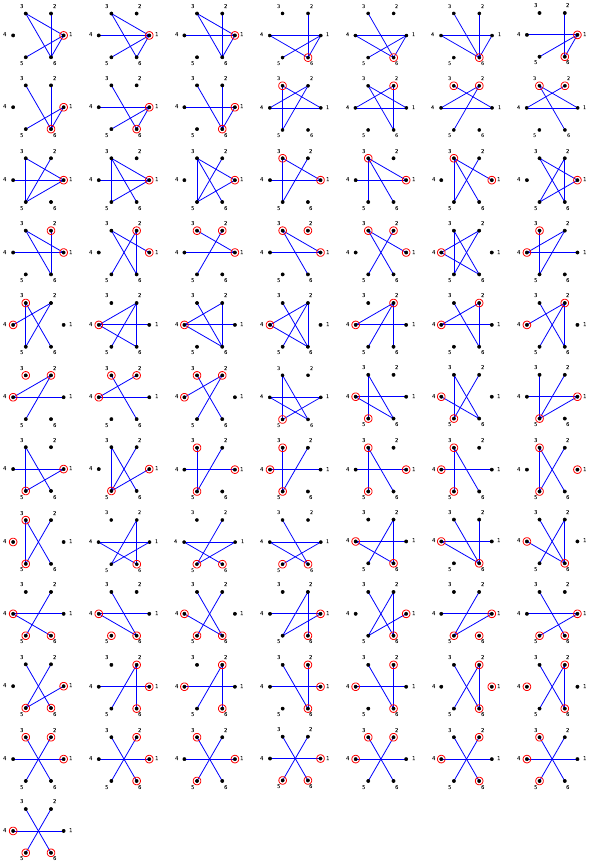}
\end{center}

\section*{Declarations}
The authors have no conflicts of interest to declare that are relevant to the content of this article.

\newpage


\end{document}